\documentclass[a4paper,11pt]{amsart}

\usepackage{hyperref}
\hypersetup{colorlinks=true,allcolors=blue}
\usepackage{hypcap}

\usepackage[leqno]{amsmath}
\usepackage{amsfonts,amssymb,amsthm}
\usepackage{tensor}
\usepackage{mathrsfs}
\usepackage{esint}
\usepackage{mathtools}
\usepackage{graphicx} 
\usepackage{galois}
\usepackage[utf8]{inputenc}
\usepackage[inline]{enumitem}
\usepackage[top=1in,bottom=1in,left=1in,right=1in]{geometry}

\usepackage{multicol}

\theoremstyle{plain}
\newtheorem{theorem}{\protect\theoremname}[section]
\theoremstyle{plain}
\newtheorem{proposition}[theorem]{\protect\propositionname}
\theoremstyle{definition}
\newtheorem{definition}[theorem]{\protect\definitionname}
\theoremstyle{plain}
\newtheorem{corollary}[theorem]{\protect\corollaryname}
\theoremstyle{remark}
\newtheorem{remark}[theorem]{\protect\remarkname}
\theoremstyle{plain}
\newtheorem{lemma}[theorem]{\protect\lemmaname}
\theoremstyle{plain}

\providecommand{\corollaryname}{Corollary}
\providecommand{\definitionname}{Definition}
\providecommand{\lemmaname}{Lemma}
\providecommand{\propositionname}{Proposition}
\providecommand{\remarkname}{Remark}
\providecommand{\theoremname}{Theorem}
\providecommand{\conjecturename}{Conjecture}

\numberwithin{equation}{section}
\numberwithin{figure}{section}
\numberwithin{table}{section}

\newcommand{\R}{\mathbb{R}}
\newcommand{\N}{\mathbb{N}}
\newcommand{\Z}{\mathbb{Z}}

\newcommand{\Sph}{\mathbb{S}}
\newcommand{\Rd}{\R^d}
\newcommand{\bydef}{\mathrel{\mathop:}=}
\newcommand{\loc}{\mathrm{loc}}

\newcommand{\jp}[1]{\langle #1 \rangle}

\newcommand{\norm}{\mathbf{n}}
\newcommand{\supp}{\mathrm{supp\,}}
\newcommand{\sing}{\mathrm{sing\,}}

\newcommand{\ul}[1]{\underline{#1}}
\newcommand{\upsi}{\underline{\psi}}

\renewcommand{\O}{\mathcal{O}}

\newcommand{\WF}{\mathrm{WF}}
\newcommand{\HWF}{\mathrm{HWF}}

\newcommand{\Geo}{\mathcal{G}}

\newcommand{\Cinf}{C^\infty}
\newcommand{\M}[3]{M^{#1}_{#2}(#3)}

\renewcommand{\H}[1]{H^{#1}}
\newcommand{\HC}[1]{\mathcal{H}^{#1}}
\newcommand{\Linf}{L^\infty}
\newcommand{\Ccinf}{C_c^\infty}
\newcommand{\Lone}{L^1}
\newcommand{\Ltwo}{L^2}
\newcommand{\Holder}[1]{W^{#1,\infty}}

\newcommand{\Dscr}{\mathscr{D}}
\newcommand{\swtz}{\mathscr{S}}

\newcommand{\OP}[2]{\mathscr{O}^{#1}_{#2}}

\newcommand{\drv}{\mathrm{d}}
\renewcommand{\d}{\,\drv}
\newcommand{\dx}{\d x}
\newcommand{\dy}{\d y}
\newcommand{\dz}{\d z}
\newcommand{\dt}{\d t}
\newcommand{\ds}{\d s}
\newcommand{\dxi}{\d \xi}

\newcommand{\pt}{\partial_t}
\newcommand{\ps}{\partial_s}
\newcommand{\pn}{\partial_\norm}
\newcommand{\px}{\partial_x}

\newcommand{\pz}{\partial_z}
\newcommand{\pxi}{\partial_\xi}

\newcommand{\T}[1]{T_{#1}}
\renewcommand{\P}[1]{\mathcal{P}_{#1}}
\newcommand{\Fourier}{\mathcal{F}}
\newcommand{\Id}{\mathrm{Id}}
\newcommand{\op}{\mathrm{Op}}

\renewcommand{\L}{\mathcal{L}}
\newcommand{\Lag}{\mathscr{L}}

\newcommand{\Partition}{\mathscr{P}}

\renewcommand{\Re}{\mathrm{Re}}
\renewcommand{\Im}{\mathrm{Im}}
\renewcommand{\div}{\mathrm{div}}
\newcommand{\A}{\mathcal{A}}
\newcommand{\B}{\mathcal{B}}

\begin{document}

\title[Propagation of singularities for gravity-capillary water waves]{Propagation of singularities for \\ gravity-capillary water waves}
\author[H. Zhu]{HUI ZHU}
\address{\parbox[t]{\linewidth}{Laboratoire de Mathématiques d’Orsay, Univ.~Paris-Sud, CNRS\\ Université Paris-Saclay, 91405 Orsay, France}}
\curraddr{\parbox[t]{\linewidth}{University of Michigan Department of Mathematics\\ 1830 East Hall, 530 Church St., Ann Arbor, MI 48109}}
\email{zhuhui@umich.edu}
\thanks{The author is partially supported by the grant ``ANAÉ'' ANR-13-BS01-0010-03 of the Agence Nationale de la Recherche, and the Allocation Doctorale of the École Normale Supérieure.}

\begin{abstract}
We obtain two results of propagation for the gravity-capillary water wave system.
The first result shows the propagation of oscillations and the spatial decay at infinity; the second result shows a microlocal smoothing effect under the non-trapping condition of the initial free surface.
These results extend the works of Craig, Kappeler and Strauss~\cite{CKS95microlocal},  Wunsch~\cite{Wunsch99propagation} and Nakamura~\cite{Nakamura05propagation} to quasilinear dispersive equations. 
These propagation results are stated for water waves with asymptotically flat free surfaces, of which we also obtain the existence.
To prove these results, we generalize the paradifferential calculus of Bony~\cite{Bony79calcul} to weighted Sobolev spaces and develop a semiclassical paradifferential calculus.
We also introduce the quasi-homogeneous wavefront sets which characterize, in a general manner, the oscillations and the spatial growth/decay of distributions.
\end{abstract}

\maketitle

\section{Introduction}

In this paper, we present two results on the propagation of singularities for the gravity-capillary water wave system, including a microlocal smoothing effect. To the best of our knowledge, these results are the first of this type for quasilinear dispersive equations.
Before stating the main results, we shall first revisit classical results of propagation for the linear half wave equation and the linear Schr\"odinger equation.
They lead us to a more generalized concept of singularities which is adaptive to various dispersive equations.

\subsection{Wavefront set and the linear half wave equation}

If~$ u \in \Dscr'(M) $ where~$ M $ is a smooth manifold without boundary, then the singular support of~$ u $, denoted by $ \sing\supp u $, is the smallest closed subset of~$ M $ outside of which~$ u $ is smooth. 
To study the propagation of singularities when~$ u $ solves some partial differential equations, the information given by $ \sing\supp u $ is usually insufficient.
Heuristically, if we consider singularities as accumulations of wavepackets with large wavenumbers, then this is because the propagation direction of a wavepacket is given by its wavenumber rather than its location.
It is probably with this mindset that H\"ormander introduced in~\cite{Hormander71fourier} the concept of the wavefront set.

The wavefront set of~$ u $, denoted by $ \WF(u) $, lifts $ \sing\supp u $ to the cotangent bundle $ T^*M\backslash 0 $ in the sense that a point $ x_0 \in M $ belongs to $ \sing\supp u $ if and only if there exists $ \xi_0 \ne 0 $ such that $ (x_0,\xi_0) \in \WF(u) $.
We shall recall an equivalent definition of $ \WF(u) $ essentially due to Guillemin and Sternberg~\cite{GS77geometric}:
in local coordinates, a point $ (x_0,\xi_0) \in T^*M \backslash 0 $ does not belong to $ \WF(u) $ if and only if there exists $ a \in \Ccinf(\R^{2d}) $ with $ a(x_0,\xi_0) \ne 0 $ such that
$ \| a(x,hD_x) u \|_{L^2} = \O(h^\infty) $
for $ h \in (0,1] $.
For the definition of the pseudodifferential operator $ a(x,hD_x) $, see~\eqref{eq:def:quantization}.

In terms of the wavefront set, H\"ormander proved in~\cite{Hormander71fourier} a propagation result for pseudodifferential equations of real principal type, improving previous works on wave propagation by Courant and Lax~\cite{CL56propagation} and Lax~\cite{Lax57asymptotic}.
\begin{theorem}[H\"ormander \cite{Hormander71fourier}]
\label{thm:hormander}
Let $ M $ be a smooth manifold without boundary. 
Let $ P \in \Psi^1(M) $ which admits a real principal symbol $ \sigma(P) = \sigma(P)(x,\xi) \in C^\infty(T^*M \backslash 0,\R)$ and let $ \Phi = \Phi_t(x,\xi) \in C^\infty(\R\times T^*M \backslash 0,T^*M \backslash 0) $ be the Hamiltonian flow of~$ \sigma(P) $.
If $ u $ solves the Cauchy problem
\begin{equation}
\label{eq:equation:hormander}
\begin{cases}
\pt u + i P u = 0, \\
u(0) = u_0 \in L^2(M),
\end{cases}
\end{equation}
then for all $ (x_0,\xi_0) \in \WF(u_0) $ and all $ t \in \R $, we have 
$ \Phi_t(x_0,\xi_0) \in \WF(u(t)). $
\end{theorem}
In particular, if $ P = \sqrt{-\Delta_g} $ where $ g $ is a Riemannian metric on $ M $, then~\eqref{eq:equation:hormander} becomes the half wave equation and $ \Phi $ is the corresponding cogeodesic flow on $ T^*M $.
Therefore, we conclude that, for solutions to the half wave equation, microlocal singularities travel at speed one along cogeodesics.
This gives a justification for the Huygens--Fresnel principal of wavefront propagation. 

For the propagation of singularities for the semilinear wave equation, we refer to Bony~\cite{Bony86singularities} and Lebeau~\cite{Lebeau89semilinearII}. 
For the propagation and the reflection of singularities for the linear wave equation on manifolds with corners, see Vasy~\cite{Vasy08corner} and Melrose, Vasy and Wunsch~\cite{MVW13corner}.

\subsection{Homogeneous wavefront set and the linear Schr\"odinger equation}

\label{sec::intro-schrodinger}

H\"ormander's theorem (Theorem~\ref{thm:hormander}) is untrue when the order of~$ P $ is higher than one. 
For example, the Schr\"odinger propagator $ e^{it\Delta/2} $ on $ \Rd $ sends $ \mathscr{E}'(\R^d) $ to $ C^\infty(\R^d) $ whenever $ t \ne 0 $.
We conclude that singularities may appear and disappear along the Schr\"odinger flow. 
These phenomena of ``microlocal smoothing effect'' and ``microlocal singularity formation'' are due to the infinite speed of propagation of the Schr\"odinger equation, as wavepackets with large wavenumbers can travel to or back from infinity instantaneously.

The study of the infinite speed of propagation of the Schr\"odinger equation probably dates back to Boutet-de-Monvel~\cite{BdM75propagation} and Lascar~\cite{Lascar77propagation,Lascar78propagation}.
They proved that space-time singularities, as elements of some space-time wavefront sets, travel along geodesics at an infinite speed. 
They did not obtain, however, a time-dependent propagation results for wavefront sets with respect to the space variable alone. 
The study of the smoothing effect for dispersive equations with an infinite speed of propagation was initiated by Kato in~\cite{Kato1983KdV} where he proved a local smoothing effect for generalized KdV equations. 
In~\cite{CKS95microlocal}, Craig, Kappeler and Strauss proved microlocal smoothing effects for the linear Schr\"odinger equation under the non-trapping condition of the geometry. 
Their results were later refined by Wunsch who obtained in~\cite{Wunsch99propagation} a time-dependent propagation after understanding the transformation between singularities and quadratic oscillations at infinity. 
The simplest example is the following identity:
\begin{equation*}
e^{it\Delta/2}\delta_{x_0}(x) = \frac{1}{(2\pi i t)^{d/2}} e^{i|x-x_0|^2/2t},
\end{equation*}
where $ \delta_{x_0} $ is the Dirac measure at~$ x_0 \in \Rd $.
Wunsch's results were stated on Riemannian manifolds endowed with a scattering metric. 
He introduced the quadratic scattering wavefront set to characterize quadratic oscillations. 

Similar results were later obtained, independently, by Nakamura in~\cite{Nakamura05propagation} via a simpler calculus but in a less general geometric setting --- asymptotically Euclidean geometries, where he introduced the homogeneous wavefront set.
By definition, if $ u \in \swtz'(\R^d)$, then the homogeneous wavefront set $ \HWF(u) $ is a subset of $ \R^{2d} $ whose complement consists of all $ (x_0,\xi_0) $ admitting a symbol $ a \in \Ccinf(\R^{2d}) $ with $ a(x_0,\xi_0) \ne 0 $ such that $ \| a(hx,hD_x) u \|_{L^2} = \O(h^\infty) $ for $ h \in (0,1] $.
It was proven by Ito in~\cite{Ito06propagation} that the quadratic scattering wavefront set and the homogeneous wavefront set are essentially equivalent in asymptotically Euclidean geometries.
In fact, heuristically, if $ x_0 \ne 0 $ and $ \xi_0 \ne 0 $, then the pseudodifferential operator $ a(hx,hD_x) $ is a microlocalization in the region of quadratic oscillation: 
\begin{equation*}
|x| \sim |\xi| \sim h^{-1}.
\end{equation*}
Take for example the free Schr\"odinger equation in $ \Rd $, of which the dispersion relation is
$ \omega = \frac{1}{2}|\xi|^2 $.
A wave packet of frequency~$ \xi \sim h^{-1} $ travels at the group velocity 
$ v = \frac{\d\omega}{\d \xi} = \xi \sim h^{-1}. $
The homogeneously scaled quantization $ a \mapsto a(hx,hD_x) $ thus allows us to keep up with the infinite speed of propagation and obtain an analogue of H\"ormander's theorem.

\begin{theorem}[Nakamura \cite{Nakamura05propagation}, similar results by Wunsch \cite{Wunsch99propagation}]
\label{thm::doi-nakamura}
Let $ g $ be an asymptotically Euclidean Riemannian metric on $ \Rd $, meaning that there exists $ \epsilon > 0 $ such that for all $ \alpha \in \N^d $ and all $ i,j\in \{1,\ldots,d\} $, we have
\begin{equation}
\label{eq:cdt:g-decay}
|\px^\alpha(g_{ij}(x)-\delta_{ij})| \lesssim \jp{x}^{-|\alpha|-\epsilon}.
\end{equation}
Consider the Cauchy problem of the linear Schr\"odinger equation 
\begin{equation*}
\begin{cases}
i\pt u + \frac{1}{2} \Delta_g u = 0, \\
u(0) = u_0 \in L^2(\R^d).
\end{cases}
\end{equation*}
Then the following propagation results hold:
\begin{enumerate}
\item \label{nakamura-doi} If $ (x_0,\xi_0) \in \HWF(u_0) $ and $ t_0 \in \R $ such that $ \xi_0 \ne 0 $ and $ x_0+t\xi_0 \ne 0 $ for all~$ t $ between~$ 0 $ and~$ t_0 $, then 
$ (x_0+t_0\xi_0,\xi_0) \in \HWF(u(t_0)). $
\item \label{nakamura-self} 
If $ (x_0,\xi_0) \in \WF(u_0) $ is forwardly resp.\ backwardly non-trapping in the sense that the cogeodesic issued from $ (x_0,\xi_0) $, denoted by $ \{(x_t,\xi_t)\}_{t\in\R} $ (with an abuse of notation), satisfies
\begin{equation*}
\lim_{t\to+\infty} |x_t| = +\infty
\quad \text{resp.} \quad
\lim_{t\to-\infty} |x_t| = +\infty,
\end{equation*}
then there exists $ \xi_+ \in \Rd $ resp.\ $ \xi_- \in \Rd $ satisfying $ \xi_\pm = \lim_{t\to\pm\infty} \xi_t $, and moreover, for all $ t_0 > 0 $ resp.\ $ t_0 < 0 $, we have
\begin{equation*}
(t_0\xi_+,\xi_+) \in \HWF(u(t_0))
\quad \text{resp.} \quad
(t_0\xi_-,\xi_-) \in \HWF(u(t_0)).
\end{equation*}
\end{enumerate}
\end{theorem}
Theorem~\ref{thm::doi-nakamura}\eqref{nakamura-doi} studies the propagation of oscillations and spatial growth/decay for Schr\"o\-dinger waves at infinity and we thus require the condition $ x_0 + t\xi_0 \ne 0 $.
In $ \Rd $, this result is a consequence of an Egorov-type argument and the commutation relation:
\begin{equation*}
\Big[i\pt + \frac{1}{2}\Delta,a(t,hx,hD_x)\Big] = (i\pt a - \xi \cdot \px a)(t,hx,hD_x) + \O(h^2)
\end{equation*}
where $ a \in C_b^\infty(\R\times\R^{2d})$.
A similar argument works in asymptotically Euclidean geometries where we replace the role of the semiclassical quantization $ x \mapsto hx $ with the spatial decay of the metric~$ g $, i.e., the condition~\eqref{eq:cdt:g-decay}. 

Theorem~\ref{thm::doi-nakamura}\eqref{nakamura-self} is a microlocal smoothing effect:
if $ (t_0\xi_\pm,\xi_\pm) $ does not belong to $ \HWF(u(t_0)) $, then $ (x_0,\xi_0) $ can not be a element of $ \WF(u_0) $. 
This result is a refinement of the result in~\cite{CKS95microlocal} and can be proven via a positive commutator estimate.
In $ \Rd $, this estimate has the form
\begin{equation*}
\Big[i\pt + \frac{1}{2}\Delta,a(t,x,hD_x)\Big] \gtrsim \O(h^\infty)
\end{equation*}
where~$ a $ is some well-chosen symbol.
For related results, see Doi \cite{Doi96smoothing,Doi00smoothing} and Burq \cite{Burq04smoothing} for the necessity of the non-trapping condition;
see Robbiano and Zuily \cite{RZ99microlocal} for a microlocal analytic smoothing effect;
see Kenig, Ponce and Vega \cite{KPV98smoothing:nonlinear} and Szeftel \cite{Szeftel05microlocal:smoothing} for local and microlocal smoothing effects for the semilinear Schr\"odinger equation. 
We should also remark that in~\cite{Hormander91quadratic} H\"ormander has also introduced an essentially equivalent counterpart of the homogeneous wavefront set to which a similar definition as that of Nakamura was given. 
See Rodino and Wahlberg \cite{RW14Gabor}, Schulz and Wahlberg \cite{SW17equality} for more comments.
However, Theorem~\ref{thm::doi-nakamura}\eqref{nakamura-self} is unable, via simply reversing the time, to show how oscillations at infinity form singularities along the Schr\"odinger flow. 
Indeed, the information about the locations of singularities is not contained in quadratic oscillations but rather in linear oscillations at infinity. 
See the works of Hassell and Wunsch~\cite{HW05schrodinger} and Nakamura~\cite{Nakamura09singularities} for more on this subject.

\subsection{Quasi-homogeneous wavefront set and the gravity-capillary water wave system}

\label{sec::water-wave-system}

The gravity-capillary water wave system describes the evolution of inviscid, incompressible and irrotational fluid with a free surface, in the presence of a gravitational field and the surface tension. 

\subsubsection{Formulations of the gravity-capillary water wave system}

We shall first recall the Eulerian formulation of the gravity-capillary water wave system.
The area occupied by the fluid is a time-dependent simply connected open subset of $ \R^{d+1} $ and is denoted by~$ \Omega $.
The boundary of~$ \Omega $ consists of two parts: the free surface $ \Sigma $ and the bottom $ \Gamma $.
The free surface of the fluid is a time-dependent hypersurface which is the graph of a function $ \eta = \eta(t,x) $ where $ (t,x) \in \R \times \Rd $, whereas the bottom is independent of time and is of depth $ b \in (0,\infty) $. 
Therefore,
\begin{equation*}
\Omega = \{ -b < y < \eta\}, \quad
\Sigma = \{y = \eta\}, \quad
\Gamma = \{y = -b\}.
\end{equation*}
The Eulerian formulation describes water waves in the unknowns $ (\eta,v,P) $ where $ v : \Omega \to \Rd $ is the Eulerian vector field and $ P : \Omega \to \R $ is the pressure of the fluid.
\begin{equation}
\label{eq:sys-water-wave}
\begin{cases}
\pt v + v \cdot \nabla_{xy} v = - \nabla_{xy} (P+gy), & \text{Euler equation};\\
\nabla_{xy} \cdot v = 0, & \text{incompressibility};\\
\nabla_{xy} \times v = 0, & \text{irrotationality};\\
(v \cdot \norm)_{y=\eta} =\pt \eta / \jp{\nabla\eta}, & \text{kinetic condition at the free surface};\\
(v \cdot \norm)_{y=-b} = 0, & \text{kinetic condition at the bottom};\\
-P|_{y=\eta} = \kappa H(\eta),  & \text{dynamic condition}.
\end{cases}
\end{equation}
Here $ g \in \R $ is the gravitational acceleration, $ \kappa > 0 $ is the surface tension, $ \norm : \partial\Omega \to \Sph^d $ denotes the exterior unit normal vector field of~$ \partial\Omega $, while
\begin{equation}
\label{eq::surface-tension}
H(\eta) = \nabla \cdot \Big( \frac{\nabla\eta}{\sqrt{1+|\nabla\eta|^2}}\Big)
\end{equation}
is the mean curvature of the free surface.
In~\eqref{eq:sys-water-wave}, the kinetic condition at the free surface implies that fluid particles which are initially on the free surface will stay on the free surface, whereas the kinetic condition at the bottom is a rephrasing of the impenetrability of the bottom.
The dynamic condition is the Laplace--Young equation which expresses the balance between the interior pressure~$ P $ and the surface tension $ \kappa $.

One of the main difficulties in the study of the Eulerian formulation of the system~\eqref{eq:sys-water-wave} is the time-dependence of the domain~$ \Omega $. 
By Zakharov~\cite{Zakharov68stability} and Craig and Sulem~\cite{CS93WaterWaves}, we can reformulate~\eqref{eq:sys-water-wave} as a system in~$ \Rd $. 
Note that due to the simply connected geometry of~$ \Omega $ and the irrotationality of the fluid, there exists a velocity potential $ \phi : \Omega \to \R $ such that
$ \nabla_{xy} \phi = v. $
By the incompressiblity of the fluid, the potential $ \phi $ is harmonic.
Therefore $ \phi $ satisfies the Laplace equation with Neumann boundary conditions:
\begin{equation*}
\Delta_{xy} \phi = 0, \quad 
\pn \phi|_{y=\eta} = \pt\eta / \jp{\nabla\eta}, \quad 
\pn \phi|_{y=-b} = 0.
\end{equation*}
Define $ \psi = \phi|_{y=\eta} $ and denote
\begin{equation*}
G(\eta) \psi 
= \jp{\nabla\eta} \pn \phi|_{y=\eta}.
\end{equation*}
Here $ G(\eta) $ is the Dirichlet--Neumann operator (see \S\ref{sec::D-N-op} for a rigorous definition).
Then the system~\eqref{eq:sys-water-wave} can be rewritten in terms of the unknowns $ (\eta,\psi) $:
\begin{equation}
\label{eq::equation-water-wave}
\begin{cases}
\pt \eta - G(\eta) \psi = 0, \\
\pt \psi + g\eta - \kappa H(\eta) + \dfrac{1}{2} |\nabla \psi|^2 - \dfrac{1}{2} \dfrac{( \nabla\eta \cdot \nabla \psi + G(\eta)\psi )^2}{1+|\nabla\eta|^2} = 0.
\end{cases}
\end{equation}
We shall assume henceforth that $ \kappa = 1 $ for simplicity.

\subsubsection{Quasi-homogeneous wavefront set and model equations}

It is known that the linearization of~\eqref{eq::equation-water-wave} about the stationary solution $ (\eta,\psi)=(0,0) $ can be symmetrized, up to a smoothing remainder, to the fraction Schr\"odinger equation or order~$ 3/2 $.
Consider the more general model equation
\begin{equation}
\label{Equation::model}
\pt u + i |D_x|^\gamma u = 0, \quad \gamma \ge 1.
\end{equation}
It is natural to ask ourselves if we can define a new family of wavefront sets and extend the results from Theorems~\ref{thm:hormander} and~\ref{thm::doi-nakamura} to~\eqref{Equation::model}.
Note that a wave packet of~\eqref{Equation::model} of frequency $ \xi \sim h^{-1} $ travels at the group velocity 
\begin{equation*}
v = \frac{\d |\xi|^\gamma}{\d\xi} = \gamma |\xi|^{\gamma-2} \xi \sim h^{-(\gamma-1)}.
\end{equation*}
It suggests that we need to use pseudodifferential operators of the form $  a(h^{\gamma-1}x,hD_x) $ as test operators.
In the following definition, we consider the more general quantization with two parameters.

\begin{definition}
\label{def::intro-quasi-homogeneous-WF}
If $ u \in \swtz'(\Rd) $, $ \mu \in \R \cup \{\infty\} $, $ \delta \ge 0 $ and $ \rho \ge 0 $ with $ \delta + \rho > 0 $, then the quasi-homogeneous wavefront set $ \WF^\mu_{\delta,\rho}(u) $ is a subset of $ \R^{2d} $ defined as follows.
A point $ (x_0,\xi_0) $ does not belong to $ \WF^\mu_{\delta,\rho}(u) $ if and only if there exists $ a \in \Ccinf(\R^{2d}) $ with $ a(x_0,\xi_0) \ne 0 $ such that
$ \|a(h^\delta x,h^\rho D_x) u\|_{L^2} = \O(h^\mu) $ for $ h \in (0,1] $. 
Here,
\begin{equation}
\label{eq:def:quantization}
a(h^\delta x,h^\rho D_x) u(x) = (2\pi)^{-d} \iint_{\R^{2d}} e^{i(x-y) \cdot \xi} a(h^\delta x, h^\rho \xi) u(y) \dy \dxi.
\end{equation}
\end{definition}
Note that $ \WF^\mu_{\delta,\rho}(u) $ is invariant under the scaling 
$ (x,\xi) \mapsto (\lambda^\delta x, \lambda^\rho \xi) $ for all $ \lambda > 0 $.
The existence of $ (x_0,\xi_0) \in \WF^\mu_{\delta,\rho}(u) $ implies an accumulation of mass near the ray $ \{(\lambda^\delta x_0, \lambda^\rho \xi_0)\}_{\lambda > 0} $.
By choosing different parameters, we recover the definitions of various wavefront sets from the quasi-homogeneous wavefront set: the wavefront set of H\"ormander $ (\delta,\rho,\mu) = (0,1,\infty) $, the homogeneous wavefront set of Nakamura $ (\delta,\rho,\mu) = (1,1,\infty) $ and the scattering wavefront set of Melrose \cite{Melrose94spectral} $ (\delta,\rho,\mu) = (1,0,\infty) $. 

\begin{theorem}
\label{thm::model-eq}
If $ u $ solves the equation~\eqref{Equation::model} with initial data $ u(0) = u_0 \in L^2(\Rd) $ and $ \mu \in \R \cup \{\infty\} $, then the following results of propagation hold:
\begin{enumerate}
\item \label{thm::model-eq-infinity} 
If $ \rho \gamma = \delta + \rho $, $ (x_0,\xi_0) \in  \WF_{\delta,\rho}^\mu(u_0) \backslash \{\xi = 0\}  $ and $ t_0 \in \R $, then
\begin{equation*}
(x_0+t_0 \gamma |\xi_0|^{\gamma-2}\xi_0,\xi_0) \in \WF_{\delta,\rho}^\mu(u(t_0)).
\end{equation*}
\item \label{thm::model-eq-finity} If $ \gamma > 1 $, $ \rho \gamma > \delta + \rho $, $ (x_0,\xi_0) \in \WF_{\delta,\rho}^\mu(u_0) \backslash \{\xi = 0\} $ and $ t_0 \ne 0 $, then
\begin{equation*}
(t_0 \gamma |\xi_0|^{\gamma-2}\xi_0,\xi_0) \in \WF_{\rho(\gamma-1),\rho}^\mu(u(t_0)).
\end{equation*}
\end{enumerate} 
\end{theorem}
Note that we do not require $ x_0 + t \gamma |\xi_0|^{\gamma-2}\xi_0 \ne 0 $ in Theorem~\ref{thm::model-eq}\eqref{thm::model-eq-infinity} while we require $ x_0 + t\xi_0 \ne 0 $ in Theorem~\ref{thm::doi-nakamura}\eqref{nakamura-doi}.
This is because in Theorem~\ref{thm::doi-nakamura} the geometry is only Euclidean at infinity.

\subsubsection{Asymptotically flat water waves}

Instead of the linearization at $ (\eta,\psi) = (0,0) $, if we paralinearize and symmetrize~\eqref{eq::equation-water-wave} as in \cite{ABZ11capillary}, then we obtain a quasilinear paradifferential fractional Schr\"odinger equation of order~$ 3/2 $.
We require the geometry of the free surface to be Euclidean at infinity and the velocity field to be zero at infinity to avoid problems caused by the infinite speed of propagation and the nonlinearity. 
We shall fulfill this requirement by proving the existence of gravity-capillary water waves in some weighted Sobolev spaces.

\begin{definition}
\label{def::weighted-Sobolev-spaces}
If $ \mu,k \in \R $, then $ \H{\mu}_k = \H{\mu}_k(\Rd) $ is the set of all~$ u \in \swtz'(\Rd) $ such that 
\begin{equation*}
\|u\|_{\H{\mu}_k} = \|\jp{x}^k\jp{D_x}^\mu u\|_{\Ltwo} < +\infty.
\end{equation*}
If in addition $ k \in \N $ and $ \delta \ge 0 $, then define
\begin{equation*}
\HC{\mu,\delta}_{k} = \bigcap_{j=0}^k H^{\mu - \delta j}_{j}.
\end{equation*}
\end{definition}
We are mostly interested in the case where $ \delta = 1/2 $.
The weighted Sobolev space $ \HC{\mu,1/2}_{k} $ is a natural space to apply the energy estimate for the fractional Schr\"odinger equation of order~$ 3/2 $ and thus also for the gravity-capillary water wave system.

\begin{theorem}
\label{thm::ww-weighted-sobolev-existence}
If $ d \ge 1 $, $ \mu > 3 + d/2 $, $ k \le 2\mu-d-6 $ and $ (\eta_0,\psi_0) \in \HC{\mu+1/2,1/2}_{k} \times \HC{\mu,1/2}_{k} $, then there exist $ T > 0 $ and a unique solution 
\begin{equation*}
(\eta,\psi) \in C([-T,T],\HC{\mu+1/2,1/2}_{k} \times \HC{\mu,1/2}_{k})
\end{equation*}
to the Cauchy problem of~\eqref{eq::equation-water-wave} with initial data $ (\eta_0,\psi_0) $.
\end{theorem}

The study of the Cauchy problem for the water wave equation dates back to Nalimov~\cite{Nalimov74CauchyPoisson}, Kano and Nishida~\cite{KN79ondes} and Yosihara~\cite{Yosihara82gravity, Yosihara83capillary}. 
The local well-posedness in Sobolev spaces with general initial data were achieved by Wu~\cite{Wu97:2d,Wu99:3d}, Beyer and G\"{u}nther~\cite{BG98capillary}.
Our analysis of the water wave equation relies on the paradifferential calculus of Bony~\cite{Bony86singularities} which was introduced to the study of the water wave equation by Alazard and Métivier in~\cite{AM09:paralinearization} and later allowed Alazard, Burq and Zuily~\cite{ABZ11capillary,ABZ14gravity} to prove the local well-posedness with low Sobolev regularities. 
For recent progress of the Cauchy problem, see e.g., \cite{AD15:2d:gravity:global,dPN16:strichartz,dPN17::paradiff,DIPP17:3d:gravity:capillary:global,HIT2016water,IT17:lifespan:2d:capillary,IP18:advances:global,MRT15:multi:solitons,RT11solitary,Wang16global:3d:capillary}.

To prove Theorem~\ref{thm::ww-weighted-sobolev-existence}, we shall combine the analysis in \cite{ABZ11capillary} and a paradifferential calculus in weighted Sobolev spaces. 
The latter can be achieved by modifying the definition of paradifferential operators via a spatial dyadic decomposition. 
More precisely, if~$ a $ is a symbol, then we define
\begin{equation*}
\P{a} = \sum_{j \in \N} \upsi_j \T{\psi_j a} \upsi_j,
\end{equation*} 
where $ \{\psi_j\}_{j\in\N} \subset \Ccinf(\Rd) $ is a dyadic partition of unity of $ \Rd $, $ \upsi_j = \sum_{|k-j|\le N} \psi_k $ for some sufficiently large $ N\in\N $, and $ \T{\psi_j a} $ is the usual paradifferential operator of Bony. 
Such dyadic paradifferential calculus inherits the symbolic calculus and the paralinearization of Bony's calculus while at the same time allows the spatial polynomial growth/decay of symbols to play their roles in estimates.

We do not attempt to lower~$ \mu $ to $ > 2 + d/2 $ as it was in \cite{ABZ11capillary}. 
The range of~$ k $ is so chosen such that $ \mu - k/2 > 3 + d/2 $, enabling us to paralinearize~\eqref{eq::equation-water-wave} in~$ \HC{\mu}_k $. 
We should mention that the existence of gravity water waves (water waves without surface tension) in uniformly local weighted Sobolev spaces was obtain by Nguyen~\cite{Nguyen16:pseudolocal} via a periodic spatial decomposition from~\cite{ABZ16nonlocal}.

\subsubsection{Propagation at infinity}

Our first main result concerns the propagation of quasi-homo\-geneous wavefront sets with parameters $ (\delta,\rho)=(1/2,1) $, corresponding to Theorem~\ref{thm::model-eq}\ref{thm::model-eq-infinity}.

\begin{theorem}
\label{thm::main-infinite}
Suppose that $ d \ge 1 $, $ \mu > 3 + d/2 $, $ 3 \le k < 2\mu-K-d $ for some $ K > 0 $, and
\begin{equation*}
(\eta,\psi) \in C([-T,T],\HC{\mu+1/2,1/2}_{k} \times \HC{\mu,1/2}_{k}),
\end{equation*}
where $ T > 0 $, solves~\eqref{eq::equation-water-wave}.
If $ t_0 \in [-T,T] $ and
\begin{equation*}
(x_0,\xi_0) \in \WF_{1/2,1}^{\mu+1/2+\sigma}(\eta(0)) \cup \WF_{1/2,1}^{\mu+\sigma}(\psi(0)),
\end{equation*}
such that $ \xi_0 \ne 0 $, $ 0 \le \sigma \le k/2 - 3/2 $ and
\begin{equation*}
x_0+\frac{3}{2}t|\xi_0|^{-1/2}\xi_0 \ne 0
\end{equation*}
for all~$ t $ between~$ 0 $ and~$ t_0 $, then
\begin{equation*}
\Big(x_0+\frac{3}{2}t_0|\xi_0|^{-1/2}\xi_0,\xi_0\Big) \in \WF_{1/2,1}^{\mu+1/2+\sigma}(\eta(t_0)) \cup \WF_{1/2,1}^{\mu+\sigma}(\psi(t_0)).
\end{equation*}
\end{theorem}

We will see that, by Lemma~\ref{lem::basic-properties-WF}, if $ (\eta,\psi) \in \HC{\mu+1/2,1/2}_{k} \times \HC{\mu,1/2}_{k} $, then
\begin{equation*}
\WF_{1/2,1}^{\mu+1/2}(\eta) \cup \WF_{1/2,1}^{\mu}(\psi) \subset \{x = 0\} \cup \{\xi = 0\}.
\end{equation*} 
By \cite{AM09:paralinearization}, we expect~$ \sigma $ to be at most $ \mu - \alpha - d/2 $ for some $ \alpha > 0 $, corresponding to the gain of regularity by the remainder in the paralinearization procedure. 
Theorem~\ref{thm::main-infinite} does not give the optimal upper bound for~$ \sigma $, as it is not our priority, but when $ k = 2 \mu - K - d $, the parameter $ \sigma $ can still be as large as $ \mu - K/2 - d/2 - 3/2 $, almost reaching the paradifferential threshold.

\subsubsection{Microlocal smoothing effect}
\label{sec::intro-Microlocal-Smoothing-Effect}

Our second main result shows that singularities of the initial data which are non-trapped with respect to the initial geometry, instantaneously generate an element in the quasi-homogeneous wavefront set with parameters $ (\delta,\rho) = (1/2,1) $, corresponding to Theorem~\ref{thm::model-eq}\eqref{thm::model-eq-finity}. 

Observe that if $ \eta $ is sufficiently regular, then $ \Sigma $ endowed with the metric inherited from $ \R^{d+1} $ is isometric to $ (\Rd,\varrho) $ where
\begin{equation*}
\varrho = \begin{pmatrix}
\Id + (\nabla\eta) \tensor[^t]{(\nabla \eta)}{} & \nabla \eta \\ \tensor[^t]{(\nabla \eta)}{} & 1
\end{pmatrix}.
\end{equation*}
Denote $ \Sigma_0 = \Sigma|_{t=0} $ and $ \varrho_0^{} = \varrho|_{t=0} $. 
We identify the cogeodesic flow $ \Geo $ on $ T^*\Sigma_0 $ with the Hamiltonian flow on $ \R^{2d} $ of the symbol
$ G(x,\xi) = {}^t\xi \varrho_0^{}(x)^{-1}\xi. $
Precisely $ \Geo = \Geo_s(x,\xi) $ is defined by the equation
\begin{equation}
\label{eq:def-cogeodesic-flow}
\partial_s \Geo_s = (\pxi G,-\px G)(\Geo_s), \quad
\Geo_0 = \Id_{\R^{2d}}.
\end{equation}

\begin{definition}
A point $ (x_0,\xi_0) \in \Rd \times (\Rd \backslash 0) $ is called forwardly resp.\ backwardly non-trapped with respect to~$ \Geo $ if, with an abuse of notation, the cogeodesic $ \{(x_s,\xi_s) = \Geo_s(x_0,\xi_0)\}_{s\in\R} $ satisfies
\begin{equation*}
\lim_{s\to+\infty} |x_s| = \infty, 
\quad \text{resp.} \quad
\lim_{s\to-\infty} |x_s| = \infty.
\end{equation*}
\end{definition}

\begin{theorem}
\label{thm::main-finite}
If $ d \ge 1 $, $ \mu > 3+d/2 $, $ 3 \le k < \frac{2}{3}(\mu-1-d/2) $, and
\begin{equation*}
(\eta,\psi) \in C([-T,T],\HC{\mu+1/2,1/2}_{k} \times \HC{\mu,1/2}_{k}),
\end{equation*}
where $ T > 0 $, solves the equation~\eqref{eq::equation-water-wave}. 
Let
\begin{equation*}
(x_0,\xi_0) \in \WF_{0,1}^{\mu+1/2+\sigma}(\eta(0)) \cup \WF_{0,1}^{\mu+\sigma}(\psi(0)),
\end{equation*}
where $ \xi_0 \ne 0 $ and $ 0 \le \sigma \le {3 \over 2} k $.
If $ (x_0,\xi_0) $ is forwardly resp.\ backwardly non-trapped, and let the cogeodesic $ \{(x_s,\xi_s)\}_{s\in\R} $ be defined as above, then there exists $ \xi_{+\infty} $, resp.\ $ \xi_{-\infty} $ in $ \Rd \backslash \{0\}$ such that,
\begin{equation*}
\lim_{s\to\infty} \xi_s = \xi_{+\infty}, 
\quad \text{resp.} \quad 
\lim_{s\to\infty} \xi_{-s} = \xi_{-\infty},
\end{equation*}
and moreover, for all $ 0 < t_0 \le T $, resp.\ $ -T \le t_0 < 0 $, we have
\begin{align*}
\Big(\frac{3}{2}t_0|\xi_{+\infty}|^{-1/2}\xi_{+\infty},\xi_{+\infty}\Big) & \in \WF_{1/2,1}^{\mu+1/2+\sigma}(\eta(t_0)) \cup \WF_{1/2,1}^{\mu+\sigma}(\psi(t_0)), \\
\text{resp.} \quad
\Big(\frac{3}{2}t_0|\xi_{-\infty}|^{-1/2}\xi_{-\infty},\xi_{-\infty}\Big) & \in \WF_{1/2,1}^{\mu+1/2+\sigma}(\eta(t_0)) \cup \WF_{1/2,1}^{\mu+\sigma}(\psi(t_0)).
\end{align*}
\end{theorem}

We remark that the asymptotic directions $ \xi_{\pm\infty} $ are determined solely by the geometry of $ \Sigma_0 $.
This is due to the infinite speed of propagation.
We can also prove that, the non-trapping assumption is, at least in the following two cases, unnecessary: if $ d = 1 $, or if $ \nabla\eta(0) \in \Linf $ and $ \|\jp{x}\nabla^2 \eta(0)\|_{\Linf} $ is sufficiently small. 
In both cases we obtain the following local smoothing effect.

\begin{corollary}
\label{cor::local-smoothing-effect}
If $ d $, $ \mu $, $ k $, $ \sigma $ satisfy the hypothesis of the previous theorem, $ T > 0 $,
\begin{equation*}
(\eta,\psi) \in C([-T,T],\HC{\mu+1/2,1/2}_{k} \times \HC{\mu,1/2}_{k})
\end{equation*} 
solves the equation~\eqref{eq::equation-water-wave}, and both of the following two conditions are satisfied:
\begin{enumerate}[label=(\arabic*)]
\item Either $ d = 1 $ or $ \|\jp{x} \nabla^2 \eta(0)\|_{\Linf} $ is sufficiently small; and
\item $ \WF_{1/2,1}^{\mu+1/2+\sigma}(\eta(0)) \cup \WF_{1/2,1}^{\mu+\sigma}(\psi(0)) \subset \{x=0\} \cup \{\xi = 0\}  $,
\end{enumerate}
then for all $ t_0 \in [-T,T] \backslash \{0\} $ and for all $ \epsilon > 0 $, 
\begin{equation*}
(\eta(t_0),\psi(t_0)) \in \H{\mu+1/2+\sigma-\epsilon}_\loc \times \H{\mu+\sigma-\epsilon}_\loc.
\end{equation*}
\end{corollary}

The second condition is satisfied if, by Lemma~\ref{lem::basic-properties-WF}, there exists $ (k,k') \in \R^2 $ such that
\begin{equation*}
(\eta(0),\psi(0)) \in \H{\mu+1/2+\sigma-k}_{2k} \times \H{\mu+\sigma-k'}_{2k'}.
\end{equation*}
This is particularly the case if $ (\eta(0),\psi(0)) \in \mathscr{E}'(\Rd) \times\mathscr{E}'(\Rd) $.

We refer to Christianson, Hur and Staffilani~\cite{CHS09:smoothing}, Alazard, Burq and Zuily~\cite{ABZ11capillary} for local smoothing effects of 2D capillary-gravity water waves. 
See also Alazard, Ifrim and Tataru~\cite{AIT18morawetz} for a Morawetz inequality of 2D gravity water waves.

\subsection{Outline of paper}

In \S\ref{sec::quasi-homogeneous-microlocal-analysis}, we present basic properties of weighted Sobolev spaces and the quasi-homogeneous wavefront set.
In \S\ref{sec::proof-model-eq}, we prove Theorem~\ref{thm::model-eq} by extending the idea of Nakamura.
In \S\ref{sec::paradiff-calculus}, we review the paradifferential calculus of Bony, and extend it to weighted Sobolev spaces by a spatial dyadic decomposition. We also develop a quasi-homogeneous semiclassical paradifferential calculus, and study its relations with the quasi-homogeneous wavefront set.
In \S\ref{sec::asymptotically-flat-ww}, we study the Dirichlet--Neumann operator in weighted Sobolev spaces and prove the existence of asymptotically flat gravity-capillary water waves, i.e., Theorem~\ref{thm::ww-weighted-sobolev-existence}.
In \S\ref{sec::proof-main-thm}, we prove our main results, i.e., Theorem~\ref{thm::main-infinite}, Theorem~\ref{thm::main-finite} and Corollary~\ref{cor::local-smoothing-effect}, by extending the proof of Theorem~\ref{thm::model-eq} to the quasilinear equation using the paradifferential calculus.

\section*{Acknowledgment}

The author would like to thank Thomas Alazard, Nicolas Burq and Claude Zuily for their constant support and encouragement. He would like to thank Shu Nakamura for helpful discussions during the early state of this project. He would also like to thank Jean-Marc Delort for his careful reading of the manuscript, and thank Daniel Tataru for his useful comments.
Finally the author would like to thanks the anonymous referees for their detailed comments which result in a better presentation of the paper.

\section{Quasi-homogeneous microlocal analysis}

\label{sec::quasi-homogeneous-microlocal-analysis}

In this section we develop the quasi-homogeneous semiclassical calculus and discuss its relation with weighted Sobolev spaces and the quasi-homogeneous wavefront set.

\subsection{Quasi-homogeneous semiclassical calculus}

\begin{definition}
For $ (\mu,k) \in \R^2 $, set $ m^\mu_k(x,\xi) = \jp{x}^k \jp{\xi}^\mu $. Let $ a_h \in \Cinf(\R^{2d}) $, we say that $ a_h \in S^\mu_k $ if for all $ \alpha, \beta \in \N^d $, there exists $ C_{\alpha\beta} > 0 $, such that for all $ (x,\xi) \in \R^{2d} $,
\begin{equation}
\label{eq:symbol-seminorm-def}
\sup_{h \in (0,1]} \big|\px^\alpha \pxi^\beta a_h(x,\xi) \big| \le C_{\alpha\beta} m^{\mu-|\beta|}_{k-|\alpha|}(x,\xi).
\end{equation}
We say that $ a_h \in S^\mu_k $ is $ (\mu,k) $-elliptic if there exists $ R > 0, C > 0 $ such that for $ |x| + |\xi| \ge R $,
\begin{equation*}
\inf_{h \in (0,1]} |a_h(x,\xi)| \ge C m^{\mu}_k(x,\xi).
\end{equation*}
Also write 
$ S^\infty_{\infty} = \cup_{(\mu,k) \in \R^2} S^\mu_k, $ and $S^{-\infty}_{-\infty} = \cap_{(\mu,k) \in \R^2} S^\mu_k. $

We say that $ a_h \in S^{-\infty}_{-\infty} $ is elliptic at $ (x_0,\xi_0) $ if for some neighborhood~$ \Omega $ of $ (x_0,\xi_0) $,
\begin{equation*}
\inf_{h \in (0,1]}  \inf_{(x,\xi) \in \Omega}  |a_h(x,\xi)| > 0.
\end{equation*}
\end{definition}

\begin{definition}
\label{def:quasi-hom-semi-calc}
Let $ \delta,\rho \in \R $ such that $ \delta + \rho > 0 $ and for all $ h \in (0,1] $, define the scaling
\begin{equation}
\label{eq:def:quasi-homogeneous-scaling}
\theta^{\delta,\rho}_h: (x,\xi) \mapsto (h^\delta x,h^\rho \xi)
\end{equation}
which induces a pullback~$ \theta^{\delta,\rho}_{h,*} $ on~$ S^{\infty}_{\infty} $:
$ \theta^{\delta,\rho}_{h,*} a_h = a_h \comp \theta^{\delta,\rho}_h. $
Then define, by~\eqref{eq:def:quantization},
\begin{equation*}
\op_h^{\delta,\rho}(a_h) = \op(\theta^{\delta,\rho}_{h,*} a_h) = a(h^\delta x,h^\rho D_x).
\end{equation*}

\end{definition}

The scaling
$ \vartheta_h^\delta u(x) = h^{\delta d/2} u(h^\delta x) $
defines an isometry on $ \Ltwo(\Rd) $.
Therefore, by the formula
\begin{equation}
\label{eq:scaling-transform}
(\vartheta_h^\delta)^{-1} \op_h^{\delta,\rho}(a) \vartheta_h^\delta = \op_h^{0,\delta+\rho}(a),
\end{equation} 
we deduce the following results from the usual semiclassical calculus for which we refer to~\cite{Zworski12semiclassical}.

\begin{proposition}
\label{prop::Ltwo-Pseu-D-O}
There exists $ K > 0 $ such that, if $ a \in \Cinf(\R^{2d}) $ with $ \| \px^\alpha \pxi^\beta a \|_{\Linf} \le M $ for all $ |\alpha|+|\beta| \le d $, then $ \op_h^{\delta,\rho}(a) :L^2 \to L^2 $ and 
$ \|\op_h^{\delta,\rho}(a)\|_{\Ltwo \to \Ltwo} \le KM. $
\end{proposition}

\begin{proposition}
\label{prop::quasi-homogeneous-composition}
There exists a bilinear operator 
$ \sharp^{\delta,\rho}_h : S^{\infty}_{\infty} \times S^{\infty}_{\infty} \to S^{\infty}_{\infty}, $
such that 
\begin{equation*}
\op_h^{\delta,\rho}(a_h) \op_h^{\delta,\rho}(b_h) = \op_h^{\delta,\rho}(a_h \sharp_h^{\delta,\rho} b_h). 
\end{equation*}
Moreover, if $ a_h \in S^\mu_k $ and $ b_h \in S^\nu_\ell $, then $ a_h \sharp_h^{\delta,\rho} b_h \in S^{\mu+\nu}_{k+\ell} $. 
For all $ r > 0 $, define 
\begin{equation}
\label{eq:def-symbolic-composition-finite}
a_h \sharp^{\delta,\rho}_{h,r} b_h = \sum_{|\alpha| < r} \frac{h^{|\alpha|(\delta+\rho)}}{\alpha!} \pxi^\alpha a_h D_x^\alpha b_h.
\end{equation}
Then we have
\begin{equation*}
a_h \sharp^{\delta,\rho}_{h} b_h - a_h \sharp^{\delta,\rho}_{h,r} b_h = \O(h^{r(\delta+\rho)})_{S^{\mu+\nu-r}_{k+\ell-r}}.
\end{equation*}
\end{proposition}

\begin{proposition}
\label{prop::adjoint-quasi-homogeneous}
There exists a linear operator
$ \zeta^{\delta,\rho}_h : S^{\infty}_{\infty} \to S^{\infty}_{\infty} $
such that 
\begin{equation*}
\op_h^{\delta,\rho}(a_h)^* = \op_h^{\delta,\rho}(\zeta^{\delta,\rho}_h a_h).
\end{equation*}
Moreover iff $ a_h \in S^\mu_k $, then $ \zeta^{\delta,\rho}_h a_h \in S^\mu_k. $
For $ r > 0 $, define 
\begin{equation}
\label{eq:def-symbolic-adjoint-finite}
\zeta^{\delta,\rho}_{h,r} a_h 
= \sum_{|\alpha| < r} \frac{h^{|\alpha|(\delta+\rho)}}{\alpha!} \pxi^\alpha D_x^\alpha \overline{a_h}.
\end{equation}
Then we have
\begin{equation*}
\zeta^{\delta,\rho}_{h} a_h - \zeta^{\delta,\rho}_{h,r} a_h = \O(h^{r(\delta+\rho)})_{S^{\mu-r}_{k-r}}.
\end{equation*} 
\end{proposition}

\begin{proposition}[Sharp G{\aa}rding Inequality]
\label{prop::garding-quasi-homogeneous}
If $ \delta + \rho > 0 $ and $ a_h \in S^0_0 $ such that $ \Re\,a_h \ge 0 $, then there exists $ C > 0 $ such that for all $ u \in \Ltwo(\Rd) $ and $ 0 < h < 1 $, we have
\begin{equation*}
\Re(\op_h^{\delta,\rho}(a_h) u, u)_{\Ltwo} \ge -Ch^{\delta+\rho} \|u\|_{\Ltwo}^2.
\end{equation*}
\end{proposition}

\subsection{Weighted Sobolev spaces}

Recall the weighted Sobolev spaces defined in Definition~\ref{def::weighted-Sobolev-spaces}.

\begin{proposition}
\label{prop::Schwartz-vs-Weighted-Sobolev}
We have $ \swtz(\Rd) = \cap_{\mu,k\in\R} H^{\mu}_k $ and $ \swtz'(\Rd) = \cup_{\mu,k\in\R} H^{\mu}_k $.
\end{proposition}
\begin{proof}
Clearly $ \swtz(\Rd) \subset \cap_{\mu,k\in\R} H^{\mu}_k $. 
The converse follows by the Sobolev embedding theorems. 
As for the second statement, clearly $ \cup_{(\mu,k)\in\R^2} H^{\mu}_k \subset \swtz'(\Rd) $. 
Conversely, if $ u \in \swtz'(\Rd) $, then there exists $ N > 0 $, such that for all $ \varphi \in \swtz(\Rd) $, we have
\begin{equation*}
\langle u,\varphi \rangle_{\swtz',\swtz} 
\lesssim \sum_{|\alpha|+|\beta| \le N} \| x^\alpha \px^\beta \varphi \|_{\Linf}
\lesssim \|\op(m^N_N) \varphi\|_{\Ltwo}.
\end{equation*}
By duality this implies that $ u \in \H{-N}_{-N} $.
\end{proof}

\begin{lemma}
\label{lem::property-u_h-finite-order}
If $ u \in \swtz'(\Rd) $, then there exists $ N > 0 $ such that
\begin{equation*}
u = h^{-N} \op_h^{\delta,\rho}(m^{-N}_{-N}) \O(1)_{\Ltwo}
\end{equation*}
Therefore, if $ \delta + \rho > 0 $, and $ a_h \in \O(h^\infty)_{S^{-\infty}_{-\infty}} $, then 
$ \op_h^{\delta,\rho}(a_h) u_h = \O(h^{\infty})_{\swtz}. $
\end{lemma}
\begin{proof}
By the proof of Proposition~\ref{prop::Schwartz-vs-Weighted-Sobolev}, there exists $ N > 0 $, such that for all $ \varphi \in \swtz(\Rd) $,
\begin{equation*}
\langle u,\varphi \rangle_{\swtz',\swtz} 
\lesssim \sum_{|\alpha|+|\beta| \le N} \| x^\alpha \px^\beta \varphi \|_{\Linf}
\lesssim h^{-N} \|\op_h^{\delta,\rho}(m^N_N) \varphi\|_{\Ltwo}.
\end{equation*}
Again we conclude by duality.
\end{proof}

\begin{definition}
\label{def::order-weight-regularity}
We say that a linear operator $ \A : \swtz(\Rd) \to \swtz'(\Rd) $ is of order $ (\nu,\ell) \in \R^2 $, and denote $ \A \in \OP{\nu}{\ell} $ if for all $ (\mu,k) \in \R^2 $, there exists $ C > 0 $ such that for all $ u \in \swtz(\Rd) $, we have
\begin{equation*}
\|\A u\|_{\H{\mu-\nu}_{k-\ell}} \le C \|u\|_{\H{\mu}_k}.
\end{equation*}
Therefore $ \A $ extends to a bounded linear operator from $ \H{\mu}_k $ to $ \H{\mu-\nu}_{k-\ell} $. We denote $ \mathcal{A} \in \OP{-\infty}{-\infty} $ if $ \A \in \OP{\nu}{\ell} $ for all $ (\nu,\ell) \in \R^2 $.

Let $ \mathscr{A} $ be any nonempty set. 
Let $ \A_\alpha : \swtz(\Rd) \to \swtz'(\Rd) $ and $ C_\alpha > 0 $ be indexed by $ \alpha \in \mathscr{A} $. 
We say $ \A_\alpha = \O(C_\alpha)_{\OP{\nu}{\ell}} $ if for all $ (\mu,k) \in \R^2 $, there exists $ K > 0 $ such that for all $ \alpha \in \mathscr{A} $, we have
\begin{equation*}
\|\A_\alpha\|_{\H{\mu}_k\to\H{\mu-\nu}_{k-\ell}} \le K C_\alpha.
\end{equation*}
\end{definition}

By Proposition~\ref{prop::Ltwo-Pseu-D-O} and Proposition~\ref{prop::quasi-homogeneous-composition}, we obtain

\begin{proposition}
The following mapping properties of pseudodifferential operators hold:
\begin{enumerate}
\item If $ a_h \in S^\nu_\ell $ with $ (\nu,\ell) \in \R^2 $, then $ \op(a_h) \in \OP{\nu}{\ell} $. 
\item If $ u \in \swtz'(\Rd) $, then $ u \in \H{\mu}_k $ if and only if there exists a $ (\mu,k) $-elliptic symbol $ a_h \in S^{\mu}_k $ such that $ \op(a_h) u = \O(1)_{L^2} $.
\end{enumerate}
\end{proposition}

Next, we characterize weighted Sobolev spaces by a dyadic decomposition.

\begin{definition}
The set $ \Partition $ consists of all maps of the form 
\begin{equation*}
\psi : \N \to \Ccinf(\Rd), \quad j \mapsto \psi_j
\end{equation*}
such that the following conditions are satisfied:
\begin{enumerate}
\item There exists $ C > 1 $ such that for all $ j \ge 1 $ we have
\begin{equation*}
\supp \psi_j \subset \{x \in \Rd:C^{-1} 2^{j} \le |x| \le C 2^{j}\} ;
\end{equation*}
\item For all $ j \ge 0 $, the function $ \psi_j $ is non-negative.
\item There exists $ C > 1 $ such that
$ C^{-1} \le \sum_{j\in \N} \psi_j \le C. $
\item For all $ \alpha \in \N $ there exists $ C_\alpha $ such that for all $ j \in \N $ we have
\begin{equation*}
\|\px^\alpha \psi_j \|_{\Linf} \le C_\alpha 2^{-j|\alpha|}.
\end{equation*}
\end{enumerate}
The set $ \Partition_* $ consists of all $ \psi \in \Partition $ such that
\begin{enumerate}[resume]
\item $ \sum_{j\in \N} \psi_j = 1 $; and
\item $ \supp \psi_j \cap \supp \psi_k = \emptyset $ whenever $ |j-k| > 2 $.
\end{enumerate}
If $ \psi,\tilde{\psi} \in \Partition $ such that $ \psi_j \tilde{\psi}_j = \psi_j $ for all $ j \in \N $, then we denote $ \psi \subset \subset \tilde{\psi} $.
\end{definition}

\begin{proposition}
\label{prop::characterization-weighted-Sobolev}
If $ \mu,k \in \R $, $ \psi \in \Partition $ and $ u \in \swtz'(\Rd) $, then $ u \in \H{\mu}_k $ if and only if 
\begin{equation*}
\sum_{j \in \N} 2^{2jk} \|\psi_j u\|_{\H{\mu}}^2 < \infty.
\end{equation*}
Moreover, there exists $ C > 1 $, such that for all $ u \in \H{\mu}_k $, we have
\begin{equation*}
C^{-1} \|u\|_{\H{\mu}_k}^2 \le  \sum_{j \in \N} 2^{2jk} \|\psi_j u\|_{\H{\mu}}^2  \le C \|u\|_{\H{\mu}_k}^2.
\end{equation*}
\end{proposition}
\begin{proof}
We may assume that $ \psi \in \Partition_* $ because if $ \phi^1,\phi^2 \in \Partition $ then
\begin{equation*}
\sum_{j \in \N} 2^{2jk} \|\phi^1_j u\|_{\H{\mu}}^2 \simeq  \sum_{j \in \N} 2^{2jk} \|\phi^2_j u\|_{\H{\mu}}^2, 
\end{equation*}
Define $ \tilde{\psi} \in \Partition $ by setting $ \tilde{\psi}_j = \sum_{|k-j|\le 2} \psi_k $ for all $ j \in \N $.
Then $ \psi \subset \subset \tilde{\psi} $.
Note that the family of multiplication operators $ \{2^{-jk} \jp{x}^k \tilde{\psi}_j\}_{j\in\mathbb{N}} $ is bounded in $ \OP{0}{0} $, which implies that for all $ \mu \in \mathbb{N} $, the family of pseudodifferential operators $ \{2^{-jk} \jp{D_x}^\mu \jp{x}^k \tilde{\psi}_j \jp{D_x}^{-\mu} \}_{j\in\mathbb{N}} $ is bounded in $ \OP{0}{0} $.
Therefore, for all $ u \in \H{\mu}_k $, we have
\begin{equation}
\label{eq:dyadic-estimate}
2^{2jk} \|\psi_j u\|_{\H{\mu}}^2 
\lesssim \|\jp{D_x}^\mu \psi_j \jp{x}^k u\|_{\Ltwo}^2 
\lesssim \|\tilde{\psi}_{j} \jp{D_x}^\mu \psi_j \jp{x}^k u\|_{\Ltwo}^2 
+ \|(1-\tilde{\psi}_{j}) \jp{D_x}^\mu \psi_j \jp{x}^k u \|_{\Ltwo}^2.
\end{equation}
Apply Proposition~\ref{prop::quasi-homogeneous-composition} with $ (\delta,\rho) = (1,0) $ and $ h = 2^{-j} $, we obtain that for all $ N > 0 $, the estimate
\begin{equation}
\label{eq:dyadic-estimate-remainder}
(1-\tilde{\psi}_{j}) \jp{D_x}^\mu \psi_j \jp{D_x}^{-\mu} = \O(2^{-jN})_{\Ltwo\to\Ltwo},
\end{equation}
holds uniformly for all $ j \in \N $. 
Therefore,
\begin{equation*}
\sum_{j\in\N} \|(1-\tilde{\psi}_{j}) \jp{D_x}^\mu \psi_j \jp{x}^k u \|_{\Ltwo}^2
\lesssim \sum_{j\in\N} 2^{-2jN} \|u\|_{\H{\mu}_k}^2
\lesssim \|u\|_{\H{\mu}_k}^2.
\end{equation*}
For $ r \in \{0,1,\ldots,9\} $, set 
\begin{equation*}
a_r = \sum_{j \in 10\N + r} \tilde{\psi}_{j} \jp{\xi}^\mu \sharp (\jp{x}^k \psi_j) \in S^\mu_k,
\end{equation*}
where $ \sharp = \sharp^{0,0}_1 $.
Observe that if $ 0 \ne j - j' \in 10 \N $, then $ \supp \tilde{\psi}_{j} \cap \supp \tilde{\psi}_{j'} = \emptyset$.
Therefore,
\begin{equation}
\label{eq:dyadic-estimate-main}
\sum_{j \in \N} \|\tilde{\psi}_{j} \jp{D_x}^\mu \psi_j \jp{x}^k u\|_{\Ltwo}^2
= \sum_{r=0}^9 \sum_{j \in 10 \N + r} \|\tilde{\psi}_{j} \jp{D_x}^\mu \psi_j \jp{x}^k u\|_{\Ltwo}^2
= \sum_{r=0}^9 \|\op(a_r) u\|_{\Ltwo}^2 \lesssim \|u\|_{\H{\mu}_k}^2.
\end{equation}
Combining~\eqref{eq:dyadic-estimate}, \eqref{eq:dyadic-estimate-remainder} and \eqref{eq:dyadic-estimate-main}, we prove that if $ u \in H^\mu_k $ then
$ \sum_{j \in \N} 2^{2jk} \|\psi_j u\|_{\H{\mu}}^2  \lesssim \|u\|_{\H{\mu}_k}^2. $

Conversely, assume that
$ \sum_{j \in \N} 2^{2jk} \|\psi_j u\|_{\H{\mu}}^2<\infty. $
Similarly as above, we have
\begin{equation}
\label{eq:dyadic-estimate-reverse-main}
\begin{split}
\infty > \sum_{j\in\N} 2^{2jk} \|\jp{D_x}^\mu \psi_j u\|_{\Ltwo}^2
& \gtrsim \sum_{r=0}^9 \sum_{j\in 10\N+r} \|\tilde{\psi}_{j} \jp{D_x}^\mu \psi_j \jp{x}^k u\|_{\Ltwo}^2 \\
& \gtrsim \sum_{r=0}^9 \|\op(a_r) u\|_{\Ltwo}^2 \gtrsim \|\op(a) u\|_{\Ltwo}^2
\end{split}
\end{equation}
where $ a = \sum_{r=0}^9 a_r $.
Observe $ a $ is~$ (\mu,k) $-elliptic, so $ u \in H^\mu_k $.
By the symbolic calculus, there exists $ r \in S^{-\infty}_{-\infty} $ such that
\begin{equation}
\label{eq:dyadic-estimate-symbolic-decomposition}
\|u\|_{\H{\mu}_k}^2 \lesssim \|\op(a) u\|_{\Ltwo}^2 + \|\op(r)u\|_{\Ltwo}^2.
\end{equation}
For the remainder term, we have
\begin{equation*}
\|\op(r)u\|_{\Ltwo}^2
= (u,\op(r^*\sharp r) u)_{\Ltwo}
= \sum_{j\in\N} (u,\op(r^*\sharp r) \psi_j u)_{\Ltwo}.
\end{equation*}
For each term in the summation, by the analysis above~\eqref{eq:dyadic-estimate}, we have for all $ N > 0 $ and $ \varepsilon > 0 $,
\begin{align*}
(u,\op(r^* \sharp r) \psi_j  u)_{\Ltwo}
& = (\op(m^\mu_k) u, \op(m^{-\mu}_{-k} \sharp r^* \sharp r \sharp m^{-\mu}_{N-k}) \jp{D_x}^\mu \jp{x}^{-N+k} \psi_j u )_{\Ltwo} \\
& \lesssim \|u\|_{\H{\mu}_k} \|\jp{D_x}^\mu \jp{x}^{-N+k} \psi_j u\|_{\Ltwo} \\
& \lesssim 2^{-jN} \|u\|_{\H{\mu}_k} 2^{jk}\|\jp{D_x}^\mu \psi_j u\|_{\Ltwo} \\
& \lesssim 2^{-jN} \big( \varepsilon \|u\|_{\H{\mu}_k}^2 + \varepsilon^{-1} 2^{2jk}\|\jp{D_x}^\mu \psi_j u\|_{\Ltwo}^2 \big),
\end{align*}
where the constants are independent of~$ \varepsilon $.
Summing up in~$ j $, 
\begin{equation}
\label{eq:dyadic-estimate-reverse-remainder}
\|\op(r)u\|_{\Ltwo}^2
\lesssim \varepsilon \|u\|_{\H{\mu}_k}^2 + \varepsilon^{-1} \sum_{j\in\N} 2^{2jk} \|\psi_j u\|_{\H{\mu}}^2.
\end{equation}
By choosing~$ \varepsilon $ sufficiently small, we conclude by~\eqref{eq:dyadic-estimate-reverse-main}, \eqref{eq:dyadic-estimate-symbolic-decomposition} and~\eqref{eq:dyadic-estimate-reverse-remainder} that
\begin{equation*}
\|u\|_{\H{\mu}_k}^2 \lesssim \sum_{j \in \N} 2^{2jk} \|\psi_j u\|_{\H{\mu}}^2
\end{equation*}
and finishes the proof.
\end{proof}

\subsection{The quasi-homogeneous wavefront set}

In this section the parameters $ \delta,\rho,\mu $ satisfy the conditions in Definition~\ref{def::intro-quasi-homogeneous-WF} without further specification.
By Lemma~\ref{lem::property-u_h-finite-order}, the following characterization of the quasi-homogeneous wavefront set is easy to prove by the symbolic calculus.

\begin{proposition}
If $ u \in \swtz'(\Rd)$, then $ (x_0,\xi_0) \notin \WF^\mu_{\delta,\rho}(u) $ if and only if there exists $ a_h \in S^{-\infty}_{-\infty} $ which is elliptic at~$ (x_0,\xi_0) $ such that $ \op_h^{\delta,\rho}(a_h) u = \O(h^\mu)_{\Ltwo} $ for $ h \in (0,1] $.
\end{proposition}

\begin{lemma}
\label{lem::support=decay}
If $ u \in \swtz'(\Rd)$ and $ a_h \in S^{-\infty}_{-\infty} $ is such that
\begin{equation*}
\overline{\cup_{h\in(0,1]} \supp a_h} \subset \subset \R^{2d} \backslash \WF_{\delta,\rho}^\mu(u),
\end{equation*}
then 
$ \langle u, \op_h^{\delta,\rho}(a_h) u \rangle_{\swtz',\swtz} = \O(h^{2\mu}) $
and consequently $ \op_h^{\delta,\rho}(a_h) u = \O(h^\mu)_{\Ltwo} $ for $ h \in (0,1] $.
\end{lemma}
\begin{proof}
Let $ K = \overline{\cup_{h\in(0,1]} \supp a_h} $ and let $ \{\Omega_i\}_{i\in I} $ be an open cover of~$ K $. 
Let $ b_h^i \in S^{-\infty}_{-\infty} $ be elliptic everywhere in~$ \Omega_i $, such that $ \op_h^{\delta,\rho}(b_h^i) u  = \O(h^\mu)_{\Ltwo}. $ By a partition of unity, we may assume that $ K \subset \Omega \bydef \Omega_{i_0} $ for some $ i_0 \in I $, and set $ b_h = b_h^{i_0} $. By the ellipticity of~$ b_h $, we can find $ c_h \in S^{-\infty}_{-\infty} $ and $ r_h = \O(h^N)_{S^{-\infty}_{-\infty}} $ for some large $N > 0$ such that 
$ a_h = (\zeta_h^{\delta,\rho} b_h) \sharp_h^{\delta,\rho} c_h \sharp_h^{\delta,\rho} b_h + r_h. $
Therefore, by Lemma~\ref{lem::property-u_h-finite-order},
\begin{align*}
\langle u, \op_h^{\delta,\rho}(a_h) u \rangle_{\swtz',\swtz}
& = ( \op_h^{\delta,\rho}(b_h) u, \op_h^{\delta,\rho}(c_h) \op_h^{\delta,\rho}(b_h) u )_{\Ltwo} \\
& \quad \quad + \langle u, \op_h^{\delta,\rho}(r_h) u \rangle_{\swtz',\swtz} 
= \O(h^\mu)^2 + \O(h^\infty) = \O(h^{2\mu}).
\end{align*}
Next, observe that there exists $ w_h \in S^{-\infty}_{-\infty} $ and $ \tilde{r}_h = \O(h^N)_{S^{-\infty}_{-\infty}} $ such that $ \supp w_h \subset K $ and
\begin{equation*}
\op_h^{\delta,\rho}(a_h)^* \op_h^{\delta,\rho}(a_h) = \op_h^{\delta,\rho}(w_h) + \op_h^{\delta,\rho}(\tilde{r}_h).
\end{equation*}
Therefore,
\begin{align*}
\|\op_h^{\delta,\rho}(a_h) u\|_{L^2}^2
& = \jp{u, \op_h^{\delta,\rho}(a_h)^* \op_h^{\delta,\rho}(a_h) u}_{\swtz',\swtz}\\
& = \jp{u,\op_h^{\delta,\rho}(w_h)u}_{\swtz',\swtz} + \O(h^{2\mu})
= \O(h^{2\mu}).\qedhere
\end{align*}
\end{proof}

\begin{lemma}
\label{lem::basic-properties-WF}
If $ u \in \swtz'(\Rd)$, then the following statements hold:
\begin{enumerate}[label=(\arabic*)]
\item \label{WF::cone} 
The quasi-homogeneous wavefront set $ \WF_{\delta,\rho}^\mu(u) $ is a closed $ (\delta,\rho) $-cone.
To be precise, this means $ \theta^{\delta,\rho}_h \WF_{\delta,\rho}^\mu(u) = \WF_{\delta,\rho}^\mu(u) $ for all $ h > 0 $ where the scaling $ \theta^{\delta,\rho}_h $ is defined by~\eqref{eq:def:quasi-homogeneous-scaling}.
\item \label{WF::scaling} 
If $ \gamma > 0 $ then $ \WF_{\delta,\rho}^\mu(u) = \WF_{\delta/\gamma,\rho/\gamma}^{\mu/\gamma}(u) $.
Therefore in all situations we can restrict our discussions to the cases where either $ \delta = 1 $ or $ \rho = 1 $.
\item \label{WF::Fourier} 
For all $ (x_0,\xi_0) \in \R^{2d} $, we have $ (x_0,\xi_0) \in \WF_{\delta,\rho}^\mu(u) $ if and only if $ (\xi_0,-x_0) \in \WF_{\rho,\delta}^\mu(\hat{u}) $.
\item \label{WF::conjugate} 
For all $ (x_0,\xi_0) \in \R^{2d} $, we have $ (x_0,\xi_0) \in \WF_{\delta,\rho}^\mu(u) $ if and only if $ (x_0,-\xi_0) \in \WF_{\delta,\rho}^{\mu}(\bar{u}) $.
\item \label{WF::vanish} Denote $ \WF_{\delta,\rho}^\mu(u)^\circ = \WF_{\delta,\rho}^{\mu}(u) \backslash \mathscr{N}_{\delta,\rho}, $ where
\begin{equation}
\label{eq::def-N}
\mathscr{N}_{\delta,\rho} = \begin{cases}
\{x = 0\} \times \Rd, & \delta > 0, \rho = 0; \\
\Rd \times \{\xi = 0\}, & \delta = 0, \rho > 0; \\
\{x = 0\} \times \Rd \cup \Rd \times \{\xi = 0\}, & \delta > 0, \rho > 0.
\end{cases}
\end{equation}
If~$ u \in \H{\mu}_k $ with $ (\mu,k) \in \R^2 $ and $ a_h \in S^{-\infty}_{-\infty} $ such that 
\begin{equation}
\label{eq:condition:spectrum}
\mathscr{N}_{\delta,\rho} \cap \overline{\cup_{0<h<1} \supp a_h} = \emptyset,
\end{equation}
then $ \op_h^{\delta,\rho}(a_h) u = \O(h^{\delta k + \rho \mu})_{\Ltwo} $ and consequently $ \WF_{\delta,\rho}^{\delta k + \rho \mu}(u)^\circ = \emptyset $.
\end{enumerate}
\end{lemma}
\begin{proof}
The statements~\ref{WF::cone} and~\ref{WF::scaling} are consequences of the quasi-homogeneous scaling~\eqref{eq:def:quasi-homogeneous-scaling} we used to define the pseudodifferential operators. 
To prove~\ref{WF::Fourier}, note that if $ a_h \in S^{-\infty}_{-\infty} $ and $ \Fourier $ is the Fourier transform operator, then
\begin{equation*}
\Fourier^{-1} \op_h^{\rho,\delta}(a_h) \Fourier = \op_h^{\delta,\rho}(b_h)
\end{equation*}
where $ b_h(x,\xi) = a_h(\xi,-x) $. 
To prove~\ref{WF::conjugate}, we use $ \overline{\op(a_h)u_h} = \op(c_h) \overline{u_h} $, where $ c_h(x,\xi) = \overline{a_h(x,-\xi)}. $  
To prove~\ref{WF::vanish}, note that if $ a_h $ satisfies the condition~\eqref{eq:condition:spectrum}, then
\begin{equation*}
(\theta_{h,*}^{\delta,\rho} a_h)\jp{\xi}^{-\mu} \sharp^{0,0}_1 \jp{x}^{-k} = \O(h^{\delta k + \rho \mu})_{S^0_0}.\qedhere
\end{equation*}
\end{proof}

\section{Model equations}

\label{sec::proof-model-eq}

We prove Theorem~\ref{thm::model-eq} by combining the ideas of Nakamura~\cite{Nakamura05propagation} and simple scaling arguments. 

\subsection{Proof of of Theorem~\ref{thm::model-eq}\eqref{thm::model-eq-infinity}}

If $ a \in \Holder{1}_\loc(\R \times \R^{2d}) $ and $ \A \in \Holder{1}_\loc(\R,\Ltwo\to\Ltwo) $, then define
\begin{align*}
\Lag_t a = \pt a + \{|\xi|^\gamma,a\}, \quad
\L_t \A = \pt \A + i[|D_x|^\gamma, \A ].
\end{align*}
Here $ \{\cdot,\cdot\} $ denotes the Poisson bracket defined by $ \{f,g\} = \pxi f \cdot \px g - \px f \cdot \pxi g $.

\begin{lemma}
\label{lem::lag-derivative-at-infinity-model-eq}
If $ a_h \in \Holder{1}_\loc(\R,S^{-\infty}_{-\infty}) $ and satisfies the condition
\begin{equation*}
\overline{\cup_{0<h<1} \supp a_h} \cap \{\xi = 0 \} = \emptyset,
\end{equation*}
then there exists $ b_h \in \Linf_\loc(\R,S^{-\infty}_{-\infty}) $ with $ \supp b_h \subset \supp a_h $, such that
\begin{equation*}
\L_t \op_h^{\delta,\rho}(a_h) = \op_h^{\delta,\rho}(\Lag_t a_h) + h^{\delta+\rho} \op_h^{\delta,\rho}(b_h) + \O(h^\infty)_{\Linf_\loc(\R,\Ltwo\to\Ltwo)}.
\end{equation*}
\end{lemma}
\begin{proof}
For all $ T > 0 $, there exists $ \epsilon > 0 $ such that 
\begin{equation*}
\overline{\cup_{t\in[-T,T]} \cup_{0<h<1} \supp a_h(t,\cdot)} \cap \{|\xi| \le \epsilon \} = \emptyset.
\end{equation*}
Let $ \pi \in \Cinf(\Rd) $ such that $ 0 \le \pi \le 1 $, $ \pi(\xi) = 0 $ for $ |\xi| \le \epsilon / 3 $, $ \pi(\xi) = 1 $ for $ |\xi| \ge 2\epsilon / 3 $. Then
\begin{equation*}
i[|D_x|^\gamma ,\op_h^{\delta,\rho}(a_h)] 
= ih^{-\rho\gamma} [|h^\rho D_x|^\gamma \pi(h^\rho D_x) ,\op_h^{\delta,\rho}(a_h)] + \O(h^\infty)_{\Linf([-T,T],\Ltwo\to\Ltwo)}.
\end{equation*}
Now that $ |\xi|^\gamma \pi(\xi) \in S^{\gamma}_0 $, we conclude by Proposition~\ref{prop::quasi-homogeneous-composition} and the hypothesis $ \rho \gamma = \delta + \rho $.
\end{proof}

Assume that $ \mu = \infty $, as the proof is similar for $ \mu < \infty $. Let $ (x_0,\xi_0) \not \in \WF_{\delta,\rho}^\mu(u_0) $ with $ \xi_0 \ne 0 $. We aim to find $ a_h \in \Holder{1}_\loc(\R,S^{-\infty}_{-\infty}) $ of the asymptotic expansion 
$ a_h \sim \sum_{j\in\N} h^{j(\delta+\rho)} a_h^j $
where $ a_h^j \in \Holder{1}_\loc(\R,S^{-\infty}_{-\infty}) $, such that for all $ t \in \R $, $ a_h(t,\cdot) $ is elliptic at $ (x_0+t\gamma|\xi_0|^{\gamma-2}\xi_0,\xi_0) $, and
\begin{equation}
\label{eq:model-infinity-transport}
\L_t \op_h^{\delta,\rho}(a_h) = \O(h^\infty)_{\Linf_\loc(\R,\Ltwo\to\Ltwo)}.
\end{equation}
If such~$ a_h $ is found, we let $ \A_h(t) = \op_h^{\delta,\rho}(a_h(t)) $ and 
\begin{equation*}
v_h(t) = e^{it |D_x|^\gamma} \A_h(t) e^{-it |D_x|^\gamma} u_0,
\end{equation*}
then by a direct computation and~\eqref{eq:model-infinity-transport}, we have
\begin{equation}
\label{eq:equation-v_h-model-infinity}
\pt v_h 
= e^{it |D_x|^\gamma} \L_t \A_h e^{-it |D_x|^\gamma} u_0
= \O(h^\infty)_{\Linf_\loc(\R,\Ltwo\to\Ltwo)}.
\end{equation}
If we assume that $ \supp a_h(0) $ is sufficiently close to $ (x_0,\xi_0) $ so that 
\begin{equation*}
\overline{\cup_{h\in(0,1]} \supp a_h(0)} \subset \subset \R^{2d} \backslash \WF_{\delta,\rho}^\mu(u_0),
\end{equation*}
then by Lemma~\ref{lem::support=decay}, we have $ v_h(0) = \op_h^{\delta,\rho}(a_h(0)) u_0 = \O(h^\infty)_{L^2} $.
Therefore by~\eqref{eq:equation-v_h-model-infinity}, we have $ v_h \in \O(h^\infty)_{\Linf_\loc(\R,\Ltwo)} $ and thus $ \A_h u \in \O(h^\infty)_{\Linf_\loc(\R,\Ltwo)} $.

To construct~$ a_h $, let $ \varphi \in \Ccinf(\Rd \times (\Rd \backslash 0)) $ with $ \varphi(x_0,\xi_0) \ne 0 $, such that $ \op_h^{\delta,\rho}(\varphi) u = \O(h^\infty)_{\Ltwo}. $ Then we can construct~$ a_h $ with $ a_h|_{t=0} = \varphi $ with $ a_h^j \in \Holder{\infty}_\loc(\R,S^{-\infty}_{-\infty}) $ by solving iteratively the transportation equations
\begin{equation*}
\begin{cases}
\Lag_t a_h^0 = 0, \\ a_h^0|_{t=0} = \varphi;
\end{cases}
 \quad
\begin{cases}
\Lag_t a_h^j + b_h^{j-1} = 0, \\ a_h^j|_{t=0} = 0, \  j \ge 1,
\end{cases}
\end{equation*}
where $ b_h^{j} \in \Holder{\infty}_\loc(\R,S^{-\infty}_{-\infty}) $ satisfies, by Lemma~\ref{lem::lag-derivative-at-infinity-model-eq}, that
\begin{equation*}
\L_t \op_h^{\delta,\rho}(a^{j}_h) = \op_h^{\delta,\rho}(\Lag_t a^j_h) + h^{\delta+\rho} \op_h^{\delta,\rho}(b^j_h) + \O(h^\infty)_{\Linf_\loc(\R,\Ltwo\to\Ltwo)}.
\end{equation*}
Thus we prove Theorem~\ref{thm::model-eq}\eqref{thm::model-eq-infinity}.

\subsection{Proof of Theorem~\ref{thm::model-eq}\eqref{thm::model-eq-finity}}

Let $ \beta = \rho \gamma - (\delta + \rho) > 0 $.
For all $ h > 0 $, introduce the semiclassical time variable $ s = h^{-\beta} t $, and rewrite~\eqref{Equation::model} as
\begin{equation}
\label{eq::equation-model-semiclassical}
\ps u + i h^{\beta} |D_x|^\gamma u = 0.
\end{equation}
If $ a = a(s,x,\xi) \in \Holder{1}_\loc(\R \times \R^{2d}) $ and $ \A = \A(s) \in \Holder{1}_\loc(\R,\Ltwo\to\Ltwo) $, then define
\begin{align*}
\Lag_s a = \ps a + \{|\xi|^\gamma,a\}, \quad
\L^h_s \A = \ps \A + ih^\beta[|D_x|^\gamma,\A].
\end{align*}

\begin{lemma}
\label{lem::symbol-prop-to-infinifty-model-eq}
If $ \phi \in \Ccinf(\Rd) $ such that $ \phi \ge 0 $, $ \phi(0) > 0 $, and $ x \cdot \nabla \phi(x) \le 0 $ for all $ x \in \Rd $, and define 
\begin{equation*}
\chi(s,x,\xi) =  \phi\Big( \frac{x-s\gamma |\xi|^{\gamma-2} \xi-x_0}{1+s} \Big) \phi\Big( \frac{\xi-\xi_0}{\epsilon} \Big)
\end{equation*}
for $ s \ge 0 $, $ \epsilon > 0 $, $ (x_0,\xi_0) \in \Rd \times (\Rd \backslash 0) $, then the following statements hold when $ \epsilon $ is sufficiently small and $ |\xi_0| $ is sufficiently large:
\begin{enumerate}
\item $ \chi \in \Holder{\infty}(\R_{\ge 0},S^{-\infty}_0) $;
\item \label{eq::lag-chi->=0}  $ \Lag_s \chi \in \Holder{\infty}(\R_{\ge 0},S^{-\infty}_{-1}) $ and $ \Lag_s \chi \ge 0 $;
\item If $ t_0 > 0 $ and set $ (\tau u)(s,x,\xi) = u\big(s,\frac{s}{t_0}x,\xi\big), $ then $ \tau\chi \in \Holder{\infty}(\R_{\ge 0},S^{-\infty}_{-\infty}) $
\item If $ s $ is sufficiently large, then $ (\tau \chi)(s,\cdot) $ is elliptic at $ (t_0\gamma|\xi_0|^{\gamma-2} \xi_0,\xi_0) $.
\end{enumerate} 
\end{lemma}
\begin{proof}
Each time we differentiate~$ \chi $ with respect to~$ x $, we get a multiplicative factor~$ (1+s)^{-1} $, which is of size $ \jp{x}^{-1} $ in $ \supp \chi $ as 
\begin{equation}
\label{eq:support-goes-to-infinity}
\supp \chi \subset \{C^{-1} s \le |x| \le C s \}
\end{equation}
for some $ C > 0 $ when~$ |s| $ and $ |\xi_0| $ are sufficiently large and~$ \epsilon $ is sufficiently small.
Therefore $ \chi \in \Holder{\infty}(\R_{\ge 0},S^{-\infty}_0) $. 
Clearly $ \tau\chi(s,\cdot) $ is bounded in $ \Ccinf(\R^{2d}) $. 
We write
\begin{equation}
\label{eq::chi-tau-form}
(\tau\chi)(s,x,\xi) =  \phi\Big(\frac{x-t_0\gamma|\xi_0|^{\gamma-2}\xi_0}{t_0(1+s)/s} - \frac{\gamma|\xi|^{\gamma-2}\xi - \gamma|\xi_0|^{\gamma-2}\xi_0}{(1+s)/s} - \frac{x_0}{1+s} \Big) \phi\Big( \frac{\xi-\xi_0}{\epsilon} \Big),
\end{equation}
where $ |\xi|^{\gamma-2}\xi - |\xi_0|^{\gamma-2}\xi_0 = o(1) $ as $ \epsilon \to 0 $, whence $ \tau\chi(s,\cdot) $ elliptic at $ (t_0\gamma|\xi_0|^{\gamma-2}\xi_0,\xi_0) $ for sufficiently large~$ s $. To estimate~$ \Lag_s \chi $, we perform an explicit computation,
\begin{align*}
\ps \chi(s,x,\xi) & = - (\nabla \phi) \Big(\frac{x-s\gamma |\xi|^{\gamma-2} \xi-x_0}{1+s}\Big) \phi\Big( \frac{\xi-\xi_0}{\epsilon} \Big)  \\
& \qquad \qquad \qquad \times \frac{(x-s\gamma |\xi|^{\gamma-2} \xi-x_0)+(1+s)\gamma |\xi|^{\gamma-2} \xi}{(1+s)^2}, \\
\{|\xi|^\gamma,\chi\}(s,x,\xi) & = {\gamma |\xi|^{\gamma-2} \xi \over 1+s} \cdot (\nabla \phi) \Big(\frac{x-s\gamma |\xi|^{\gamma-2} \xi-x_0}{1+s}\Big) \phi\Big( \frac{\xi-\xi_0}{\epsilon} \Big).
\end{align*}
Therefore, 
\begin{equation*}
\Lag_s \chi(s,x,\xi) = - (\nabla \phi) \Big(\frac{x-s\gamma |\xi|^{\gamma-2} \xi-x_0}{1+s}\Big) \phi\Big(\frac{\xi-\xi_0}{\epsilon}\Big) \cdot \frac{x-s\gamma |\xi|^{\gamma-2} \xi-x_0}{(1+s)^2} \ge 0.
\end{equation*}
Note that on $ \supp \Lag_s \chi $, we have
\begin{equation*}
\frac{x-s\gamma |\xi|^{\gamma-2} \xi-x_0}{(1+s)^2} 
= \O\Big( {1+s \over (1+s)^2} \Big)
= \O\Big( {1 \over 1+s} \Big) 
= \O\Big( {1 \over \jp{x}} \Big).
\end{equation*}
So we prove similarly that $ \Lag_s \chi \in \Holder{\infty}(\R_{\ge 0},S^{-\infty}_{-1}) $.
\end{proof}

Now fix $ t_0 > 0 $ and let $ \mu = \infty $ as the other cases are similar. 
Let $ \epsilon > 0 $ be sufficiently small and let $ \{\lambda_j\}_{j \in \N} \subset [1,1+\epsilon) $ be strictly increasing. 
Choose~$ \phi $ as in Lemma~\ref{lem::symbol-prop-to-infinifty-model-eq}, and set
\begin{equation*}
\chi_j(s,x,\xi) =  \phi\Big(\frac{x-s\gamma |\xi|^{\gamma-2} \xi-x_0}{\lambda_j(1+s)}\Big) \phi\Big(\frac{\xi-\xi_0}{\lambda_j\epsilon}\Big).
\end{equation*}
Then  $ \supp \chi_j \subset \{\chi_{j+1} > 0 \} $ for all $ j\in\N $.
We aim to construct~$ a_h \in \Holder{\infty}(\R_{\ge 0},S^{-\infty}_0) $, such that the following statements hold:
\begin{enumerate}
\item \label{a_h::support} 
For all $ s\ge 0 $ and $ h \in (0,1] $, we have $ \supp a_h \subset \cup_{j \in \N} \,\supp \chi_j $.
\item \label{a_h::s=0::ellipticity} 
The symbol $ a_h|_{s=0} $ is elliptic at $ (x_0,\xi_0) $; more precisely , we have
\begin{equation*}
\big(a_h - (\zeta_h^{\delta,\rho} \chi_0) \sharp_h^{\delta,\rho} \chi_0\big)|_{s=0} = \O(h^\infty)_{S^{-\infty}_{-\infty}}.
\end{equation*}
\item \label{a_h::t=t_0::ellipticity}
For $ t_0 > 0 $, let $ \tau $ be defined as in the lemma, then $ \tau a_h \in \Holder{\infty}(\R_{\ge 0},S^{-\infty}_{-\infty}) $ and $ \tau a_h(s,\cdot) $ is elliptic at $ (t_0\gamma|\xi_0|^{\gamma-2}\xi_0,\xi_0) $ when~$ s $ is sufficiently large.
\item \label{a_h::positivity} $ \L^h_s \op_h^{\delta,\rho}(a_h) \ge \O(h^\infty)_{\Linf(\R_{\ge 0},\Ltwo\to\Ltwo)}. $
\end{enumerate}
Assume that such~$ a_h $ is found, and that 
\begin{equation*}
(t_0\gamma|\xi_0|^{\gamma-2}\xi_0,\xi_0) \not \in \WF_{\rho(\gamma-1),\rho}^\mu(u(t=t_0)).
\end{equation*} 
By~\eqref{a_h::support} and~\eqref{eq::chi-tau-form}, if we choose~$ \phi $ such that $ \supp \phi $ is sufficiently close to the origin, then for sufficiently small $ h > 0 $, we have
\begin{equation*}
\supp \theta_{1/h,*}^{\beta,0}a_h|_{s=h^{-\beta} t_0} \subset \subset \R^{2d} \backslash \WF_{\rho(\gamma-1),\rho}^{\infty}(u|_{t=t_0}).
\end{equation*}
By~\eqref{a_h::t=t_0::ellipticity}, the symbol $ \theta_{1/h,*}^{\beta,0} a_h|_{s=h^{-\beta} t_0} \in S^{-\infty}_{-\infty} $ is elliptic at $(t_0\gamma|\xi_0|^{\gamma-2}\xi_0,\xi_0)$. 
Therefore, by Lemma~\ref{lem::support=decay},
\begin{equation*}
\big( u, \op_h^{\delta,\rho}(a_h) u \big)_{\Ltwo}\big|_{s=h^{-\beta} t_0} 
= \big( u, \op_h^{\rho(\gamma-1),\rho}(\theta_{1/h,*}^{\beta,0} a_h) u \big)_{\Ltwo}\big|_{s=h^{-\beta} t_0} 
= \O(h^\infty).
\end{equation*} 
By~\eqref{eq::equation-model-semiclassical}, we have
\begin{equation*}
\frac{\d}{\ds} \big( u, \op_h^{\delta,\rho}(a_h) u \big)_{\Ltwo} 
= \big( u, \L^h_s \op_h^{\delta,\rho}(a_h) u \big)_{\Ltwo},
\end{equation*}
which implies, by~\eqref{a_h::positivity}, that
\begin{align*}
\big( u, \op_h^{\delta,\rho}(a_h) u \big)_{\Ltwo}\big|_{s=0}
& = \big( u, \op_h^{\delta,\rho}(a_h) u \big)_{\Ltwo}\big|_{s=h^{-\beta} t_0} - \int_0^{h^{-\beta} t_0} \big( u,\L^h_s \op_h^{\delta,\rho}(a_h) u \big)_{\Ltwo} \ds \\
& \le \O(h^\infty) + \O(h^{-\beta} \times h^\infty) 
= \O(h^\infty).
\end{align*}
Therefore, by~\eqref{a_h::s=0::ellipticity}, we have
\begin{equation*}
\|\op_h^{\delta,\rho}(\chi_0) u |_{s=0} \|_{\Ltwo}^2 = \big( u, \op_h^{\delta,\rho}(a_h) u \big)_{\Ltwo}\big|_{s=0}  + \O(h^\infty) = \O(h^\infty).
\end{equation*}
We conclude that $ (x_0,\xi_0) \not \in \WF_{\delta,\rho}^\infty(u_0) $.

We shall construct~$ a_h $ in the following form of asymptotic expansion
\begin{equation*}
a_h(s,x,\xi) \sim \sum_{j \in \N} h^{j(\delta+\rho)} \varphi^j(s) a_h^j(s,x,\xi),
\end{equation*}
where $ a_h^j \in \Holder{\infty}(\R_{\ge 0},S^{-\infty}_0) $, with $ \supp a^j_h \subset \supp \chi_j $, and $ \varphi^j \in P_j $, with
\begin{equation}
\label{eq::def-P_j}
P_j = \Big\{f(\ln(1+s)) : f(X) = \sum_{k=0}^j c_k X^k;  c_k \ge 0, \forall k \Big\}.
\end{equation}
The above asymptotic expansion is in the weak sense that, for some $ \epsilon' > 0 $, and all $ N \in \N $,
\begin{equation*}
a_h - \sum_{j < N} h^{j(\delta+\rho)} \varphi^j a_h^j \in \O(h^{N(\delta+\rho-\epsilon')})_{\Holder{\infty}([0,h^{-\beta}T],S^{-\infty}_0)}.
\end{equation*}
The following properties for functions in $ \cup_{j\in\N} P_j $ will be used in the construction of $ a_h $.
\begin{lemma}
If $ \psi \in P_j $ for some $ j \in \N $, then $ \psi $ is smooth and non-negative on $ [0,+\infty) $ and
\begin{equation*}
((1+s) \ps)^{-1} \psi(s) \bydef \int_0^s (1+\sigma)^{-1}\psi(\sigma) \d \sigma \in P_{j+1}. 
\end{equation*}
\end{lemma}
\begin{proof}
The function~$ \psi $ is smooth because it is the composition between a polynomial and the smooth function $ s \mapsto \ln(1+s) $.
The non-negativity of~$ \psi $ is the consequence of the non-negativity of the coefficients~$ c_k $ in~\eqref{eq::def-P_j} and the fact that $ \ln(1+s) \ge 0 $ when $ s \ge 0 $.
To prove that $ ((1+s) \ps)^{-1} \psi \in P_{j+1} $, note that for all $ n \in \N $, we have
\begin{equation*}
((1+s) \ps)^{-1} (\ln(1+\cdot))^n = (n+1)^{-1} (\ln(1+\cdot))^{n+1}.
\end{equation*} 
The claim follows by the linearity of the operator $ ((1+s) \ps)^{-1} $.
\end{proof}

To construct~$ a_h $, we begin by setting $ \varphi^0 \equiv 1 $ and choosing~$ a_h^0 $ satisfying
\begin{align*}
a_h^0 & - (\zeta_h^{\delta,\rho} \chi^{}_0) \sharp_h^{\delta,\rho} \chi^{}_0  = \O(h^\infty)_{\Holder{\infty}(\R_{\ge 0},S^{-\infty}_0)}, \\
\big(a_h^0 & - (\zeta_h^{\delta,\rho} \chi^{}_0) \sharp_h^{\delta,\rho} \chi^{}_0\big)|_{s=0}  = \O(h^\infty)_{S^{-\infty}_{-\infty}}.
\end{align*}
By the definition of~$ \beta $ and Propositions~\ref{prop::quasi-homogeneous-composition} and ~\ref{prop::adjoint-quasi-homogeneous}, there exists $ r_h^0 \in \Linf(\R_{\ge 0},S^{-\infty}_{-1}) $ with $ \supp r_h^0 \subset \supp \chi_0 $ such that
\begin{equation}
\label{eq:lagrange-ah0}
\L^h_s \op_h^{\delta,\rho}(a_h^0) = 2 \op_h^{\delta,\rho}(\chi_0 \Lag_s \chi_0) + h^{\delta+\rho} \op_h^{\delta,\rho}(r_h^0) + \O(h^{\infty})_{\Linf(\R_{\ge 0},\Ltwo\to\Ltwo)}.
\end{equation}
By~\eqref{eq:support-goes-to-infinity}, we have $ \jp{s} r_h^0 \in \Linf(\R_{\ge 0},S^{-\infty}_0) $ and similarly 
\begin{equation}
\label{eq:decay-ah0}
\jp{s} \chi_0 \Lag_s \chi_0 \in \Linf(\R_{\ge 0},S^{-\infty}_0).
\end{equation}
By Lemma~\ref{lem::symbol-prop-to-infinifty-model-eq}, we have 
\begin{equation}
\label{eq:positivity-ah0}
\chi_0 \Lag_s \chi_0 \ge 0
\end{equation}
Recall that, by the sharp G{\aa}rding inequality (Proposition~\ref{prop::garding-quasi-homogeneous}), if a symbol $ p_h \in S^0_0 $ satisfies $ p_h \ge 0 $, then $ \op_h^{0,1}(p_h) \gtrsim -h $.
By~\eqref{eq:scaling-transform}, we deduce that $ \op_h^{0,1}(p_h) \gtrsim -h^{\delta+\rho} $.
To be precise, this means there exists $ C > 0 $ which only depends on a finite number of seminorms defined by~\eqref{eq:symbol-seminorm-def}, such that for all $ u \in L^2 $,
\begin{equation*}
\jp{u,\op_h^{\delta,\rho}(p_h)u}_{L^2} \ge -Ch^{\delta+\rho}\|u\|_{L^2}^2.
\end{equation*}
Take $ c_h^0 \in \Linf(\R_{\ge 0},S^{-\infty}_0) $ such that 
\begin{equation*}
\supp a_h^0 \subset \subset \{c_h^0=1\} \subset \supp c_h^0 \subset \{\chi_1 > 0\}.
\end{equation*} 
By~\eqref{eq:lagrange-ah0} and~\eqref{eq:positivity-ah0}, for all $ u \in L^2 $, we have
\begin{align*}
\jp{u,\L^h_s \op_h^{\delta,\rho}(a_h^0) u}_{L^2}
& = \jp{\op_h^{\delta,\rho}(c_h) u,\L^h_s \op_h^{\delta,\rho}(a_h^0) \op_h^{\delta,\rho}(c_h)u}_{L^2} + \O(h^\infty) \|u\|_{L^2}^2 \\
& \ge -C\jp{s}^{-1}h^{\delta+\rho} \|\op_h^{\delta,\rho}(c_h) u\|_{L^2}^2 + \O(h^\infty),
\end{align*}
where the factor $ \jp{s}^{-1} $ comes from the estimate~\eqref{eq:decay-ah0}.
By the symbolic calculus, there exists $ b_h \in \Linf(\R_{\ge 0},S^{-\infty}_0) $ such that 
\begin{equation*}
\op_h^{\delta,\rho}(b_h) - C \op_h^{\delta,\rho}(c_h)^* \op_h^{\delta,\rho}(c_h) = \O(h^\infty)_{\Linf(\R_{\ge 0},\Ltwo\to\Ltwo)}
\end{equation*}
and $ \supp b_h \subset \supp c_h $. 
Therefore,
\begin{equation}
\label{eq:from-garding-to-symbol-construction}
\begin{split}
\L^h_s \op_h^{\delta,\rho}(a_h^0) 
& \ge - C\jp{s}^{-1} h^{\delta+\rho}  \op_h^{\delta,\rho}(c_h)^* \op_h^{\delta,\rho}(c_h) + \O(h^\infty)_{\Linf(\R_{\ge 0},\Ltwo\to\Ltwo)} \\
& \ge -\jp{s}^{-1} h^{\delta+\rho}  \op_h^{\delta,\rho}(b_h^0) + \O(h^\infty)_{\Linf(\R_{\ge 0},\Ltwo\to\Ltwo)}.
\end{split}
\end{equation}
Suppose that for some $ \ell \ge 1 $, we can find $ \varphi^j \in P_j $, $ a_h^j $ for $ j = 0,\ldots,\ell-1 $ and $ \psi^{\ell-1} \in P_{\ell-1} $, $ b_h^{\ell-1} \in \Linf(\R_{\ge 0},S^{-\infty}_0) $ with $ \supp b_h^{\ell-1} \subset \{\chi_{\ell}>0\} $ such that
\begin{align}
\label{eq::hypothesis-of-induction-propagation-to-infinity-model-eq}
\L^h_s \op_h^{\delta,\rho}\Big(\sum_{j=0}^{\ell-1} h^{j(\delta+\rho)} \varphi^j a_h^j \Big) 
\ge -\jp{s}^{-1}\psi^{\ell-1} h^{\ell(\delta+\rho)} \op_h^{\delta,\rho}(b_h^{\ell-1}) + \O(h^\infty)_{\Linf(\R_{\ge 0},\Ltwo\to\Ltwo)}.
\end{align}
If we choose $ B_\ell > 0 $ sufficiently large and set $ \varphi^\ell = ((1+s)\ps)^{-1} \psi^{\ell-1} $ and $ a_h^\ell = B_\ell \chi_\ell $, then by a direct calculation, we have
\begin{equation*}
\begin{split}
\Lag_s (\varphi^\ell a_h^\ell)
& = B_\ell (1+s)^{-1} \psi^{\ell-1} \chi_\ell + B_\ell \varphi^\ell \Lag_s \chi_\ell  \\
& \ge B_\ell (1+s)^{-1}\psi^{\ell-1} \chi_\ell
\ge \jp{s}^{-1} \psi^{\ell-1} b_h^{\ell-1}.
\end{split}
\end{equation*}
Observe that 
\begin{equation*}
\Lag_s (\varphi^\ell a_h^\ell) = \O(\jp{s}^{-1}(\psi^{\ell-1}+\varphi^\ell))_{S^{-\infty}_0}, \quad \jp{s}^{-1}\psi^{\ell-1} b_h^{\ell-1} = \O(\jp{s}^{-1}\psi^{\ell-1})_{S^{-\infty}_0}. 
\end{equation*}
Similarly as above, applying the sharp G{\aa}rding inequality to the symbol 
\begin{equation*}
\Lag_s (\varphi^\ell a_h^\ell) - \jp{s}^{-1} \psi^{\ell-1} b_h^{\ell-1} 
= \O\big(\jp{s}^{-1} (\varphi^\ell + \psi^{\ell-1})\big)_{S^{-\infty}_0},
\end{equation*}
we can find $ b_h^{\ell} \in \Linf(\R_{\ge 0},S^{-\infty}_0) $ with $ \supp b_h^{\ell} \subset  \{\chi_{\ell+1} > 0\} $ such that
\begin{align}
\label{eq::current-step-of-induction-propagation-to-infinity-model-eq}
\L^h_s \op_h^{\delta,\rho}(\varphi^\ell a_h^\ell) - \jp{s}^{-1}\psi^{\ell-1} \op_h^{\delta,\rho}(b_h^{\ell-1}) 
\ge -\jp{s}^{-1} \psi^\ell h^{\delta+\rho}  \op_h^{\delta,\rho}(b_h^\ell)  + \O(h^\infty)_{\Ltwo\to\Ltwo}, 
\end{align}
with $ \psi^\ell = \psi^{\ell-1}+\varphi^\ell \in P_\ell $.
Summing up~\eqref{eq::hypothesis-of-induction-propagation-to-infinity-model-eq} and $ h^{\ell(\delta+\rho)} \times $~\eqref{eq::current-step-of-induction-propagation-to-infinity-model-eq}, we close the induction procedure. 

Finally we conclude the asymptotic expansion by observing that, for all $ \epsilon' > 0 $, we have
\begin{equation*}
\|\varphi^\ell\|_{\Linf([0,h^{-\beta}T])} = \O(|\log h|^{\ell}) = \O(h^{-\epsilon'\ell}).
\end{equation*}
Thus we prove Theorem~\ref{thm::model-eq}~\eqref{thm::model-eq-finity}.

\section{Paradifferential calculus}

\label{sec::paradiff-calculus}

In this section, we develop a paradifferential calculus on weighted Sobolev spaces and a semiclassical paradifferential calculus.

\subsection{Classical paradifferential calculus}

\label{sec::paradiff-classical-bony}

We recall some classical results of the paradifferential calculus. We refer to the original work of Bony~\cite{Bony79calcul} and the books \cite{Hormander97lecture,Metivier08paradiff,BCD11Fourier:PDE}. 
The results and proofs below are mainly based on~\cite{Metivier08paradiff}, so we shall only sketch them. 
In the meanwhile, we shall also make some refinements regarding the estimates of the remainder terms, for the sake of the semiclassical paradifferential calculus that will be developed later.

\subsubsection{Symbol classes and paradifferential operators}

\begin{definition}
For $ m \in \R $, $ r \ge 0 $, let $ \Gamma^{m,r} $ be the space of all $ a(x,\xi) \in L^\infty_\loc\big(\Rd \times (\Rd \backslash 0)\big) $ such that:
\begin{enumerate}
\item For all $ x \to \Rd $, the function $ \xi \mapsto a(x,\xi) $ is smooth; and
\item For all $ \alpha \in \N^d $, there exists $ C_\alpha > 0 $ such that for all $ \xi \in \R^d $ with $ |\xi| \ge 1/2 $, we have
\begin{equation*}
\|\pxi^\alpha a(\cdot,\xi)\|_{\Holder{r}} \le C_\alpha \jp{\xi}^{m-|\alpha|}.
\end{equation*}
\end{enumerate}
If $ a \in \Gamma^{m,r} $, then we denote for all $ n \ge 0 $ the seminorm
\begin{equation*}
\M{m,r}{n}{a} = \sup_{|\alpha|\le n} \sup_{|\xi|\ge 1/2} \jp{\xi}^{|\alpha|-m} \|\pxi^\alpha a(\cdot,\xi)\|_{\Holder{r}}.
\end{equation*}
We also denote $ \M{m,r}{}{a} = \M{m,r}{\tilde{d}+r}{a} $, where $ \tilde{d} = 1+[d/2] $.
\end{definition}

\begin{definition}
A pair of non-negative functions $ (\chi,\pi) \in \Cinf(\R^{2d}\backslash 0) \times \Cinf(\Rd) $ is called admissible if the following conditions are satisfied:
\begin{enumerate}
\item The function $ 1-\pi $ is a cutoff function of the origin. To be precise, if $ |\eta| \ge 1 $, then $ \pi(\eta) = 1 $, and if $ |\eta| \le 1/2 $, then $ \pi(\eta) = 0 $.
\item The function~$ \chi $ is an even and homogeneous of degree~$ 0 $, and there exists $ \epsilon_1, \epsilon_2 \in (0,1)$ with $ \epsilon_1 < \epsilon_2 $, such that
\begin{equation}
\label{eq::admissible-function-def}
\begin{cases}
\chi(\theta,\eta) = 1, & |\theta| \le \epsilon_1 |\eta|, \\
\chi(\theta,\eta) = 0, & |\theta| \ge \epsilon_2 |\eta|.
\end{cases}
\end{equation}
\end{enumerate} 
\end{definition}

\begin{definition}
If $ m \in \R $ and $ a \in \Gamma^{m,0} $, then the paradifferential operator~$ \T{a} $ is defined by
\begin{equation}
\label{eq::definition-paradiff-T}
\widehat{\T{a} u}(\xi) = (2\pi)^{-d} \int_{\Rd} \chi(\xi-\eta,\eta) \pi(\eta) \hat{a}(\xi-\eta,\eta) \hat{u}(\eta) \d\eta,
\end{equation}
where $ (\chi,\pi) $ is admissible and 
$ \hat{a}(\theta,\xi) = \int e^{-ix\cdot\theta} a(x,\xi) \dx. $
In other words $ \T{a} = \op(\sigma_a) $ where
\begin{equation}
\label{eq:def-sigma_a}
\sigma_a(\cdot,\xi) = \pi(\xi) \chi(D_x,\xi) a(\cdot,\xi).
\end{equation}
\end{definition}

\begin{proposition}
\label{prop:paradiff-operator-norm}
If $ m \in \R $ and $ a \in \Gamma^{m,0} $, then $ \T{a} = \O(\M{m,0}{}{a})_{\OP{m}{0}}. $
\end{proposition}

\begin{remark}
\label{rmk:Sm11}
A symbol $ p $ satisfies the spectral condition if there exists $ \epsilon \in (0,1) $ such that
\begin{equation*}
\supp \hat{p} \subset \{(\eta,\xi) \in \R^{2d}:|\eta| \ge \epsilon \jp{\xi}\}.
\end{equation*}
By~\cite{Metivier08paradiff}, if $ a \in \Gamma^{m,0} $, then 
$ \sigma_a \in \Gamma^{m,0} $
and satisfies the spectral condition. 
The above Proposition~\ref{prop:paradiff-operator-norm} is in fact a consequence of the following estimate~\eqref{eq:estimate-para-symbol} and the mapping property: if $ p \in \Gamma^{m,0} $ satisfies the spectral condition, then $ \op(p) $ defines a bounded operator from $ H^{\mu+m} \to H^{\mu} $ for all $ \mu \in \R $.
\end{remark}

Note the the definition~\eqref{eq::definition-paradiff-T} depends on the choice of admissible pairs of functions.
The following lemma and corollary shows that if we change the admissible pair, then the error term is regularizing.

\begin{lemma}
\label{lem::change-of-admissible-function}
If $ m \in \R $, $ r \ge 0 $ and $ a \in \Gamma^{m,r} $, then for all $ n \ge 0 $, we have
\begin{equation}
\label{eq:estimate-para-symbol}
\M{m,r}{n}{ \sigma_a} \lesssim \M{m,r}{n}{a}.
\end{equation}
If in addition $ r \in \N $, then for all $ \beta \in \N^d $ with $ |\beta| \le r $, we have 
\begin{equation}
\label{eq::estimate-un-paralinearization}
\M{m-r+|\beta|,0}{n}{\px^\beta (\sigma_a - a\pi)} \lesssim \M{m,0}{n}{\nabla^r a}.
\end{equation}
\end{lemma}
\begin{proof}
The first statement is proven in \cite{Metivier08paradiff}. We only prove the second statement. 
We shall only prove the case where $ \beta = 0 $ for the rest is similar. 
By~\cite{Metivier08paradiff}, we have
\begin{align*}
(\sigma_a - a\pi)(x,\xi) = \pi(\xi) \int \rho(x,y,\xi) \Phi(y,\xi) \dy,
\end{align*}
for all $ (x,\xi) \in \Rd \times (\Rd\backslash 0) $, where $ \Phi(\cdot,\xi) = \Fourier^{-1} \chi(\cdot,\xi) $ and
\begin{equation*}
\rho(x,y,\xi) = \sum_{|\gamma| = r} \frac{(-y)^\gamma}{\gamma !} \int_0^1 r(1-t)^{r-1}  \px^\gamma a(x-ty,\xi)  \dt.
\end{equation*} 
Therefore, if $ |\xi| \ge 1/2 $ and $ |\alpha| \le n $, then
\begin{equation}
\label{eq:estimate-rho}
\|\pxi^\alpha \rho(\cdot,y,\xi) \|_{\Linf} 
\lesssim |y|^r \|\pxi^\alpha \nabla^r a(\cdot,\xi)\|_{\Linf}
\lesssim |y|^r |\xi|^{m-|\alpha|} \M{m,0}{n}{\nabla^r a}.
\end{equation}
Note that the admissibility of~$ (\pi,\chi) $ implies that for all $ \alpha,\beta \in \N $, there exists $ C_{\alpha,\beta} > 0 $ such that for all $ (x,\xi) \in \R^{2d} $, we have
$ |x^{\beta} \pxi^{\alpha} \Phi(x,\xi)| \le C_{\alpha,\beta} \jp{\xi}^{d-|\alpha|-|\beta|}. $
Therefore, for all $ \alpha \in \N $ and all $ \xi \in \Rd $, there exists $ C_\alpha > 0 $ such that
\begin{equation}
\label{eq:estimate-Phi}
\|\pxi^\alpha \Phi(\cdot,\xi)\|_{L^1} \le C_\alpha \jp{\xi}^{-|\alpha|}.
\end{equation}
We conclude~\eqref{eq::estimate-un-paralinearization} by estimates~\eqref{eq:estimate-rho} and~\eqref{eq:estimate-Phi}.
\end{proof}

\begin{corollary}
\label{cor:change-admissible-function}
Let $ a \in \Gamma^{m,r} $ with $ m \in \R $ and $ r \in \N $. 
Let~$ (\chi,\pi) $ and~$ (\chi',\pi') $ be admissible. 
Denote by~$ \T{a} $ and~$ \T{a}' $ the paradifferential operators respectively defined by these two admissible pairs. 
Then 
\begin{equation*}
\T{a} - \T{a}' = \O(\M{m,0}{}{\nabla^r a})_{\OP{m-r}{0}} + \O(\M{m,r}{}{a})_{\OP{-\infty}{0}}.
\end{equation*}
If in addition $ a \pi  =  a \pi' = a $, then
\begin{equation*}
\T{a} - \T{a}' = \O(\M{m,0}{}{\nabla^r a})_{\OP{m-r}{0}}.
\end{equation*}
\end{corollary}
\begin{proof}
Let $ \T{a}'' $ be the paradifferential operator defined with respect to $ (\chi',\pi) $, then by Lemma~\ref{lem::change-of-admissible-function}, 
$ \T{a} - \T{a}'' = \O(\M{m,0}{}{\nabla^r a})_{\OP{m-r}{0}}. $
Note that $ \T{a}'' - \T{a}' $ is a composition with a paradifferential operator with a smoothing operator $ \pi(D_x) - \pi'(D_x) $, which implies
$ \T{a}'' - \T{a}' = \O(\M{m,r}{}{a})_{\OP{-\infty}{0}}. $
This term vanishes if $ a \pi  =  a \pi' = a $.
\end{proof}

\begin{corollary}
\label{cor::paralinearization-smooth}
Let $ \psi \in C_b^\infty(\Rd) $, then $ \T{\psi} - \psi \in \OP{-\infty}{0} $.
\end{corollary}
\begin{proof}
This is a consequence of~\eqref{eq::estimate-un-paralinearization} and the Calder\'{o}n--Vaillancourt theorem.
\end{proof}

\subsubsection{Symbolic calculus and paralinearization}

\begin{proposition}
\label{prop::composition-T}
If $ a \in \Gamma^{m,r} $ and $ b \in \Gamma^{m',r} $ where $ r \in \N $, $ m \in \R $ and $ m' \in \R $, then
\begin{align*}
\T{a} \T{b} - \T{a \sharp b}
& = \O\big(\M{m,r}{}{a} \M{m',0}{}{\nabla^r b} + \M{m,0}{}{\nabla^r a} \M{m',r}{}{b}\big)_{\OP{m+m'-r}{0}} \\
& \quad \quad + \O\big(\M{m,r}{}{a} \M{m',r}{}{b}\big)_{\OP{-\infty}{0}},
\end{align*}
where the symbol $ a \sharp b = a \sharp_{1,r}^{0,0} b $ is defined by~\eqref{eq:def-symbolic-composition-finite}.
If in addition $ a\pi = a $ and $ b \pi = b $, then,
\begin{equation*}
\T{a} \T{b} - \T{a \sharp b}
= \O\big(\M{m,r}{}{a} \M{m',0}{}{\nabla^r b} + \M{m,0}{}{\nabla^r a} \M{m',r}{}{b}\big)_{\OP{m+m'-r}{0}}.
\end{equation*}
\end{proposition}

\begin{proof}
By Corollary~\ref{cor:change-admissible-function}, we may choose an admissible pair $ (\pi,\chi) $ to define paradifferential operators while assuming that $ \epsilon_2 < 1/4 $. 
We shall only prove the case where $ a \pi = a $ and $ b \pi = b $ as the general case follows easily.
The following proof follows~\cite{Metivier08paradiff}.
Decompose
$ \T{a} \T{b} - \T{a \sharp b} = \mathrm{(I)} + \mathrm{(II)}, $
where 
\begin{equation*}
\mathrm{(I)} = \op(\sigma_a) \op(\sigma_b) - \op(\sigma_a \sharp \sigma_b),  \quad
\mathrm{(II)} = \op(\sigma_a \sharp \sigma_b) - \op(\sigma_{a\sharp b}).
\end{equation*}
Write $ \op(\sigma_a) \op(\sigma_b) = \op(\sigma), $ where
\begin{equation*}
\sigma(x,\xi) 
= \frac{1}{(2\pi)^d} \iint e^{i(x-y)\cdot\eta} \sigma_a(x,\xi+\eta) \theta(\eta,\xi) \sigma_b(y,\xi) \dy \d\eta.
\end{equation*}
Here $ \theta \in \Cinf(\R^{2d}\backslash 0) $ satisfies that $ (\theta,\pi) $ is admissible and $ \theta \chi = \chi $. 
By Taylor's formula, we decompose 
\begin{equation*}
\sigma_a(x,\xi+\eta) = \sum_{|\alpha| < r} {1\over \alpha!} \pxi^\alpha \sigma(x,\xi)\eta^\alpha + \sum_{|\alpha| = r} \rho_\alpha(x,\xi,\eta)\eta^\alpha
\end{equation*}
where the functions $ \rho_\alpha $ depend on $ \nabla_\xi^r \sigma_a $.
Then write
$ \sigma = \sigma_a \sharp \sigma_b + \sum_{|\alpha|=r} q_\alpha $
where
\begin{align*}
q_\alpha(x,\xi) 
& = \int R_\alpha(x,x-y,\xi) (D_x^\alpha \sigma_b)(y,\xi) \dy, \\
R_\alpha(x,y,\xi) 
& = (2\pi)^{-2} \int e^{iy\eta} \rho_\alpha(x,y,\eta) \theta(\eta,\xi) \mathrm{d}\eta.
\end{align*}
By the same estimate in~\cite{Metivier08paradiff},
\begin{equation*}
\|\pxi^\beta R_\alpha(x,\cdot,\xi)\|_{\Lone} \lesssim \M{m,r}{}{a} \jp{\xi}^{m-r-|\beta|}.
\end{equation*}
Using $ D_x^\alpha \sigma_b = \sigma_{D_x^\alpha b} $, we verify that
\begin{equation*}
\|\pxi^\beta q_\alpha(\cdot,\xi)\|_{\Linf} 
\lesssim \M{m,r}{}{a} \M{m',0}{}{\nabla^r b} \jp{\xi}^{m+m'-r-|\beta|},
\end{equation*}
and consequently, by Remark~\ref{rmk:Sm11},
\begin{equation*}
\| \mathrm{(I)} \|_{\H{s}\to\H{s-m-m'+r}} 
\lesssim \sum_{|\alpha|=r} \M{m+m'-r,0}{}{q_\alpha} 
\lesssim \M{m,r}{}{a} \M{m',0}{}{\nabla^r b}.
\end{equation*}

To Estimate (II), for all $ |\alpha| < r $, decompose
$ \pxi^\alpha \sigma_a D^\alpha_x \sigma_b - \sigma_{\pxi^\alpha a D^\alpha_x b}
=  (\mathrm{i}) + (\mathrm{ii}) + (\mathrm{iii}), $
where 
\begin{equation*}
(\mathrm{i}) = \pxi^\alpha (\sigma_a - a) D^\alpha_x \sigma_b, \quad
(\mathrm{ii}) = \pxi^\alpha a D^\alpha_x (\sigma_b - b), \quad
(\mathrm{iii}) = \pxi^\alpha a D^\alpha_x b - \sigma_{\pxi^\alpha a D^\alpha_x b}. 
\end{equation*}
By Lemma~\ref{lem::change-of-admissible-function}, Leibniz's rule and interpolation,
\begin{align*}
\M{m+m'-r,0}{}{\mathrm{i}} 
& \lesssim \M{m-r,0}{}{\sigma_a-a} \M{m',0}{}{D^\alpha_x \sigma_b} 
\lesssim \M{m,0}{}{\nabla^r a} \M{m',r}{}{b}, \\
\M{m+m'-r,0}{}{\mathrm{ii}}
& \lesssim \M{m,r}{}{a} \M{m'-r+|\alpha|,0}{}{D^\alpha_x (\sigma_b - b\pi)}
\lesssim \M{m,r}{}{a} \M{m',0}{}{\nabla^r b},\\
\M{m+m'-r,0}{}{\mathrm{iii}}
& \lesssim \M{m+m'-|\alpha|,0}{}{\nabla^{r-|\alpha|}(\pxi^\alpha a D^\alpha_x b)} \\
& \lesssim \M{m-|\alpha|,0}{}{\nabla^r\pxi^\alpha a} \M{m',0}{}{b} + \M{m-|\alpha|,0}{}{\pxi^\alpha a} \M{m',0}{}{\nabla^r b} \\
& \lesssim \M{m,0}{}{\nabla^r a} \M{m',r}{}{b} + \M{m,r}{}{a} \M{m',0}{}{\nabla^r b}.
\end{align*}
By Remark~\ref{rmk:Sm11}, these estimates imply that,
\begin{equation*}
(\mathrm{II})
= \O\big(\M{m,r}{}{a} \M{m',0}{}{\nabla^r b} + \M{m,0}{}{\nabla^r a} \M{m',r}{}{b}\big)_{\OP{m+m'-r}{0}}.
\end{equation*}
The proposition follows.
\end{proof}

\begin{proposition}
\label{prop::adjoint-T}
Let $ a \in \Gamma^{m,r}_{} $ with $ r \in \N $ and $ m \in \R $, then
\begin{equation*}
\T{a}^* - \T{a^*} = \O(\M{m,0}{}{\nabla^r a})_{\OP{m-r}{0}} + \O(\M{m,r}{}{a} )_{\OP{-\infty}{0}},
\end{equation*}
where the symbol $ a^* = \zeta_{1,r}^{0,0} a $ is defined by~\eqref{eq:def-symbolic-adjoint-finite}.
If in addition $ a \pi = a $, then 
\begin{equation*}
\T{a}^* - \T{a^*} = \O(\M{m,0}{}{\nabla^r a})_{\OP{m-r}{0}}.
\end{equation*}
\end{proposition}
\begin{proof}
Similarly as in the proposition for the composition, we shall only prove the case where $ a\pi = a $.
Let~$ (\theta,\pi) $ be admissible such that $ \theta\chi=\chi $, then $ \T{a}^* = \op(\sigma_a^*) $ with
\begin{equation*}
\sigma_a^*(x,\xi) 
= (2\pi)^{-d} \int e^{-iy\cdot\eta} \overline{\sigma_a}(x+y,\xi+\eta) \d\eta\dy 
= a^*(x,\xi) + \sum_{|\alpha|=r} r_\alpha(x,\xi), 
\end{equation*}
where by Taylor's formula,
\begin{equation*}
r_\alpha(x,\xi) = \frac{2\pi}{\alpha!}  \iiint_{\R^{2d}\times[0,1]} r(1-t)^{r-1} e^{-iy\cdot\eta} D_x^\alpha \pxi^\alpha \overline{\sigma_a}(x,\xi+t\eta) \theta(\eta,\xi) \dt\d\eta\dy,
\end{equation*}
The term $ D_x^\alpha \pxi^\alpha \overline{\sigma_a}(x,\xi+t\eta) $ in the integral and the analysis in \cite{Metivier08paradiff} imply that
\begin{equation*}
\M{m-r,0}{}{\sigma_{a}^*-\sigma_{a^*}} 
\le \sum_{|\alpha|=r} \M{m-r,0}{}{r_\alpha} + \M{m-r,0}{}{a^*-\sigma_{a^*}}  
\lesssim \M{m,0}{}{\nabla^r a}.
\end{equation*}
The proposition follows by Remark~\ref{rmk:Sm11}.
\end{proof}

Recall the following results of paralinearization. See e.g., \cite{Metivier08paradiff}.

\begin{proposition}
\label{prop::paraproduct-T}
If $ a \in \H{\alpha} $ and $ b \in \H{\beta} $ with $ \alpha > d/2 $ and $ \beta > d/2 $, then
\begin{equation*}
\|ab - \T{a} b - \T{b} a\|_{\H{\alpha+\beta-d/2}} \lesssim \|a\|_{\H{\alpha}} \|b\|_{\H{\beta}}.
\end{equation*}
\end{proposition}

\begin{proposition}
\label{prop::paralinearization-T}
If $ F \in \Cinf(\R) $ with $ F(0) = 0 $, then for all $ \mu > d/2 $, there exists a monotonically increasing function $ C : \R_{\ge 0} \to \R_{\ge 0} $, such that for all $ u \in \H{\mu} $, we have
\begin{equation*}
\|F(u)\|_{\H{\mu}} + \|F(u) - \T{F'(u)} u\|_{\H{2\mu-d/2}} \le C(\|u\|_{\H{s}}) \|u\|_{\H{\mu}}.
\end{equation*}
\end{proposition}

\subsection{Dyadic paradifferential calculus}

Now we develop the theory of paradifferential calculus with weighted symbols on weighted Sobolev spaces via a dyadic decomposition of the space.

\subsubsection{Weighted symbol classes and dyadic paradifferential operators}

We define a family of symbol classes which take into consideration the spacial decay of symbols.

\begin{definition}
If $ r \in \N $, $ k \in \R $, and $ \delta \in [0,1] $, then  $ \Holder{r}_{k,\delta} $ is the set of all $ u \in \swtz'(\Rd) $ such that 
\begin{equation*}
\|u\|_{\Holder{r}_{k,\delta}} = \sum_{|\alpha|\le r} \|\jp{x}^{-k+\delta|\alpha|} \px^\alpha u\|_{\Linf}<\infty.
\end{equation*}
\end{definition}

\begin{definition}
If $ m,k \in \R $, $ r \in \N $ and $ \delta \in [0,1] $, then $ \Gamma^{m,r}_{k,\delta} $ is the set of all $ a(x,\xi) \in \Linf_\loc( \Rd \times (\Rd \backslash 0)) $ such that
\begin{enumerate}
\item for all $ x \in \Rd $, the function $ \xi \mapsto a(x,\xi) $ is smooth; and
\item for all $ \alpha \in \N^d $, there exists $ \exists C_\alpha > 0 $, such that
\begin{equation*}
\|\pxi^\alpha a(\cdot,\xi)\|_{\Holder{r}_{k,\delta}} \le C_\alpha \jp{\xi}^{m-|\alpha|}, \quad \forall |\xi| \ge 1/2.
\end{equation*}
\end{enumerate}
Moreover, we denote
\begin{equation*}
\M{m,r}{k,\delta}{a} 
= \sup_{|\alpha|\le r+\tilde{d}} \sup_{|\xi|\ge 1/2} \jp{\xi}^{|\alpha|-m} \|\pxi^\alpha a(\cdot,\xi)\|_{\Holder{r}_{k,\delta}}.
\end{equation*}
Let 
$ \Gamma^{-\infty,r}_{k,\delta} = \bigcap_{m \in \R} \Gamma^{m,r}_{k,\delta}, $
$ \Gamma^{m,r}_{-\infty,\delta} = \bigcap_{k\in\R} \Gamma^{m,r}_{k,\delta}. $
Then for $ (m,k) \in (\R \cup \{-\infty\})^2 $, define
\begin{equation*}
\Sigma^{m,r}_{k,\delta} 
= \sum_{0 \le j \le r} \Gamma^{m-j,r-j}_{k-\delta j,\delta}.
\end{equation*}
We say that $ a_h = \sum_{0 \le j \le r} h^j a_h^j \in {}_h\Sigma^{m,r}_{k,\delta}  $ if
\begin{equation*}
\sup_{0<h<1} \sum_{0 \le j \le r} \M{m-j,r-j}{k-\delta j,\delta}{a_h^j} < \infty.
\end{equation*}
We shall denote $ \Sigma^{m,r}  = \Sigma^{m,r}_{0,0}  $, $ {}_h\Sigma^{m,r}  = {}_h\Sigma^{m,r}_{0,0}  $.
\end{definition}
We are mostly interested in the cases where $ \delta \in \{0,1\} $.
Note that $ \Holder{r}_{k,0} = \jp{x}^{k} \Holder{r} $ and thus $ \Gamma^{m,r}_{k,0} = \jp{x}^{k} \Gamma^{m,r} $; whereas $ \Gamma^{m,r}_{k,1} $ is a natural extension of $ S^m_k $ to symbols of finite regularities.
We will encounter symbols defined by solutions of the water wave system and thus have coefficients in weighted Sobolev spaces.
We need the following lemma.
\begin{lemma}
\label{lem:weighted-symbol-injection}
If $ u \in \HC{\mu,\delta}_{k} $ where $ \mu \ge \tilde{d} $, $ k \in \N $ and $ \delta \in (0,1] $, then for all $ \alpha \in \N^d $ with $ |\alpha| \le \min\{(\mu -\tilde{d})/(1+\delta),k\} $, we have 
$ \jp{x}^{|\alpha|} \partial_x^\alpha u \in \Linf $ and consequently we have the inclusion
\begin{equation*}
\HC{\mu,\delta}_{k} \subset \Holder{\min\{[(\mu -\tilde{d})/(1+\delta)],k\}}_{0,1} \cap \jp{x}^{-\min\{[\mu-\tilde{d}]/\delta,k\}} L^\infty.
\end{equation*}
In particular $ \HC{\mu,1/2}_{k} \subset \Holder{\min\{[2(\mu -\tilde{d})/3],k\}}_{0,1} \cap \jp{x}^{-\min\{k,2[\mu-\tilde{d}]\}} L^\infty $.
\end{lemma}
\begin{proof}
The lemma follows directly from the Sobolev injection:
\begin{align*}
\|\jp{x}^{|\alpha|} \partial_x^\alpha u\|_{\Linf}
& \lesssim \|u\|_{\Holder{|\alpha|}_{-|\alpha|}}
\lesssim \|u\|_{H^{|\alpha|+\tilde{d}}_{|\alpha|}}
\lesssim \|u\|_{\HC{\mu,\delta}_k}, \\
\|\jp{x}^{n} u\|_{\Linf}
& \lesssim \|u\|_{W^{0,\infty}_{-n,0}}
\lesssim \|u\|_{H^{\tilde{d}}_{n}}
\lesssim \|u\|_{\HC{\mu,\delta}_k},
\end{align*}
which hold provided $ |\alpha| + \tilde{d} \le \mu - \delta |\alpha| $, $ |\alpha| \le k $, $ \tilde{d} \le \mu-\delta n $ and $ n \le k $, that is 
\begin{equation*}
|\alpha| \le \min\{(\mu - \tilde{d})/(1+\delta),k\}, \quad
n \le \min\{(\mu-\tilde{d})/\delta,k\}.\qedhere
\end{equation*}
\end{proof}

\begin{lemma}
\label{lem::key-estimate-para-diff-dyadic}
Let $ \A $ be a linear operator from $ \swtz(\Rd) $ to $ \swtz'(\Rd) $ and let $ m,k \in \R $.
If there exists $ \{\A_j\}_{j\in\N} \in \ell^\infty(\OP{m}{0}) $ and $ \psi, \phi \in \Partition $ such that 
$ \A = \sum_{j \in \N} 2^{jk} \psi_j \A_j \phi_j, $
then $ \A \in \OP{m}{k} $.
\end{lemma}
\begin{proof}
The lemma is a consequence of Proposition~\ref{prop::characterization-weighted-Sobolev}.
\end{proof}

\begin{definition}
Let $ \psi \in \Partition_* $ and define $ \upsi \in \Partition $ by setting $ \upsi_j = \sum_{|j-k|\le 10} \psi_k. $ 
If $ a \in \Gamma^{m,r}_{k,\delta} $ where $ m,k \in \R $, $ r \in \N $ and $ \delta \in [0,1] $, then define the dyadic paradifferential operator
\begin{equation*}
\P{a} = \sum_{j \in \N} \upsi_{j} \T{\psi_j a} \upsi_{j}.
\end{equation*}
\end{definition}

\begin{proposition}
\label{prop:boundedness-P}
If $ a \in \Gamma^{m,r}_{k,\delta} $, then 
$ \P{a} = \O(\M{m,r}{k}{a})_{ \OP{m}{k} }. $
\end{proposition}
\begin{proof}
Note that if $ a \in \Gamma^{m,r}_{k,\delta} $ then $ a \in \jp{x}^{k} \Gamma^{m,0} $.
Therefore, by Proposition~\ref{prop:paradiff-operator-norm}, we have
\begin{equation*}
\| \T{\psi_j a}\|_{\H{\nu}\to\H{\nu-m}} \lesssim \M{m,0}{}{\psi_j a} \lesssim 2^{-jk} \M{m,0}{k,0}{a}.
\end{equation*}
We conclude by Lemma~\ref{lem::key-estimate-para-diff-dyadic}.
\end{proof}

\subsubsection{Symbolic calculus}

\begin{proposition}
\label{prop::composition-P}
Let $ a \in \Gamma^{m,r}_{k,\delta} $, $ b \in \Gamma^{m',r}_{k',\delta} $, $ r \in \N $, $ (m,k), (m',k') \in \R^2 $, $ 0 \le \delta \le 1 $, then
\begin{equation*}
\P{a} \P{b} - \P{a \sharp b}
= \O(\M{m,r}{k,\delta}{a} \M{m',r}{k',\delta}{b} )_{\OP{m+m'-r}{k+k'-\delta r} + \OP{-\infty}{k+k'}}.
\end{equation*}
where $ a \sharp b = \sum_{\alpha\in\N^d}^{|\alpha| < r} \frac{1}{\alpha !} \pxi^\alpha a D_x^\alpha b \in \Sigma^{m+m',r}_{k+k',\delta}. $
\end{proposition}

\begin{proof}
Let $ \tilde{\psi}_j : \N \to \Ccinf $, $ \tilde{\psi}_j = \sum_{|j-j'|\le 50} \psi_{j'} $, so $ \upsi_{j'}\tilde{\psi}_j = \upsi_{j'} $ if $ |j-j'| \le 20 $. Then write,
\begin{equation*}
\P{a} \P{b} 
= \sum_{(j,j')\in\N^2}^{|j-j'|\le 20} \upsi_{j} \T{\psi_j a} \upsi_{j} \upsi_{j'} \T{\psi_{j'}b} \upsi_{j'}
= \sum_{(j,j')\in\N^2}^{|j-j'|\le 20} \tilde{\psi}_j \T{\psi_j a}  \T{\psi_{j'}b} \tilde{\psi}_{j} + \tilde{\psi}_j R_{j,j'} \tilde{\psi}_j,
\end{equation*}
the remainder being
\begin{align*}
R_{j,j'} & = \upsi_{j} \T{\psi_j a} \upsi_{j} \upsi_{j'} \T{\psi_{j'}b} \upsi_{j'} - \T{\psi_j a}  \T{\psi_{j'}b}  \\
& = \O(2^{j(k+k'-\delta r)} \M{m,r}{k,\delta}{a} \M{m',r}{k',\delta}{b})_{\OP{m+m'-r}{0}}
+ \O(2^{j(k+k')}\M{m,r}{k,\delta}{a}\M{m',r}{k',\delta}{b})_{\OP{-\infty}{0}}
\end{align*}
by Proposition~\ref{prop:paradiff-operator-norm}, Proposition~\ref{prop::composition-T} and Corollary~\ref{cor::paralinearization-smooth}. More precisely, when composing $ \T{\psi_j a} $ and $ \T{\upsi_j} $, we use $ \psi_j \upsi_j = \psi_j $ and have
\begin{align*}
\T{\psi_j a} \T{\upsi_j} 
= \T{\psi_j a} & + \O\big(\M{m,r}{}{\psi_j a} \M{0,0}{}{\nabla^r \upsi_j})_{\OP{m-r}{0}} \\
& + \O(\M{m,0}{}{\nabla^r (\psi_j a)} \M{0,r}{}{\upsi_j}\big)_{\OP{m-r}{0}} \\
& + \O\big(\M{m,r}{}{\psi_j a} \M{0,r}{}{\upsi_j}\big)_{\OP{-\infty}{0}},
\end{align*}
where $ \M{0,0}{}{\nabla^r \upsi_j} = \O(2^{-jr}) $, $ \M{m,r}{}{\psi_j a} = \O(2^{jk}) $, and we use $ 0 \le \delta \le 1 $ to induce that
\begin{equation*}
\M{m,0}{}{\nabla^r (\psi_j a)} 
= \O\big(\max_{0\le n \le r}\{2^{-j(r-n)+j(k-\delta n)}\}\big) 
= \O(2^{j(k-\delta r)}).
\end{equation*}
Similar arguments work for the composition $ \T{\upsi_j} \T{\psi_j a} $.

Observe that
$ \sum_{j':|j-j'|\le 20} (\psi_j a) \sharp \psi_{j'}b = (\psi_j a) \sharp b, $ $ \forall j\in\N $. Hence
\begin{equation*}
\sum_{j':|j-j'|\le 20}  \T{\psi_j a}  \T{\psi_{j'}b} 
= \upsi_{j} \T{\psi_j (a\sharp b)} \upsi_{j} +  R_j,
\end{equation*}
where the remainder can be estimated similarly as above,
\begin{equation*}
R_{j} 
= \O(2^{j(k+k'-\delta r)} \M{m,r}{k,\delta}{a} \M{m',r}{k',\delta}{b})_{\OP{m+m'-\delta r}{0}} 
+ \O(2^{j(k+k')} \M{m,r}{k,\delta}{a} \M{m',r}{k',\delta}{b})_{\OP{-\infty}{0}}.
\end{equation*}
We conclude by Lemma~\ref{lem::key-estimate-para-diff-dyadic}.
\end{proof}

\begin{proposition}
\label{prop::adjoint-P}
Let $ a \in \Gamma^{m,r}_{k,\delta} $ with $ (m,k) \in \R^2 $, and $ r \in \N $, $ 0 \le \delta \le 1 $, then
\begin{equation}
\P{a}^* - \P{a^*} = \O(\M{m,r}{k,\delta}{a})_{\OP{m-r}{k-\delta r}+\OP{-\infty}{k}},
\end{equation}
where $ a^* = \sum_{\alpha\in\N^d}^{|\alpha|<r} \frac{1}{\alpha!} \pxi^\alpha D_x^\alpha \bar{a} \in \Sigma^{m,r}_{k,\delta}. $
\end{proposition}
\begin{proof}
Observe that for any real valued $ \psi \in \Ccinf(\Rd) $, 
\begin{equation}
\label{eq::identity-proof-adjoint}
(\psi a)^* = a^* \sharp \psi.
\end{equation}
More precisely, this means that
\begin{align*}
(\psi a)^* 
& = \sum_{|\gamma|<r} \frac{1}{\gamma !} \pxi^\gamma D_x^\gamma (\psi\bar{a}) 
= \sum_{|\gamma|<r} \frac{1}{\gamma!} \sum_{\alpha+\beta=\gamma} \frac{\gamma!}{\alpha! \beta!} D_x^{\alpha}\psi \pxi^\gamma D_x^{\beta} \bar{a}  \\
& = \sum_{|\alpha|+|\beta|<r} \frac{1}{\alpha!} \pxi^\alpha \big(\frac{1}{\beta!} \pxi^\beta D_x^\beta \bar{a}\big) D_x^\alpha \psi
= \sum_{|\beta| < r} \sum_{|\alpha| < r - |\alpha|} \frac{1}{\alpha !} \pxi^\alpha \Big(  \frac{1}{\beta!} \pxi^\beta D_x^\beta \bar{a} \Big) D_x^\alpha \psi  \\
& = \sum_{|\beta| < r} \Big( \frac{1}{\beta!} \pxi^\beta D_x^\beta \bar{a}\Big) \sharp \psi = a^* \sharp \psi.
\end{align*}
Then write 
$ \P{a}^* - \P{a^*} = \sum_{j\in\N} \upsi_{j} (R^1_j + R^2_j) \upsi_{j}, $
where by~\eqref{eq::identity-proof-adjoint},
\begin{equation*}
R^1_j = \T{\psi_j a}^* - \T{(\psi_j a)^*}, \quad
R^2_j = \T{(\psi_j a)^*} - \T{\psi_j a^*} = \T{a^*\sharp\psi_j-\psi_j a^*}.
\end{equation*}
For $ R^1_j $ we use Proposition~\ref{prop::adjoint-T},
\begin{equation*}
R^1_j = \O(\M{m,0}{}{\nabla_{x}^{r}(\psi_j a)})_{\OP{m-r}{0}} 
 = \O(2^{j(k-\delta r)}\M{m,r}{k,\delta}{a})_{\OP{m-r}{0}}.
\end{equation*}
By Lemma~\ref{lem::key-estimate-para-diff-dyadic},
\begin{equation*}
\sum_{j\in\N} \upsi_{j} R^1_j \upsi_{j} = \O(\M{m,r}{k,\delta}{a})_{\OP{m-r}{k-\delta r}+\OP{-\infty}{k}}.
\end{equation*}
Using $ \sum_{j \in \N} \psi_j \equiv 1 $, we induce that
\begin{equation}
\label{eq::vanishing-adjoint-remainder}
\sum_{j \in \N} \px^\alpha \psi_j \equiv 0, \forall \alpha \in \N^d \backslash 0; \quad \sum_{j\in\N} a^*\sharp\psi_{j}-\psi_{j} a^*= 0.
\end{equation}
Then we write
\begin{equation*}
a^*\sharp\psi_{j}-\psi_{j} a^* = \sum_{\alpha \ne 0 \atop |\alpha|+|\beta| < r} D_x^\alpha \psi_j \cdot w_{\alpha\beta}, \quad w_{\alpha\beta} \in \Gamma^{m-|\alpha|-|\beta|,r-|\beta|}_{k-|\beta|\delta,\delta},
\end{equation*}
where the symbols $ w_{\alpha\beta} $ are independent of~$ j $. Write
\begin{equation*}
\sum_{j\in\N} \upsi_{j} R^2_j \upsi_{j} = \sum_{\alpha,\beta} R_{\alpha\beta},\quad
R_{\alpha\beta} = \sum_{j \in \N} \upsi_j \T{D_x^\alpha \psi_j \cdot w_{\alpha\beta}} \upsi_j.
\end{equation*}
By~\eqref{eq::vanishing-adjoint-remainder}, we prove similarly as in Proposition~\ref{prop::composition-P} that
\begin{align*}
\psi_j R_{\alpha\beta}
& = \psi_j \sum_{|j-j'|\le 20} \upsi_{j'} \T{D_x^\alpha \psi_{j'} \cdot w_{\alpha\beta}} \upsi_{j'} \\
& = \O(2^{j(-|\alpha|+k-|\beta|\delta - (r-|\beta|)\delta)}\M{m-|\alpha|-|\beta|,r-|\beta|}{k-|\beta|\delta,\delta}{w_{\alpha\beta}})_{\OP{m-r-|\alpha|}{0}} \\
& \quad \quad + \O(2^{j(-|\alpha|+k-|\beta|\delta)}\M{m-|\alpha|-|\beta|,r-|\beta|}{k-|\beta| \delta,\delta}{w_{\alpha\beta}})_{\OP{-\infty}{0}} \\
&  = \O(2^{j(k-\delta r)}\M{m,r}{k,\delta}{a})_{\OP{m-r}{0}} + \O(2^{jk}\M{m,r}{k,\delta}{a})_{\OP{-\infty}{0}}.
\end{align*}
Setting $ \psi'_j = \sum_{|j'-j|\le 100} \psi_{j'} $. We again conclude by Lemma~\ref{lem::key-estimate-para-diff-dyadic}, and the identity
\begin{equation*}
R_{\alpha\beta} = \sum_{j\in\N} \psi_j R_{\alpha\beta} \psi'_j,
\end{equation*}
that $ R_{\alpha\beta} = \O(\M{m,r}{k,\delta}{a})_{\OP{m-r}{k-\delta r}+\OP{-\infty}{k}}. $
\end{proof}

\subsubsection{Paralinearization}

\begin{proposition}
\label{prop::paralinearization-ab-weighted}
If $ a \in \H{\alpha}_k $, $ b \in \H{\beta}_\ell $ with $ \alpha > d/2 $, $ \beta > d/2 $, $ k \in \R $, $ \ell \in \R $, then for all $ \epsilon > 0 $,
\begin{equation*}
\|ab - \P{a} b - \P{b} a\|_{\H{\alpha+\beta-d/2-\epsilon}_{k+\ell}} \lesssim \|a\|_{\H{\alpha}_k} \|b\|_{\H{\beta}_\ell}.
\end{equation*}
Consequently if $ a \in \HC{\alpha,\delta}_m $ and $ b \in \HC{\beta,\delta}_n $ with $ \delta \ge 0 $, $ \alpha - \delta m > d/2 $, $ \beta - \delta n > d/2 $, then for all $ \epsilon > 0 $,
\begin{equation*}
\|ab - \P{a} b - \P{b} a\|_{\HC{\alpha+\beta-d/2-\epsilon,\delta}_{m+n}} \lesssim \|a\|_{\HC{\alpha,\delta}_m} \|b\|_{\HC{\beta,\delta}_n}.
\end{equation*}
\end{proposition}
\begin{proof}
Decompose the product $ ab $ as follows,
\begin{align*}
ab 
= \sum_{j \in \N}  \psi_j (\upsi_j a) (\upsi_j b)
& = \P{a}b + \P{b}a + R_j^1 + R^2_j,
\end{align*}
where the remainders $ R_j^1 $ and $ R_j^2 $ are defined by
\begin{align*}
R_j^1 & = \psi_j \big( \upsi_j a \cdot \upsi_j b - \T{\upsi_j a} (\upsi_j b) - \T{\upsi_j b} (\upsi_j a) \big), \\
R_j^2 & = \upsi_{j} (\psi_j \T{\upsi_j a} - \T{\psi_j a}) \upsi_{j} b + \upsi_{j} (\psi_j \T{\upsi_j b}  - \T{\psi_j a}) \upsi_{j} a.
\end{align*}
By Proposition~\ref{prop::paraproduct-T},
\begin{equation*}
\|R^1_j\|_{\H{\alpha+\beta-d/2}}
\lesssim \|\upsi_j a\|_{\H{\alpha}} \|\upsi_j b\|_{\H{\beta}}
\lesssim 2^{-j(k+\ell)} \|a\|_{\H{\alpha}_k} \|b\|_{\H{\beta}_\ell}.
\end{equation*}
By Proposition~\ref{prop::composition-T} and Corollary~\ref{cor::paralinearization-smooth} and the Sobolev embedding theorem, for all $ \epsilon > 0 $ we have
\begin{align*}
\psi_j \T{\upsi_j a} \upsi_j - \T{\psi_j a} 
& = 2^{-jk} \O(\|a\|_{\H{\alpha}_k})_{\OP{\alpha-d/2-\epsilon}{0}},\\
\psi_j \T{\upsi_j b} \upsi_j - \T{\psi_j b} 
& = 2^{-j\ell} \O(\|b\|_{\H{\beta}_\ell})_{\OP{\beta-d/2-\epsilon}{0}}
\end{align*} 
We conclude the first statement by Proposition~\ref{prop::characterization-weighted-Sobolev}.

As for the second statement, observe that if $ 0 \le k \le m $ and $ 0 \le \ell \le n $, then
\begin{equation*}
\|ab-\P{a}b - \P{b}a\|_{H^{(\alpha-\delta k)+(\beta-\delta \ell)-d/2-\epsilon,\delta}_{k+\ell}}
\lesssim \|a\|_{H^{\alpha-\delta k}_k} \|b\|_{H^{\beta- \delta \ell}_\ell}
\lesssim \|a\|_{\HC{\alpha,\delta}_m} \|b\|_{\HC{\beta,\delta}_n}.
\end{equation*}
We conclude by noting that for all $ p \in \N \cap [0,m+n]$, there exists $ k \in \N \cap [0,m] $ and $ \ell \in \N \cap [0,n] $ such that $ p = k+\ell $.
\end{proof}

\begin{proposition}
\label{prop::paralinearization-weighted}
Let $ F \in \Cinf(\R) $ with $ F(0) = 0 $.
For all $ \mu > d/2 $, there exists some monotonically increasing function $ C : \R_+ \to \R_+ $, such that for all $ k \ge 0 $ and all $ u \in \H{\mu}_k $, we have
\begin{equation*}
\|F(u)\|_{\H{\mu}_k} + \|F(u) - \P{F'(u)} u\|_{\H{2\mu-d/2}_k} \le C(\|u\|_{\H{\mu}}) \|u\|_{\H{\mu}_k}.
\end{equation*}
Consequently, if $ u \in \HC{\mu,\delta}_k $ with $ \delta \ge 0 $ and $ \mu - \delta k > d/2 $, then
\begin{equation*}
\|F(u)\|_{\HC{\mu,\delta}_k} + \|F(u) - \P{F'(u)} u\|_{\HC{2\mu-d/2,2\delta}_k} \le C(\|u\|_{\H{\mu}}) \|u\|_{\HC{\mu,\delta}_k}.
\end{equation*}
\end{proposition}
\begin{proof}
Decompose 
$ F(u) = \sum_{j \ge 0} \psi_j F(\upsi_{j} u). $
By Proposition~\ref{prop::paralinearization-T},
\begin{equation*}
\|F(\upsi_{j} u)\|_{\H{\mu}} \le C(\|\upsi_{j} u\|_{\H{\mu}}) \|\upsi_{j} u\|_{\H{\mu}}
\le C(\|u\|_{\H{\mu}}) \|\upsi_{j} u\|_{\H{\mu}}.
\end{equation*}
Then write 
$ \psi_j F(\upsi_{j} u) = \upsi_{j} \T{\psi_j F'(u)} \upsi_{j} u + \upsi_j R_j,  $
where
\begin{equation*}
R_j  = \psi_j ( F(\upsi_{j} u) - \T{F'(\upsi_{j} u)} \upsi_{j} u) + \upsi_{j} (\psi_j \T{F'(\upsi_{j} u)} - \T{\psi_j F'(\upsi_{j} u)} )\upsi_{j} u.
\end{equation*}
By Proposition~\ref{prop::paralinearization-T}, Proposition~\ref{prop::paraproduct-T} and Corollary~\ref{cor::paralinearization-smooth}, 
\begin{equation*}
\|R_j\|_{\H{2\mu-d/2}} \le C(\|u\|_{\H{\mu}}) \|\upsi_{j} u\|_{\H{\mu}}.
\end{equation*}
We conclude the first statement with Proposition~\ref{prop::characterization-weighted-Sobolev}.
To prove the second statement, note that for all $ j \in \N \cap [0,k] $, we have
\begin{equation*}
\|F(u)\|_{\H{\mu-\delta j}_j} + \|F(u) - \P{F'(u)} u\|_{\H{2(\mu-\delta j)-d/2}_j} \le C(\|u\|_{\H{\mu-\delta j}}) \|u\|_{\H{\mu-\delta j}_j}.\qedhere
\end{equation*}
\end{proof}

\subsection{Semiclassical paradifferential calculus}

We develop a semiclassical dyadic paradifferential calculus and a quasi-homogeneous semiclassical paradifferential calculus, using scaling arguments inspired by Métivier--Zumbrun~\cite{MZ05viscous}.

\subsubsection{Semiclassical paradifferential operators}

\begin{definition}
For all $ h \in (0,1] $, define the scaling operator 
\begin{equation}
\label{eq:def-semiclassical-isometry}
\tau_h : u(\cdot) \mapsto h^{d/2} u(h\cdot).
\end{equation}
\begin{enumerate}
\item If $ b \in \Gamma^{m,r} $, then define $ \T{b}^{h} = \tau_h^{-1} \T{\theta_{h,*}^{1,0} b} \tau_h. $
\item If $ a \in \Gamma^{m,r}_{k,\delta} $, then define $ \P{a}^{h} =  \sum_{j\in\N} \upsi_{j} \T{\psi_j  a}^{h}  \upsi_{j}. $
\item If $ \epsilon \ge 0 $, then define $ \P{a}^{h,\epsilon} = \P{\theta_{h,*}^{\epsilon,0} a}^h. $
\end{enumerate}

\end{definition}

\begin{proposition}
If $ \epsilon \ge 0 $ and $ a \in \Gamma^{m,0}_{k,0} $ where $ m \le 0 $, $ k \le 0 $ then 
$ \sup_{h \in (0,1]} \|\P{a}^{h,\epsilon}\|_{\Ltwo \to \Ltwo} < \infty. $
\end{proposition}
\begin{proof}
Observe that $ \theta_{h,*}^{1+\epsilon,0} a = \O(1)_{\Gamma^{0,0}} $. We conclude with Lemma~\ref{lem::key-estimate-para-diff-dyadic}.
\end{proof}

\subsubsection{Semiclassical symbolic calculus}

\begin{definition}
If $ a_h \in \Dscr'(\R^{2d}) $ and $ \epsilon \ge 0 $, we say that $ a_h \in \sigma_\epsilon $ if 
\begin{equation*}
\overline{\cup_{0<h<1} \supp a_h} \cap \mathscr{N}_{\epsilon,1} = \emptyset.
\end{equation*}
\end{definition}

\begin{proposition}
\label{prop::composition-P-h}
Let $ (m,k), (m',k') \in (\R \cup \{-\infty\}) ^2  $, $ r \in \N $, with $ r \ge m+m' $, $ \delta r \ge k + k' $.
Let $ a_h \in \Gamma^{m,r}_{k,\delta} \cap \sigma_0 $ and $ b_h \in \Gamma^{m',r}_{k',\delta} \cap \sigma_0 $ such that for some $ R_h \ge 0 $ depending on~$ h $,
\begin{equation}
\label{eq::position-condition-semiclassical-composition}
\supp a_h \cap \supp b_h \subset \{|x| \ge R_h\} \times \Rd. 
\end{equation}
Then for $ h > 0 $ sufficiently small,
\begin{equation*}
\P{a_h}^{h} \P{b_h}^{h} - \P{a_h \sharp_h b_h}^{h} 
= \O(h^{r}(1+R_h)^{k+k'-\delta r})_{\Ltwo\to\Ltwo},
\end{equation*}
where the symbol $ a_h \sharp_h b_h = a_h \sharp_{h,r}^{0,1} b_h \in {}_h\Sigma^{m+m',r}_{k+k',\delta} $ is defined by~\eqref{eq:def-symbolic-composition-finite}.
\end{proposition}
\begin{proof}
By~\eqref{eq::position-condition-semiclassical-composition}, if $ \psi_j a_h \ne 0 $ and $ \psi_j b_h \ne 0 $, then $ j \gtrsim \log_2(1+R_h). $ 
We claim that
\begin{equation}
\label{eq:semiclassical-para-composition-remainder}
\begin{split}
\P{a}^{h} \P{b}^{h}
& = \sum_{j \gtrsim \log_2 (1+R_h) \atop |j'-j| \le 20} \upsi_{j} \T{\psi_j a_h}^h \upsi_{j} \upsi_{j'} \T{\psi_{j'} b_h}^h \upsi_{j'} \\
& = \sum_{j \gtrsim \log_2 (1+R_h) \atop |j'-j| \le 20} \upsi_{j} \T{(\psi_j a_h) \sharp_h (\psi_{j'} b_h)}^h \upsi_{j'}
+ \O(h^r (1+R_h)^{k+k'-\delta r})_{\Ltwo\to\Ltwo}.
\end{split}
\end{equation}
Then we conclude by
\begin{equation*}
\sum_{j': |j'-j|\le 20} (\psi_j a_h) \sharp_h (\psi_{j'} b_h)
= \psi_j (a_h \sharp_h b_h).
\end{equation*}
It remains to prove~\eqref{eq:semiclassical-para-composition-remainder}.
We use~\eqref{eq::admissible-function-def} to deduce that $ \Fourier(\T{ \theta_{h,*}^{1,0} (\psi_{j'} b_h)}u) $ vanishes in a neighborhood of $ \xi = 0 $. By~\eqref{eq::estimate-un-paralinearization}, for some $ \pi' \in \Cinf(\Rd) $ which vanishes near $ \xi = 0 $ and equals to~$ 1 $ outside a neighborhood of $ \xi=0 $, and for all $ m + m' \le N \in \N $,
\begin{equation}
\label{eq:semiclassical-para-composition-remainder-2}
\begin{split}
\tau_h \T{\psi_j a_h}^h \upsi_{j} \upsi_{j'} \T{\psi_{j'} b_h}^h \tau_h^{-1}
& =  \T{\theta_{h,*}^{1,0}(\psi_j  a_h)} \theta_{h,*}^{1,0}(\upsi_{j} \upsi_{j'}) \pi'(D_x) \T{ \theta_{h,*}^{1,0} (\psi_{j'} b_h)}  \\
& = \T{\theta_{h,*}^{1,0} ( \psi_j a_h)} \T{\theta_{h,*}^{1,0}(\upsi_{j} \upsi_{j'}) \otimes \pi'} \T{\theta_{h,*}^{1,0} (\psi_{j'} b_h)}  \\ 
& \quad + \O(\M{m,0}{}{\psi_j a_h})_{\OP{m}{0}} \O(2^{-jN} h^N)_{\OP{-N}{0}} \O(\M{m',0}{}{\psi_j b_h})_{\OP{m'}{0}}.
\end{split}
\end{equation} 
Then we use Proposition~\ref{prop::composition-T} and the fact that $ a_h, b_h \in \sigma_0 $ to deduce
\begin{align*}
\T{\theta_{h,*}^{1,0} ( \psi_j a_h)} & \T{\theta_{h,*}^{1,0}(\upsi_{j} \upsi_{j'}) \otimes \pi'} \T{\theta_{h,*}^{1,0} (\psi_{j'} b_h)} \\
& =  \T{\theta_{h,*}^{1,0}(\psi_j a_h) \sharp \theta_{h,*}^{1,0} (\upsi_{j} \upsi_{j'} \otimes \pi') \sharp \theta_{h,*}^{1,0} (\psi_{j'} b_h)} \\
& \quad + \O(\M{m,0}{}{\nabla^r \theta_{h,*}^{1,0}(\psi_j a_h)} \M{0,r}{}{\theta_{h,*}^{1,0}(\psi_j \psi_{j'})} \M{m',r}{}{\theta_{h,*}^{1,0}(\psi_{j'}b_h)})_{\OP{m+m'-r}{0}}\\
& \quad + \O(\M{m,r}{}{\theta_{h,*}^{1,0}(\psi_j a_h)} \M{0,0}{}{\nabla^r  \theta_{h,*}^{1,0}(\psi_j \psi_{j'})} \M{m',r}{}{\theta_{h,*}^{1,0}(\psi_{j'}b_h)})_{\OP{m+m'-r}{0}} \\
& \quad + \O(\M{m,r}{}{\theta_{h,*}^{1,0}(\psi_j a_h)} \M{0,r}{}{\theta_{h,*}^{1,0}(\psi_j \psi_{j'})} \M{m',0}{}{\nabla^r  \theta_{h,*}^{1,0}(\psi_{j'}b_h)})_{\OP{m+m'-r}{0}}.
\end{align*}
To estimate the remainders, we see that for each $ \alpha \in \N^d $ with $ |\alpha| = r $,
\begin{align*}
\px^\alpha \theta_{h,*}^{1,0}(\psi_j a_h)
& = \sum_{\alpha_1+\alpha_2 = \alpha} \frac{\alpha!}{\alpha_1! \alpha_2!} \px^{\alpha_1} \theta_{h,*}^{1,0}\psi_j \px^{\alpha_2} \theta_{h,*}^{1,0} a_h \\
& = \sum_{\alpha_1+\alpha_2 = \alpha} \frac{\alpha!}{\alpha_1! \alpha_2!} \O(h^{|\alpha_1|} 2^{-j|\alpha_1|} \times h^{|\alpha_2|} 2^{j(k-\delta|\alpha_2|)})_{\Linf}
= \O(h^r 2^{j(k-\delta r)})_{\Linf},
\end{align*}
where we use $ 0 \le \delta \le 1 $. Therefore, the first term in the remainder is
\begin{equation*}
\O(h^r 2^{j(k+k'-\delta r)})_{\Ltwo\to\Ltwo}
= \O(h^r (1+R_h)^{k+k'-\delta r})_{\Ltwo\to\Ltwo}.
\end{equation*}
Similar methods apply to the other two terms and we conclude that
\begin{equation}
\label{eq:semiclassical-para-composition-remainder-3}
\T{\theta_{h,*}^{1,0} ( \psi_j a_h)} \T{\theta_{h,*}^{1,0}(\upsi_{j} \upsi_{j'}) \otimes \pi'} \T{\theta_{h,*}^{1,0} (\psi_{j'} b_h)}
= \T{\theta_{h,*}^{1,0} ((\psi_j a_h)  \sharp_h (\psi_{j'} b_h))} + \O(h^r (1+R_h)^{k+k'-\delta r})_{\Ltwo\to\Ltwo}.
\end{equation}
The estimate~\eqref{eq:semiclassical-para-composition-remainder} follows from~\eqref{eq:semiclassical-para-composition-remainder-2} and~\eqref{eq:semiclassical-para-composition-remainder-3}.
\end{proof}

Combining the analysis of Proposition~\ref{prop::composition-P-h}, Proposition~\ref{prop::adjoint-P}, using Proposition~\ref{prop::composition-T}, we obtain a similar result for the adjoint, to which the proof we shall omit, as it is similar as above.

\begin{proposition}
\label{prop::adjoint-P-h}
Let $ (m,k) \in (\R \cup \{-\infty\}) ^2  $, $ r \in \N $, with $ r \ge m $, $ \delta r \ge k. $ Let $ a_h \in \Gamma^{m,r}_{k,\delta} \cap \sigma_0 $, such that for some $ R_h \ge 0 $ depending on~$ h $,
$ \supp a_h \subset \{|x| \ge R_h\} \times \Rd,  $
then for $ h > 0 $ sufficiently small,
\begin{equation*}
(\P{a_h}^{h})^* - \P{a^*_h}^{h} 
= \O(h^{r}(1+R_h)^{k-\delta r})_{\Ltwo\to\Ltwo},
\end{equation*}
where $ a^*_h = \zeta_{h,r}^{0,1} a_h \in {}_h\Sigma^{m,r}_{k,\delta} $ is defined by~\eqref{eq:def-symbolic-adjoint-finite}.
\end{proposition}

\begin{corollary}
\label{cor::mixed-composition-epsilon-0}
Let $ \epsilon \ge 0 $, $ (m,k), (m',k') \in (\R \cup \{-\infty\})^2 $, $ r \in \N $ with $ r \ge \max\{m+m',k'\} $, $ k \le 0. $ 
If $ a_h \in \Gamma^{m,r}_{k,1} \cap \sigma_\epsilon $ and $ b_h \in \Gamma^{m',r}_{k',1} \cap \sigma_0 $, then
\begin{align*}
\P{a_h}^{h,\epsilon} \P{b_h}^h - \P{(\theta_{h,*}^{\epsilon,0} a_h) \sharp_h b_h }^h
& = \O(h^{(1+\epsilon)r-\epsilon (k+k')})_{\Ltwo\to\Ltwo}, \\
\P{b_h}^{h} \P{a_h}^{h,\epsilon} - \P{b_h \sharp_h (\theta_{h,*}^{\epsilon,0} a_h)}^h
& = \O(h^{(1+\epsilon)r-\epsilon (k+k')})_{\Ltwo\to\Ltwo}.
\end{align*}
\end{corollary}
\begin{proof}
It suffices to observe that, if $ \epsilon > 0 $ then 
$ \supp \theta_{h,*}^{\epsilon,0} a_h \subset \{|x| \gtrsim h^{-\epsilon}\} $
and $ \theta_{h,*}^{\epsilon,0} a_h = \O(h^{-\epsilon k})_{\Gamma^{m,r}_{0,1}}. $ 
We conclude by Proposition~\ref{prop::composition-P-h}.
\end{proof}

\begin{corollary}
\label{cor:paradiff-dyadic-h-epsilon}
Let $ \epsilon \ge 0 $, $ (m,k), (m',k') \in (\R \cup \{-\infty\})^2 $, $ r \in \N $, with $ r \ge m+m' $, $ k \le 0 $, $ k' \le 0. $ 
If $ a_h \in \Gamma^{m,r}_{k,1} \cap \sigma_\epsilon $ and $ b_h \in \Gamma^{m',r}_{k',1} \cap \sigma_\epsilon $, then for $ h > 0 $ sufficiently small,
\begin{equation*}
\P{a_h}^{h,\epsilon} \P{b_h}^{h,\epsilon} - \P{a_h \sharp_h^\epsilon b_h }^{h,\epsilon}
= \O(h^{(1+\epsilon)r-\epsilon(k+k')})_{\Ltwo\to\Ltwo},
\end{equation*}
where the symbol $ a_h \sharp_h^\epsilon b_h = a_h \sharp_{h,r}^{\epsilon,1} \in {}_{h^{1+\epsilon}}\Sigma^{m+m',r}_{k+k',1} $ is defined by~\eqref{eq:def-symbolic-composition-finite}.
\end{corollary}
\begin{proof}
It suffices to use the identity $ (\theta_{h,*}^{\epsilon,0} a_h) \sharp_h (\theta_{h,*}^{\epsilon,0} b_h)
= \theta_{h,*}^{\epsilon,0} (a_h \sharp_h^\epsilon b_h) $.
\end{proof}

\subsubsection{Some technical lemmas}

The results above only concerned the high frequency regime as we require the $ \sigma_\epsilon $ condition. The next lemma studies the interaction of high frequencies and low frequencies.

\begin{lemma}
\label{lem::low-freq-composition}
Let $ m \in \R $, $ a_h \in \Gamma^{0,0} $, $ b_h \in \Gamma^{0,0} $, such that for some $ R > 0 $,
\begin{equation*}
\supp a_h \in \{|\xi| \ge R \},\quad
\supp b_h \in \{|\xi| \le h^{-1}R/4 \}.
\end{equation*}
Then $ \P{a_h}^h \P{b_h} = \O(h^{\infty})_{\Ltwo\to\Ltwo}. $
\end{lemma}
\begin{remark}
This lemma concerns the estimate of $ \P{a_h}^h \P{b_h} $, not $ \P{a_h}^h \P{b_h}^h $.
This is not a typo.
\end{remark}
\begin{proof}
By definition
\begin{equation*}
\widehat{\T{\psi_j b_h} u}(\xi) 
= (2\pi)^{-d}  \int \chi(\xi-\eta,\eta) \pi(\eta)  \widehat{\psi_j b_h}(\xi-\eta,\eta) \widehat{u} (\eta) \d\eta.
\end{equation*}
The admissibility of~$ \chi $ implies 
$ \supp \widehat{\T{\psi_j b_h} u} \subset \{|\xi| \le h^{-1}R/3\}. $
Therefore, for any $ |j'-j| \le 20 $,
\begin{align*}
\upsi_j \T{\psi_j a_h}^h & \upsi_j \upsi_{j'} \T{\psi_{j'} b_h} \upsi_{j'} \\
& = \upsi_j \T{\psi_j a_h}^h \pi(h D_x/R) \upsi_j \upsi_{j'} \big(1-\pi(2 h D_x/R)\big) \T{\psi_{j'} b_h} \upsi_{j'} \\
& = \upsi_j \O(h^\infty)_{\Ltwo\to\Ltwo} \upsi_{j'}.
\end{align*}
We conclude by Lemma~\ref{lem::key-estimate-para-diff-dyadic}.
\end{proof}

\begin{corollary}
\label{cor::homogenization-semiclassical}
If $ a \in \Gamma^{m,0} $ is homogeneous of degree~$ m $ with respect to~$ \xi $, then for $ b \in \Gamma^{0,0} \cap \sigma_0 $ and $ h > 0 $ sufficiently small,
\begin{equation*}
\P{b}^h\big(h^m \P{a} - \P{a}^h\big) = \O(h^\infty)_{\Ltwo \to \Ltwo}.
\end{equation*}
\end{corollary}
\begin{proof}
By a direct verification using~\eqref{eq::definition-paradiff-T}, the homogeneity of~$ a $ and the admissible function~$ \chi $, and Corollary~\ref{cor:change-admissible-function}, we see that 
$ h^m \P{a} - \P{a}^h = \P{\tilde{a}_h}' $
where $ \mathcal{P}' $ denotes the paradifferential quantization with any admissible pair $ (\pi',\chi) $ such that $ \pi \pi' = \pi' $, and the symbol 
\begin{equation*}
\tilde{a}_h(x,\xi) = (\pi(h\xi) - \pi(\xi))a(x,h\xi) \in \Gamma^{0,0}
\end{equation*}
satisfies the condition
\begin{equation*}
\supp \tilde{a}_h \subset \Rd \times \supp (1-\pi(h\cdot))
\subset \Rd \times \{|\xi| \le 2h^{-1} \}.
\end{equation*}
We conclude by Lemma~\ref{lem::low-freq-composition}.
\end{proof}

\begin{lemma}
\label{lem::un-paradiff-T-h}
If $ a_h \in \Gamma^{m,r} \cap \sigma_0 $ with $ r \ge \max\{m,0\} + \tilde{d} $, then for $ h > 0 $ sufficiently small,
\begin{equation*}
\T{a_h}^h - \op_h(a_h) = \O(h^r)_{\Ltwo\to\Ltwo}.
\end{equation*}
\end{lemma}
\begin{proof}
By Calder\'{o}n--Vaillancourt theorem, we have
\begin{align*}
\T{a_h}^h - \op_h(a_h) 
& = \tau_h^{-1} \big( \T{\theta_{h,*}^{1,0} a_h} - \op(\theta_{h,*}^{1,0} a_h) \big)\tau_h \\
& = \O\Big(\sum_{|\alpha|,|\beta|\le \tilde{d}}\big\| \pxi^\alpha \px^\beta\big(\sigma_{\theta_{h,*}^{1,0} a_h}-\theta_{h,*}^{1,0} a_h\big)\big\|_{\Linf}\Big)_{\Ltwo\to\Ltwo}.
\end{align*}
By hypothesis $ r \ge \max\{m,0\} + |\beta| \ge |\beta| $. We use~\eqref{eq::estimate-un-paralinearization} to deduce that
\begin{align*}
\big\| \pxi^\alpha \px^\beta\big(\sigma_{\theta_{h,*}^{1,0} a_h}-\theta_{h,*}^{1,0} a_h\big)\big\|_{\Linf}
& \lesssim \M{0,0}{}{\px^\beta(\sigma_{\theta_{h,*}^{1,0} a_h}-\theta_{h,*}^{1,0} a_h)} \\
& \lesssim \M{\max\{m,0\} - r + |\beta|,0}{}{\px^\beta(\sigma_{\theta_{h,*}^{1,0} a_h}-\theta_{h,*}^{1,0} a_h)}\\
& \lesssim \M{\max\{m,0\},0}{}{\nabla_{x}^r(\theta_{h,*}^{1,0}a_h)} \\
& \lesssim h^r \M{m,0}{}{a_h}.\qedhere
\end{align*}
\end{proof}

\begin{lemma}
\label{lem::un-paradiff-estimate-smooth-symbol}
If $ a_h \in \Gamma^{m,\infty} \cap \sigma_0 $ with $ m \in \R \cup \{-\infty\} $, then for $ h > 0 $ sufficiently small,
\begin{equation*}
\P{a_h}^{h} - \op_h(a_h) = \O(h^{\infty})_{\Ltwo\to\Ltwo}.
\end{equation*}
\end{lemma}
\begin{proof}
By Lemma~\ref{lem::un-paradiff-T-h} and Lemma~\ref{lem::key-estimate-para-diff-dyadic}, 
\begin{equation*}
\P{a_h}^{h} - \sum_{j \in \N} \upsi_j \op_h(\psi_j a_h) \upsi_j = \O(h^\infty)_{\Ltwo\to\Ltwo}.
\end{equation*}
Note that, uniformly in $ j \in \N $, we have
\begin{equation*}
\upsi_{j} \sharp_h (\psi_j a_h) \sharp_h \upsi_{j} = \psi_j a_h + \upsi_j \O(h^\infty)_{\Gamma^{-\infty,\infty}}.
\end{equation*}
Therefore,
\begin{equation*}
\sum_{j \in \N} \upsi_{j} \sharp_h (\psi_j a_h) \sharp_h \upsi_{j}
= a_h + \O(h^\infty)_{\Gamma^{-\infty,\infty}}.\qedhere
\end{equation*}
\end{proof}

\subsubsection{Symbols with limited regularities in $ (x,\xi) $}

The symbols we have encountered so far have limited regularities in the $ x $-variable but are smooth with respect to the $ \xi $-variable.
When studying the propagation of singularities for nonlinear equations, we need to solve Hamiltonian equations which transfer the limited regularity in the $ x $-variable to the $ \xi $-variable.
Therefore we need to discuss in this section paradifferential operators with symbols that have limited regularities in both the~$ x $- and $ \xi $-variables.
As we do not intend to obtain optimal regularities, we shall content ourselves with an approach by approximation.

\begin{definition}
For all $ r \in \N $, the symbol class $ \Upsilon^r $ is the set of all $ a \in \Linf_\loc\big(\Rd \times (\Rd \backslash 0)\big) $ compactly supported in $ \Rd \times (\Rd \backslash 0) $ such that $ N^r(a) < + \infty $, where
\begin{equation*}
N^r(a) = \sum_{\alpha \in \N^{2d},|\alpha|\le r} \|\partial_{x,\xi}^\alpha a\|_{\Linf_x W^{\tilde{d},\infty}_\xi}.
\end{equation*}
\end{definition}
If $ a \in \Upsilon^{r+1} $ with $ r \in \N $, then the paradifferential operator $ T_a $ is defined via approximating~$ a $ by smooth symbols.
To be precise, let $ \Omega \subset \subset \Rd \times (\Rd \backslash 0) $ be an open neighborhood of $ \supp a $ and let $ \{a_n\}_{n\in\N} \subset \Ccinf(\Omega) $ such that
\begin{equation*}
\lim_{n\to\infty} N^{r}(a_n - a) = 0.
\end{equation*}
Note that such an approximation is always possible because $a$ is compactly supported and we only require the convergence with respec to the $N^r$-norm (not the $N^{r+1}$-norm)!
By Proposition~\ref{prop:paradiff-operator-norm} and Lemma~\ref{lem::change-of-admissible-function}, for all $ n,m \in \N $ we have
\begin{equation*}
\|T_{a_n}-T_{a_m}\|_{L^2 \to L^2}\lesssim M^{0,0}_{\tilde{d}} (a_n - a_m) \lesssim N^0(a_n - a_m) \le N^{r}(a_n - a_m).
\end{equation*}
Therefore, for all $ u \in L^2 $, the sequence $ \{T_{a_n} u\}_{n\in\N} $ is Cauchy in~$ L^2 $ and we define
\begin{equation*}
T_a u = \lim_{n\to\infty} T_{a_n} u.
\end{equation*}
Clearly this definition is independent of the choice of the sequence $ \{a_n\}_{n\in\N} $ and extends the definition of paradifferential operators with symbols that are smooth with respect to~$ \xi $.
Then we define the operators $ T_a^h $, $ \P{a} $, $ \P{a}^h $ and $ \P{a}^h $ exactly as before.

\begin{proposition}
\label{prop:est:upsilon:bounded}
If $ a \in \Upsilon^{r+1} $ with $ r \ge 0 $, then for all $ h \in (0,1] $, we have $ T_a^h : L^2 \to L^2 $.
Moreover,
\begin{equation*}
\sup_{h \in (0,1]}\|T_a^h\|_{L^2 \to L^2} \lesssim N^0(a).
\end{equation*}
Consequently, for all $ \epsilon \ge 0 $ we have
\begin{equation*}
\sup_{h \in (0,1]} \|\P{a}^{h,\epsilon}\|_{L^2 \to L^2} \lesssim N^0(a).
\end{equation*}
\end{proposition}
\begin{proof}
The general case $ h \in (0,1] $ follows from the case $ h = 1 $ and we shall assume $ h = 1 $.
Choose a convergent sequence $ \{a_n\}_{n\in\N} \subset \Ccinf(\Omega) $ as above.
For all $ u \in L^2 $ with $ \|u\|_{L^2} = 1 $, we have
\begin{equation*}
\|T_a u\|_{L^2} 
\le \|T_a u - T_{a_n} u\|_{L^2} + \|T_{a_n} u\|_{L^2},
\end{equation*}
where $ \lim_{n\to\infty} \|T_a u - T_{a_n} u\|_{L^2} = 0$ by the definition of $ T_a u $ and 
\begin{equation*}
\|T_{a_n} u\|_{L^2} \lesssim N^0(a_n) \lesssim N^0(a-a_n) + N^0(a) \to N^0(a).
\end{equation*}
Therefore, passing $ n \to \infty $ we conclude that $ \|T_a\|_{L^2\to L^2} \lesssim N^0(a) $.
The estimate for $ \P{a}^{h,\epsilon} $ follows similarly as in Proposition~\ref{prop:boundedness-P}.
\end{proof}

Combining the approximation method above and the analysis in Proposition~\ref{prop::composition-P-h}, we obtain the following corollaries similarly to Corollaries~\ref{cor::mixed-composition-epsilon-0} and~\ref{cor:paradiff-dyadic-h-epsilon}.

\begin{corollary}
Let $ \epsilon \ge 0 $, $ (m,k) \in (\R \cup \{-\infty\})^2 $, $ r \in \N $ with $ r \ge 0 $.
If $ a_h \in \Upsilon^{r+1} \cap \sigma_\epsilon $ and $ b_h \in \Gamma^{m,r}_{k,1} \cap \sigma_0 $, then for all $ k' \in \R $ such that $ r \ge k + k' $, we have
\begin{align*}
\P{a_h}^{h,\epsilon} \P{b_h}^h - \P{(\theta_{h,*}^{\epsilon,0} a_h) \sharp_h b_h }^h
& = \O(h^{(1+\epsilon)r-\epsilon (k+k')} M^{-m,r}_{k',1}(a_h) M^{m,r}_{k,1}(b_h))_{\Ltwo\to\Ltwo}, \\
\P{b_h}^{h} \P{a_h}^{h,\epsilon} - \P{b_h \sharp_h (\theta_{h,*}^{\epsilon,0} a_h)}^h
& = \O(h^{(1+\epsilon)r-\epsilon (k+k')} M^{-m,r}_{k',1}(a_h) M^{m,r}_{k,1}(b_h))_{\Ltwo\to\Ltwo}.
\end{align*}
\end{corollary}
\begin{proof}
Let $ a_h $ be a sequence of approximating symbols $ \{a_h^n\}_{n \in \N}\subset \Ccinf(\R^{2d}) \cap \sigma_\epsilon $ which is bounded, uniformly in~$ h \in (0,1] $, with respect to the norm $ N^r(\cdot) $.
Note that for all $ k' \in \R $, we have $ M^{-m,r}_{k',1}(\cdot) \lesssim N^r(\cdot) $.
And thus,  when $ n \in \N $ is sufficiently large, we have $ M^{-m,r}_{k',1}(a_h^n-a_h) \le 2 N^r(a_h^n - a_h) = o(1) $.
By Corollary~\ref{cor::mixed-composition-epsilon-0}, if $ r \ge k + k' $, we have
\begin{align*}
\P{a_h^n}^{h,\epsilon} \P{b_h}^h - \P{(\theta_{h,*}^{\epsilon,0} a_h^n) \sharp_h b_h }^h
& = \O(h^{(1+\epsilon)r-\epsilon (k+k')}M^{-m,r}_{k'1}(a_h^n) M^{m,r}_{k,1}(b_h))_{\Ltwo\to\Ltwo}\\
& = \O(h^{(1+\epsilon)r-\epsilon (k+k')}M^{-m,r}_{k',1}(a_h) M^{m,r}_{k,1}(b_h))_{\Ltwo\to\Ltwo} + o(1)_{L^2\to L^2}.
\end{align*}
In fact, for all $ u \in \swtz(\Rd) $, as $ n \to \infty $, by Proposition~\ref{prop:est:upsilon:bounded}, we have
\begin{align*}
\big\| \P{a_h}^{h,\epsilon} \P{b_h}^h u & - \P{(\theta_{h,*}^{\epsilon,0} a_h) \sharp_h b_h }^h u \big\|_{L^2}\\
& = \big\|\big( \P{a_h}^{h,\epsilon}-\P{a_h^n}^{h,\epsilon} \big) \P{b_h}^h u \big\|_{L^2} + \big\| \P{(\theta_{h,*}^{\epsilon,0} (a_h-a_h^n)) \sharp_h b_h }^h) u \big\|_{L^2}
+ \big\| \P{a_h^n}^{h,\epsilon} \P{b_h}^h u - \P{(\theta_{h,*}^{\epsilon,0} a_h^n) \sharp_h b_h }^h u \big\| \\
& = o(1) \big(  \|\P{b_h}^h u \|_{L^2} + \|u\|_{L^2} \big) + \O(h^{(1+\epsilon)r-\epsilon (k+k')} M^{-m,r}_{k',1}(a_h) M^{m,r}_{k,1}(b_h)) \|u\|_{\Ltwo}.
\end{align*}
Passing $ n \to \infty $ and then use the density of $ \swtz(\Rd) $ in $ L^2 $, we conclude that for all $ u \in L^2 $, we have
\begin{equation*}
\big\| \P{a_h}^{h,\epsilon} \P{b_h}^h u - \P{(\theta_{h,*}^{\epsilon,0} a_h) \sharp_h b_h }^h u \big\|_{L^2} = \O(h^{(1+\epsilon)r-\epsilon (k+k')}M^{-m,r}_{k',1}(a_h) M^{m,r}_{k,1}(b_h)) \|u\|_{\Ltwo}.
\end{equation*}
The estimate for $ \P{b_h}^{h} \P{a_h}^{h,\epsilon} - \P{b_h \sharp_h (\theta_{h,*}^{\epsilon,0} a_h)}^h $ is similar.
\end{proof}

\begin{corollary}
\label{cor:symbolic-calculus-limited-xi-regularity}
If $ \epsilon \ge 0 $ and $ a_h, b_h \in \Upsilon^{r+1} $ where $ r \in \N $, then
\begin{equation*}
\P{a_h}^{h,\epsilon} \P{b_h}^{h,\epsilon} - \P{a_h \sharp_h^\epsilon b_h }^{h,\epsilon}
= \O(h^{(1+\epsilon)r})_{\Ltwo\to\Ltwo},
\end{equation*}
where the symbol $ a_h \sharp_h^\epsilon b_h = a_h \sharp_{h,r}^{0,\epsilon} b_h $ is defined by~\eqref{eq:def-symbolic-composition-finite}.
\end{corollary}

\subsubsection{Almost sharp G{\aa}rding Inequality for paradifferential operators}

We need an almost sharp G{\aa}rding inequality for our paradifferential calculus.
There are various work on the (almost) sharp G{\aa}rding inequality for pseudodifferential operators with limited regularities, see e.g.,~\cite{Taylor91pseudodiff-nonlinearPDE,Tataru02:Fefferman-Phong,Herau02:melin}.

\begin{lemma}
If $ \epsilon \in (0,1) $, $ a_h \in M_{n\times n}(\Gamma^{0,r}) \cap \sigma_0 $ and is compactly supported, where~$ n \in \N $, $ r \ge \max\{\tilde{d},\epsilon^{-1}-1\} $ and $ \Re\, a \ge 0 $, then for all $ \epsilon \in (0,1) $, there exists $ C > 0 $ such that for all $ u \in L^2 $,
\begin{equation*}
\Re(T_{a_h}^h u,u)_{L^2} \ge -C h^{1-\epsilon} \|u\|_{L^2}^2.
\end{equation*}
\end{lemma}
\begin{proof}
By Lemma~\ref{lem::un-paradiff-T-h} and the condition $ r \ge \tilde{d} $, we may replace $ T_{a_h}^h $ with $ \op_h(a_h) $ in the above inequality.
As $ a_h \in \sigma_0 $ and is compactly supported, we have $ \{b_h(x,\xi) = h^{-1+\epsilon} a_h(x,h\xi)\}_{h\in(0,1]} $ is bounded in $ \Gamma^{1-\epsilon,r} $.
By \cite[\S2.4 (2.4.6)]{Taylor91pseudodiff-nonlinearPDE}, as $ r \ge \epsilon^{-1}-1 $, we have $ 1-\epsilon \le r/(1+r) $ and thus
\begin{equation*}
\Re(\op(b_h) u,u)_{L^2} \gtrsim -\|u\|_{L^2}^2.
\end{equation*}
We conclude by $ \op(b_h) = h^{-1+\epsilon} \op_h(a_h) $.
\end{proof}
We are mostly interested in the case where $ \epsilon = 1/2 $.
In this case, the condition for~$ r $ is simply $ r \ge \max\{\tilde{d},(1/2)^{-1}-1\} = \tilde{d} $.
Next we show that the almost sharp G{\aa}rding inequality also applies to symbols in $ \Upsilon^{1+r} $.

\begin{lemma}
\label{lem::Garding-paradiff}
If $ \epsilon \in (0,1) $ and $ a_h \in M_{n\times n}(\Upsilon^{1+r}) $ with $ n \in \N $, $ r \ge \max\{\tilde{d},\epsilon^{-1}-1\}$, then there exists $ C > 0 $ such that for all $ u \in L^2 $,
\begin{equation*}
\Re(T_{a_h}^h u,u)_{\Ltwo}  \ge -C h^{1-\epsilon} \|u\|_{\Ltwo}^2, \quad
\Re(\P{a_h}^h u,u)_{\Ltwo}  \ge -C h^{1-\epsilon} \|u\|_{\Ltwo}^2.
\end{equation*}
\end{lemma}
\begin{proof}
Choose a sequence $ a_h^j \in M_{n\times n}(\Gamma^{0,r}) $ which converge to $ a_h $ with respect to the norm $ N^r(\cdot) $ and are uniformly compactly supported in $ \R^d \times (\Rd\backslash 0) $.
Apply the almost sharp G{\aa}rding inequality for $ a_h^j $, there exists a constant $ C > 0 $ which is independent of~$ j $, such that for all $ u \in L^2 $, we have
\begin{equation*}
\Re(T_{a_h}^h u,u)_{L^2} 
= \Re(T_{a_h-a_h^j}^h u,u)_{L^2}
+ \Re(T_{a_h^j}^h u,u)_{L^2} \ge o(1)-C h^{1-\epsilon} \|u\|_{L^2}^2.
\end{equation*}
We conclude the almost sharp G{\aa}rding inequality for $ T_{a_h}^h $ by passing $ j \to \infty $.
Therefore,
\begin{align*}
\Re(\P{a_h}^h u,u)_{\Ltwo}
& = \sum_{j \in \N} \Re( \upsi_j \T{\psi_j a_h}^h \upsi_j u, u )_{\Ltwo}
= \sum_{j \in \N} \Re(  \T{\psi_j a_h}^h \upsi_j u, \upsi_j u )_{\Ltwo} \\
& \gtrsim -h^{1-\epsilon} \sum_{j \in \N} \|\upsi_j u\|_{\Ltwo}^2
\gtrsim -h^{1-\epsilon} \|u\|_{\Ltwo}^2.\qedhere
\end{align*}
\end{proof}

\subsubsection{Relation with quasi-homogeneous wavefront sets}

\begin{lemma}
\label{lem::WF-characterization-paradiff}
If $ r \ge 0 $ and $ a_h = \sum_{j=0}^{r} h^j a_h^{j} $ where $ a_h^{j} \in \Upsilon^{1+r-j} $ such that $ a_h $ is elliptic at $ (x_0,\xi_0) \in \Rd \times (\Rd \backslash 0) $ in the sense that, for some neighborhood~$ \Omega $ of $ (x_0,\xi_0) $, we have
\begin{equation*}
\inf_{0 < h < 1} \inf_{(x,\xi)\in\Omega} |a_h(x,\xi)| > 0,
\end{equation*}
then for all $ u \in \Ltwo $ such that $ \T{a_h}^h u = \O(h^\sigma)_{\Ltwo} $ where $ 0 \le \sigma \le r $, we have $ (x_0,\xi_0) \not \in \WF_{0,1}^\sigma(u). $
\end{lemma}
\begin{proof}
Assume that $ \Omega \subset \Rd \times (\Rd \backslash 0) $. Let $ b_h \in S^{-\infty}_{-\infty} $ with $ \supp b_h \subset \Omega $. 
Then by the symbolic calculus stated in Corollary~\ref{cor:symbolic-calculus-limited-xi-regularity}, there exists $ c_h = \sum_{j=0}^r h^j c_h^j $ where $ c_h^j \in \Upsilon^{1+r-j} $, such that 
\begin{equation*}
\T{b_h}^h = \T{c_h}^h \T{a_h}^h + \O(h^r)_{\Ltwo\to\Ltwo}.
\end{equation*}
Therefore $ \T{b_h}^h u = \O(h^\sigma)_{\Ltwo} $.
By Lemma~\ref{lem::un-paradiff-T-h} we have $ \op_h(b_h) u = \O(h^\sigma)_{\Ltwo}. $
We conclude by Lemma~\ref{lem::support=decay}.
\end{proof}

\begin{lemma}
\label{lem::WF-paradiff-elliptic}
Let $ \epsilon \ge 0 $, $ e \in \Gamma^{m,r}_{0,0} $ (if $ \epsilon= 0 $) resp.\ $ \Gamma^{m,r}_{0,1} $ (if $ \epsilon > 0 $), and suppose that~$ e $ is homogeneous of degree~$ m $ with respect to~$ \xi $. Then for $ f \in \H{s} $ and $ 0 \le \sigma \le (1+\epsilon) r $,
\begin{equation*}
\WF_{\epsilon,1}^{s+\sigma-m}(\P{e} f)^\circ \subset \WF_{\epsilon,1}^{s+\sigma}(f)^\circ.
\end{equation*}
If in addition $ e $ is elliptic, i.e., for some $ C > 0 $ and $ |\xi| $ sufficiently large, $ |e(x,\xi)| \ge C|\xi|^m, $ then
\begin{align*}
\WF_{\epsilon,1}^{s+\sigma-m}(\P{e} f)^\circ = \WF_{\epsilon,1}^{s+\sigma}(f)^\circ.
\end{align*}
\end{lemma}
\begin{proof}
For $ \mu \in \R $, denote $ Z^\mu = \P{|\xi|^\mu} $. Then $ Z^{-\mu} Z^\mu - \Id \in \OP{-\infty}{-\infty}. $ Therefore, 
\begin{equation*}
f - Z^{-s} Z^s f \in \H{\infty}_{\infty}, \quad \P{e} f  - \P{e \sharp |\xi|^{-s}} Z^s f \in \H{s+r-m}_{\delta r} + H^\infty, 
\end{equation*}
where $ \delta = 0 $ if $ \epsilon = 0 $, while $ \delta = 1 $ if $ \epsilon > 0 $. By Lemma~\ref{lem::basic-properties-WF} and the fact that $ Z^{\pm s} $ are pseudodifferential operators with elliptic symbols in $ S^{\pm s}_0 $, we readily have
\begin{align*}
\WF^{s+\sigma}_{\epsilon,1}(f)^\circ = \WF^\sigma_{\epsilon,1}(Z^s f)^\circ, \quad
\WF^{s+\sigma-m}_{\epsilon,1}(\P{e} f)^\circ = \WF^{\sigma - (m-s)}_{\epsilon,1}(\P{e|\xi|^{-s}} Z^s f)^\circ.
\end{align*}
So we may assume that $ s = 0 $. Let $ a, b \in S^{-\infty}_{-\infty} \cap \sigma_\epsilon $, such that 
\begin{equation*}
\supp b \subset \{ a = 1 \} \subset \supp a \subset \R^{2d} \backslash \WF_{\epsilon,1}^{\sigma}(f),  
\end{equation*}
then by Lemma~\ref{lem::support=decay}, $ \op_h^{\epsilon,1}(a) f = \O(h^{\sigma})_{\Ltwo}. $
By Corollary~\ref{cor::homogenization-semiclassical}, Lemma~\ref{lem::un-paradiff-estimate-smooth-symbol}, Proposition~\ref{prop::composition-P-h}, and Corollary~\ref{cor::mixed-composition-epsilon-0},
\begin{align*}
h^m \op_h^{\epsilon,1}(b) \P{e} f
& = \op_h^{\epsilon,1}(b) \P{e}^h f + \O(h^\infty)_{\Ltwo} \\
& = \op_h^{\epsilon,1}(b) \P{e}^h \op_h^{\epsilon,1}(a) f + \op_h^{\epsilon,1}(b) \P{e}^h \op_h^{\epsilon,1}(1-a) f + \O(h^\infty)_{\Ltwo}  \\
& = \O(1)_{\Ltwo \to \Ltwo} \op_h^{\epsilon,1}(a) f + \O(h^{r(1+\epsilon)})_{\Ltwo}  \\
& = \O(h^{\sigma})_{\Ltwo},
\end{align*}
proving the first statement. The second statement follows by a construction of parametrix.
\end{proof}

\section{Asymptotically flat water waves}

\label{sec::asymptotically-flat-ww}

In this section we prove Theorem~\ref{thm::ww-weighted-sobolev-existence}. The idea is to combine the analysis in \cite{ABZ11capillary} with the dyadic paradifferential calculus in weighted Sobolev spaces. We shall use the following formal notations for simplicity. Let $ w $ be a function on $\mathbb{R}^d$ which is nowhere vanishing, then for any operator $ \A $ between some function spaces on $\mathbb{R}^d$ and for any function $ f $ on $\mathbb{R}^d$, we introduce the following notations whenever they are well-defined:
\begin{equation*}
\A^{(w)} = w \A w^{-1}, \quad f^{(w)} = wf.
\end{equation*}
Note that $ (\A f)^{(w)} = \A^{(w)} f^{(w)}. $ 
For $ k \in \R $, we also denote by an abuse of notation that
\begin{equation*}
\A^{(k)} = \A^{(\jp{x}^k)}, \quad f^{(k)} = f^{(\jp{x}^k)}.
\end{equation*}
Observe that $ \Ltwo_k = \H{0}_k $ is an Hilbert space with the inner product
\begin{equation*}
(f,g)_{\Ltwo_k} = (f^{(k)}, g^{(k)})_{\Ltwo}.
\end{equation*}

\subsection{Dirichlet--Neumann operator}

\label{sec::D-N-op}

We study the Dirichlet--Neumann operator on weighted Sobolev spaces and its paralinearization. The time variable will be temporarily omitted for simplicity.

\subsubsection{Boundary flattening}

Let $ \eta \in \Holder{1}(\Rd) $, such that
\begin{equation}
\label{eq::depth-delta}
\delta \bydef b + \inf_{x \in \Rd} \eta(x) > 0.
\end{equation}
Define $ \tau(x,z) = (x,z + \eta(x)) $ and set
\begin{align*}
\tilde{\Omega} & = \tau^{-1}(\Omega) = \{-b-\eta(x) < z < 0\},\\
\tilde{\Sigma} & = \tau^{-1}(\Sigma) = \{z=0\},\\
\tilde{\Gamma} & = \tau^{-1}(\Gamma) = \{z=-b-\eta(x)\}.
\end{align*}
Let~$ \tau_* $ be the pullback deduced by~$ \tau $, then 
\begin{equation*}
\tau_*(\dx^2+\dy^2) = (\dx\ \dz) \varrho \binom{\dx}{\dz}
\end{equation*}
where 
\begin{equation*}
\varrho = \begin{pmatrix}
\Id + (\nabla\eta) {}^t(\nabla\eta) & \nabla \eta \\ {}^t(\nabla\eta) & 1
\end{pmatrix}.
\end{equation*}
We verify that
\begin{equation*}
\varrho^{-1} = \begin{pmatrix}
\Id & -\nabla \eta \\ -{}^t(\nabla\eta) & 1 + |\nabla \eta|^2
\end{pmatrix}.
\end{equation*}
Let $ \nabla_{xz}=(\nabla_x,\pz) $, then the divergence, gradient and Laplacian with respect to the metric~$ \varrho $ are
\begin{align*}
\div_\varrho u & = \nabla_{xz} \cdot u, \\
\nabla_\varrho u & = \big(\nabla u - \nabla\eta \pz u,-\nabla\eta \cdot \nabla u + (1 + |\nabla\eta|^2) \pz u\big), \\
\Delta_\varrho u &  = \pz^2 u + (\nabla - \nabla\eta \pz)^2 u.
\end{align*}
The exterior unit normal to~$ \partial\tilde{\Omega} = \tilde{\Sigma} \cup \tilde{\Gamma} $ is
\begin{equation*}
\norm_\varrho = \langle (D\tau)^{-1}|_{T\partial\tilde{\Omega}}, \norm\rangle = \begin{cases}
\frac{{}^t(-\nabla\eta, 1 + |\nabla\eta|^2)}{\sqrt{1+|\nabla\eta|^2}}, & \tilde{\Sigma}; \\
{}^t(0,1), & \tilde{\Gamma}.
\end{cases}
\end{equation*}

Let $ \psi \in \H{1/2} $, and suppose that $ \phi $ satisfies the equation
\begin{equation*}
\Delta_{xy} \phi = 0, \quad \phi|_\Sigma = \psi, \quad \pn \phi|_\Gamma = 0,
\end{equation*}
then $ v = (\tau|_{\tilde{\Omega}})_* \phi $ satisfies
\begin{equation}
\label{eq::equation-elliptic-v}
\Delta_\varrho v = 0, \quad 
v|_{\tilde{\Sigma}} = \psi, \quad \partial_{\norm_\varrho} v|_{\tilde{\Gamma}} = 0.
\end{equation}
The Dirichlet--Neumann operator now writes as
\begin{equation*}
\sqrt{1+|\nabla\eta|^2}^{-1} G(\eta) \psi =  \partial_{\norm_\varrho} v|_{\tilde{\Sigma}}
= \norm_\varrho \cdot \nabla_{xz} v|_{z=0}.
\end{equation*}

\subsubsection{Elliptic estimate}
\label{sec:elliptic-estimate}

Let $ \chi_0 \in \Cinf(\R) $ with $ \chi_0(z) = 0 $ for $ z \le -\delta/2 $ and $ \chi_0(z) = 1 $ for $ z \ge 0 $. Decompose $ v = \tilde{v} + \ul{\psi} $, where
\begin{equation*}
\ul{\psi}(x,z) = \chi_0(z) e^{z\jp{D_x}} \psi(x).
\end{equation*}

\begin{lemma}
\label{lem::lift-of-psi}
Let $ n \in \N $, $ m \in \R $, $ \mu \in \R $, $ k \in \R $, $ a \in S^m_0 $, then
\begin{equation*}
\|\pz^n \op(a) \ul{\psi}\|_{\Ltwo_z(\R_{\le 0},\H{\mu-n-m+1/2}_k)} \lesssim \|\psi\|_{\H{\mu}_k}.
\end{equation*}
\end{lemma}
\begin{proof}
We only prove the case where $ n = 0 $. The general case follows with a similar argument and the identity 
\begin{equation*}
\pz^n \ul{\psi}(x,z) = \sum_{j=0}^n \binom{n}{j} \chi_0^{(n-j)}(z) \jp{D_x}^je^{z\jp{D_x}} \psi(x). 
\end{equation*}
Let 
\begin{align*}
b(x,\xi) & = a(x,\xi)\jp{\xi}^{\mu-m} \in S^\mu_0, \\
\lambda(z,\xi) & = \chi_0(z) e^{z\jp{\xi}} \jp{\xi}^{1/2} \in \Linf_{z\le 0} S^{1/2}_0.
\end{align*}
Then for all $ N \ge 0 $,
\begin{align*}
\|\op(a) \ul{\psi}\|_{\Ltwo_z(\R_{\le 0},\H{\mu-m+1/2}_k)}
& \lesssim \|\op(\lambda) \op(b) \psi\|_{\Ltwo_z(\R_{\le 0},\Ltwo_k)} + \|\psi\|_{\H{-N}_k}.
\end{align*}
Observe that
\begin{align*}
\op(\lambda)^{(k)} - \big( \op(\lambda)^{(k)} \big)^* & \in \Linf_{z\le 0}\OP{-1/2}{0}, \\
\big( \op(\lambda)^{(k)} \big)^2 - \op(\lambda^2)^{(k)} & \in \Linf_{z\le 0}\OP{0}{0}.
\end{align*}
Also note that
\begin{equation*}
\sigma(\xi) \bydef \int_{-\infty}^0 \lambda^2(z,\xi) \dz
=  \jp{\xi} \int_{-\infty}^{0} \chi_0^2(z) e^{2\jp{\xi}z} \dz
\in S^{0}_0.
\end{equation*}
Therefore,
\begin{align*}
\|\op(\lambda) \op(b) \psi\|_{\Ltwo_z(\R_{\le 0},\Ltwo_k)}^2 
& = \big( \op(\lambda^2) \op(b) \psi, \op(b) \psi \big)_{\Ltwo_z(\R_{\le 0},\Ltwo_k)} 
+ \O(\|\psi\|_{\H{\mu}_k}^2) \\
& = \big( \op(\sigma) \op(b) \psi, \op(b) \psi \big)_{\Ltwo_k} + \O(\|\psi\|_{\H{\mu}_k}^2)
= \O(\|\psi\|_{\H{\mu}_k}^2).\qedhere
\end{align*}
\end{proof}

\begin{lemma}
\label{lem::elliptic-estimate-v}
For all $ k \in \R $, $ \|\tilde{v}\|_{\H{1}_k} \le C(\|\eta\|_{\Holder{1}}) \|\psi\|_{\H{1/2}_k}. $
\end{lemma}
\begin{proof}
Let~$ \H{1,0}_\varrho $ be the completion of the space 
\begin{equation*}
\{ f \in \Cinf(\tilde{\Omega}) : f\mathrm{\ vanishes\ in\ a\ neighborhood\ of\ } \tilde{\Sigma} \}
\end{equation*}
with respect to the norm
\begin{equation*}
\|u\|_{\H{1,0}_\varrho} \bydef \|\nabla_\varrho u\|_{\Ltwo_\varrho} = (\nabla_\varrho u,\nabla_\varrho u)_{\Ltwo_\varrho}^{1/2},
\end{equation*}
where $ (X,Y)_{\Ltwo_\varrho} \bydef \int_{\tilde{\Omega}} \varrho(X,Y) \dx\dz. $ As $ b < \infty $, by Poincaré inequality, 
\begin{equation*}
\|u\|_{\Ltwo} \le C(\|\eta\|_{\Linf}) \|\pz u\|_{\Ltwo} \le C(\|\eta\|_{\Holder{1}}) \|u\|_{\H{1,0}_\varrho}
\end{equation*}
for all $ u \in \H{1,0}_\varrho $.
Let $ 0 < \zeta \in \Cinf(\R) $ be such that $ \zeta(z) = 1 $ for $ |z| \le 1 $, and $ \zeta(z) = z $ for $ |z| \ge 2 $. For some $ R > 0 $ sufficiently large to be determined later, set $ w(x) = R \times \zeta\big(\jp{x}^k / R\big). $ Then $ \jp{x}^k \lesssim w(x) \lesssim R \jp{x}^k $, $ \supp \nabla w \subset \{\jp{x} \gtrsim R^{1/k}\} $, and $ |\nabla w(x)| \lesssim R^{(k-1)/k}  $.

As $ \tilde{v} $ satisfies the equation $ \Delta_\varrho \tilde{v} = - \Delta_\varrho \upsi, $ we consider~$ \tilde{v}^{(w)} $ as the variational solution to the equation 
$ B\big(\tilde{v}^{(w)},\cdot\big) = - L(\cdot), $
where for $ u, \varphi \in \H{1,0}_\varrho $,
\begin{equation*}
B(u,\varphi) = \big( \nabla_\varrho^{(w)} u, \nabla_\varrho^{(1/w)} \varphi\big)_{\Ltwo(\tilde{\Omega})}, \quad
L(\varphi) = \big( \nabla_\varrho^{(w)} \ul{\psi}^{(w)}, \nabla_\varrho^{(1/w)} \varphi\big)_{\Ltwo(\tilde{\Omega})}. 
\end{equation*}
Observe that $ \nabla_\varrho^{(w^{\pm 1})} = \nabla_\varrho \mp b_w, $ where $ b_w = (w^{-1} \nabla w, -\nabla\eta \cdot w^{-1} \nabla w ) \in \Linf $ satisfies 
$ \|b_w\| \le C(\|\eta\|_{\Holder{1}}) R^{-1/k}. $
We verify that~$ L $ and~$ B $ are continuous linear and bilinear forms on~$ \H{1,0}_\varrho $. Moreover~$ B $ is coercive when~$ R $ is sufficiently large, indeed,
\begin{equation}
\label{eq:coercivity-B}
B(\varphi,\varphi) = \|\nabla_\varrho \varphi\|_{\Ltwo_\varrho}^2 - \|b_w \varphi\|_{\Ltwo_\varrho}^2
\ge \big(1 - C(\|\eta\|_{\Holder{1}})R^{-2/k} \big) \|\nabla_\varrho \varphi\|_{\Ltwo_\varrho}^2.
\end{equation}
Therefore, by Lax-Milgram's Theorem and Lemma~\ref{lem::lift-of-psi},
\begin{equation*}
\|\tilde{v}\|_{\H{1}_k} \lesssim \|\tilde{v}^{(w)}\|_{\H{1,0}_\varrho} \lesssim \|L\|_{(\H{1,0}_\varrho)^*}
\lesssim \|\ul{\psi}\|_{\H{1}} \lesssim \|\psi\|_{\H{1/2}_k}.\qedhere
\end{equation*}
\end{proof}

\begin{proposition}
\label{prop::G(eta)-boundedness-Ltwo}
Let $ (\eta,\psi) \in \Holder{1} \times \H{1/2}_k $, $ k \in \R $, then
$ \|G(\eta) \psi\|_{\H{-1/2}_k} \le C(\|\eta\|_{\Holder{1}}) \|\psi\|_{\H{1/2}_k}. $
\end{proposition}
\begin{proof}
By Lemma~\ref{lem::lift-of-psi} and Lemma~\ref{lem::elliptic-estimate-v}, 
$ v \in \Ltwo_z(]{-\delta},0[,\H{1}_k) \cap \H{1}_z(]{-\delta},0[,\Ltwo_k). $
By a classical interpolation result (see, e.g., \cite[Lemma~2.19]{ABZ14gravity}) and the equation satisfied by~$ v $, we deduce that
\begin{equation*}
v \in C^0_z([-\delta,0],\H{1/2}_k) \cap C^1_z([-\delta,0],\H{-1/2}_k).\qedhere
\end{equation*}
\end{proof}

\subsubsection{Higher regularity}

\begin{proposition}
\label{prop::D-N-higher-regularity}
Let $ (\eta,\psi) \in \H{\mu+1/2} \times \H{\sigma+1/2}_k $ where $ k \in \R $, $  \mu > 1/2+d/2 $, $ 0 \le \sigma \le [\mu- 1/2] $ then
\begin{equation*}
\|G(\eta) \psi\|_{\H{\sigma-1/2}_k} \le C(\|\eta\|_{\H{\mu+1/2}}) \|\psi\|_{\H{\sigma+1/2}_k}.
\end{equation*}
Consequently, if $ (\eta,\psi) \in \H{\mu+1/2} \times \HC{\sigma+1/2,\delta}_{k} $ with $ \delta \ge 0 $, $ k \in \N $ and $ \sigma - k\delta \ge 0 $, then
\begin{equation*}
\|G(\eta) \psi\|_{\HC{\sigma-1/2,\delta}_{k}} \le C(\|\eta\|_{\H{\mu+1/2}}) \|\psi\|_{\HC{\sigma+1/2,\delta}_{k}}.
\end{equation*}
\end{proposition}
\begin{proof}
We shall only prove the cases where $ \sigma \in \N $.
The remaining cases follow by interpolation.
By \S\ref{sec:elliptic-estimate}, it suffices to prove that for all $ \sigma \in [0,\mu-1/2] \cap \N $, there exists $  \delta > 0 $ such that
\begin{equation*}
\tilde{v} \in \Ltwo(]{-\delta},0[,\H{\sigma+1}_k) \cap \H{1}(]{-\delta},0[,\H{\sigma}_k).
\end{equation*}
Let $ N_\sigma $ be the corresponding norm of $ \tilde{v} $, we shall prove that $ N_\sigma < +\infty $.
The case where $ \sigma=0 $ has already been proven by Lemma~\ref{lem::elliptic-estimate-v}.
It remains to bound $ N_{\sigma+1} $ by $ N_\sigma $ via a mathematical induction.
Note that if $ \chi \in \Ccinf(]-\delta,\delta[) $, then $ \chi \px^{\sigma} \tilde{v} $ satisfies the equation
\begin{equation}
\label{eq:elliptic-estimate}
-\Delta_\varrho (\chi \px^{\sigma} \tilde{v}) + K \tilde{v}
= \Delta_\varrho (\chi \px^{\sigma} \ul{\psi}) - K \ul{\psi}.
\end{equation}
where $ K = [\Delta_\varrho,\chi \px^{\sigma}] $.
Note that $ \Delta_\varrho = P \cdot P $ with $ P = (\nabla - \nabla\eta\pz,\pz) $,
so 
\begin{equation*}
K = P \cdot [P,\chi \px^{\sigma}] + [P,\chi \px^{\sigma}] \cdot P.
\end{equation*}
By an explicit calculation
\begin{equation*}
[P,\chi \px^{\sigma}] 
= (-\chi[\nabla\eta, \px^{\sigma}]\pz - \nabla\eta \chi' \px^{\sigma},\chi'\px^{\sigma}).
\end{equation*}
Integrating the following pairings by part using $ \tilde{v}|_{z=0} = 0 $, we have by Lemma~\ref{lem::lift-of-psi} that
\begin{align*}
\big| \big(\chi\px^\sigma\tilde{v}, K\chi\px^\sigma\tilde{v}\big)_{L^2L^2_k} \big|
& \lesssim \|P^*(\chi\px^\sigma\tilde{v})\|_{L^2L^2_k} \|[P,\chi \px^\sigma] \chi\tilde{v}\|_{L^2L^2_k} \\
& \qquad + \|P(\chi\px^\sigma\tilde{v})\|_{L^2L^2_k} \|[P,\chi \px^\sigma]^* \chi\tilde{v}\|_{L^2L^2_k}
\lesssim N_\sigma N_{\sigma+1};\\
\big| \big(\chi\px^\sigma\tilde{v}, K\chi\px^\sigma\ul{\psi}\big)_{L^2L^2_k} \big|
& \lesssim \|P^*(\chi\px^\sigma\tilde{v})\|_{L^2L^2_k} \|[P,\chi \px^\sigma] \chi\ul{\psi}\|_{L^2L^2_k} \\
& \qquad + \|P(\chi\px^\sigma\ul{\psi})\|_{L^2L^2_k} \|[P,\chi \px^\sigma]^* \chi\tilde{v}\|_{L^2L^2_k}
\lesssim \|\psi\|_{H^{\sigma+1/2}} (N_\sigma+N_{\sigma+1}),\\
\big| \big(\chi \px^\sigma \tilde{v}, -\Delta_\varrho (\chi \px^\sigma \ul{\psi})\big)_{L^2 L^2_k}  \big|
& \lesssim \|P(\chi \px^\sigma \tilde{v})\|_{L^2L^2_k} \|P(\chi \px^\sigma \ul{\psi})\|_{L^2L^2_k}^2
\lesssim \|\psi\|_{H^{\sigma+1/2}} N_{\sigma+1}.
\end{align*}
In the above inequalities, the adjoint operators are taken with respect to $ L^2 L^2_k $.
Use again the structure of $ \Delta_\varrho $, we have by~\eqref{eq:coercivity-B} that
\begin{align*}
\big(\chi \px^\sigma \tilde{v}, -\Delta_\varrho (\chi \px^\sigma \tilde{v})\big)_{L^2 L^2_k} 
& \gtrsim \|P(\chi \px^\sigma \tilde{v})\|_{L^2L^2_k}^2 - \|\chi \px^\sigma \tilde{v}\|_{L^2L^2_k}^2
\gtrsim \|\chi \px^\sigma \tilde{v}\|_{H^1H^1_k}^2 - N_\sigma^2.
\end{align*}
Pair~\eqref{eq:elliptic-estimate} with $ \chi \px^\sigma \tilde{v} $ and use the estimates above, for all $ \epsilon > 0 $,
\begin{equation*}
N_{\sigma+1}^2 \lesssim
\|\chi \px^\sigma \tilde{v}\|_{H^1H^1_k}^2 \lesssim N_\sigma N_{\sigma+1} + \|\psi\|_{H^{\sigma+1/2}}(N_\sigma + N_{\sigma+1})
\lesssim \epsilon N_{\sigma+1}^2 + \epsilon^{-1} (N_{\sigma}^2+\|\psi\|_{H^{\sigma+1/2}}^2).
\end{equation*}
All the constants hidden by $ \lesssim $ are of the form $ C(\|\eta\|_{H^{\mu+1/2}}) $.
We thus conclude the induction by choosing $ \epsilon > 0 $ sufficiently small.
By interpolation as in Proposition~\ref{prop::G(eta)-boundedness-Ltwo},
\begin{equation*}
v \in C^0_z([-\delta,0],\H{\sigma+1/2}_k) \cap C^1_z([-\delta,0],\H{\sigma-1/2}_k).
\end{equation*}
When $ \psi \in \HC{\sigma,\delta}_k $, then we apply the above estimate to $ \psi \in H^{\sigma-\delta j}_j $ and conclude.
\end{proof}

\subsection{Paralinearization}

Now we paralinearize the system of water waves. The following results are immediate consequences of the analysis in~\cite{ABZ11capillary} and our dyadic paradifferential calculus on weighted Sobolev spaces. 

\begin{proposition}
\label{prop::paralinearization-G(eta)psi-weighted}
Let $ (\eta,\psi) \in \HC{\mu+1/2,\delta}_{k} \times \HC{\mu,\delta}_{k} $ with $ \mu-1/2 \in \N $, $ k \in \N $ and $ \mu - \delta k > 3+d/2 $.
Let
\begin{equation*}
B  = \frac{\nabla\eta \cdot \nabla\psi + G(\eta)\psi}{1+|\nabla\eta|^2}, \quad
V  = \nabla\psi -  B \nabla\eta,
\end{equation*}
and $ \lambda = \lambda^{(1)} + \lambda^{(0)} \in \Gamma^{3/2,\mu-1/2-\tilde{d}}_{0,0}+\Gamma^{1/2,\mu-3/2-\tilde{d}}_{0,0} $, where
\begin{align*}
\lambda^{(1)}(x,\xi) & = \sqrt{ (1+|\nabla\eta|^2) |\xi|^2 - (\nabla\eta \cdot \xi)^2 }, \\
\lambda^{(0)}(x,\xi) & = \frac{1+|\nabla\eta|^2}{2 \lambda^{(1)}} \big\{ \nabla \cdot \big( \alpha^{(1)} \nabla\eta \big) + i \pxi \lambda^{(1)} \cdot \nabla \alpha^{(1)}  \big\}, 
\end{align*}
and $ \alpha^{(1)}(x,\xi) = \frac{\lambda^{(1)} + i \nabla\eta \cdot \xi}{1+|\nabla\eta|^2}, $ then 
\begin{equation*}
G(\eta)\psi = \P{\lambda}(\psi-\P{B}\eta) - \P{V} \cdot \nabla\eta + R(\eta,\psi),
\end{equation*}
where $ R(\eta,\psi) \in \HC{\mu + 1/2,\delta}_{k} $.
\end{proposition}
We shall denote $ \omega = \psi - \P{B} \eta $, which is called the good unknown of Alinhac.
\begin{proof}
We only sketch the proof, for the key ingredients are already given in \cite{ABZ11capillary}. 
We simply replace the paradifferential calculus in~\cite{ABZ11capillary} by our dyadic paradifferential calculus.
Let~$ v $ be defined as in~\S\ref{sec::D-N-op}.
Rewrite~\eqref{eq::equation-elliptic-v} as
\begin{equation*}
\alpha \pz^2 v + \Delta v + \beta \cdot \nabla \pz v - \gamma \pz v = 0.
\end{equation*}
where $ \alpha = 1+|\nabla\eta|^2 $, $ \beta = -2\nabla\eta $, $ \gamma = \Delta\eta $.
Applying Proposition~\ref{prop::paralinearization-ab-weighted}, we obtain as in~\cite[Lemma~3.17]{ABZ11capillary},
\begin{equation}
\label{eq:equation-v-paralinearized}
\P{\alpha} \pz^2 u + \Delta u + \P{\beta} \cdot \nabla \pz u - \P{\gamma} \pz u \in C([-\delta,0],\HC{\mu,\delta}_k),
\end{equation}
where $ u = v - \P{\pz v} \zeta $ with $ \zeta(x,z) = z+\eta(x) $.
Define $ a_\pm = a_\pm^{(1)} + a_\pm^{(0)} \in \Gamma^{1,\mu-1/2-\tilde{d}}_{0,0} + \Gamma^{0,\mu-3/2-\tilde{d}}_{0,0} $~by
\begin{align*}
a_\pm^{(1)}(x,\xi) & = {1 \over 2\alpha} \Big( -\beta \cdot \xi \pm \sqrt{4\alpha|\xi|^2 -(\beta\cdot\xi)^2} \Big), \\
a_\pm^{(0)}(x,\xi) & = \pm {1 \over a_-^{(1)} - a_+^{(1)}} \Big( i\pxi a_-^{(1)} \cdot \px a_+^{(1)} - {\gamma \over \alpha} a_\pm^{(1)} \Big),
\end{align*}
then we factorize~\eqref{eq:equation-v-paralinearized} as
\begin{equation*}
\P{\alpha}(\pz - \P{a_-})(\pz-\P{a_+}) u \in C([-\delta,0],\HC{\mu,\delta}_k).
\end{equation*}
Because $ \Re\, a_-^{(1)} \le 0 $, a parabolic estimate (see e.g., \cite[Proposition~3.19]{ABZ11capillary}) implies that
\begin{equation*}
(\pz u - \P{a_+} u)|_{z=0} \in \HC{\mu+1/2,\delta}_k.
\end{equation*}
We conclude by setting 
$ \lambda = (1+|\nabla\eta|^2) a_+ - i \nabla\eta\cdot\xi. $
\end{proof}

The proofs of the following results are in the same spirit and much simpler.
Their proofs are exactly the same as in~\cite{ABZ11capillary}, simply replacing the usual paradifferential calculus with our dyadic paradifferential calculus, particularly the Propositions~\ref{prop::paralinearization-weighted} and~\ref{prop::paralinearization-ab-weighted}.
Therefore we shall omit the proofs.

\begin{proposition}
\label{prop::paralinearization-surface-tension-weighted}
Let $ \eta \in \HC{\mu+1/2,\delta}_{k} $ with $ \mu - 1/2 \in \N $, $ \mu - \delta k > 3 +  d/2 $ and define 
$ \ell = \ell^{(2)} + \ell^{(1)} \in \Gamma^{2,\mu-1/2-\tilde{d}}_{0,0} + \Gamma^{1,\mu-3/2-\tilde{d}}_{0,0} $ where
\begin{equation*}
\ell^{(2)} = \frac{(1+|\nabla\eta|^2)|\xi|^2 - (\nabla\eta \cdot \xi)^2}{(1+|\nabla\eta|^2)^{3/2}}, \quad
\ell^{(1)} = \frac{1}{2} \pxi \cdot D_x \ell^{(2)},
\end{equation*}
then
$ H(\eta) = -\P{\ell} \eta + f(\eta), $
where $ f(\eta) \in \HC{2\mu-2-d/2,2\delta}_{k} $.
\end{proposition}

\begin{proposition}
\label{prop::paralinearization-other-terms-weighted}
Let $ (\eta,\psi) \in \HC{\mu+1/2,\delta}_{k} \times \HC{\mu,\delta}_{k} $, with $ \mu-1/2 \in \N $, $ \mu - \delta k > 3 +  d/2 $, then
\begin{equation*}
\frac{1}{2} |\nabla\psi| - \frac{1}{2} \frac{(\nabla\eta \cdot \nabla\psi + G(\eta) \psi)^2}{1+|\nabla\eta|^2}
= \P{V} \cdot \nabla\psi - \P{B} \P{V} \cdot \nabla\eta - \P{B} G(\eta) \psi + f(\eta,\psi),
\end{equation*}
where $ f(\eta,\psi) \in \HC{2\mu-2-d/2,2\delta}_{k} $.
\end{proposition}

Note that in the above paralinearization results, we do not use the spatial decay of the symbols, as we only require the symbols to be in the classes $ \Gamma^{m,r}_{0,0} $.
These results will only be used in the proof of the Cauchy theory, where the spatial decay of the symbols is not important.
Later when we study the propagation of singularities, we will heavily use the spatial decay of the symbols.

Combining Propositions~\ref{prop::paralinearization-G(eta)psi-weighted}, \ref{prop::paralinearization-surface-tension-weighted} and~\ref{prop::paralinearization-other-terms-weighted}, we obtain the paralinearization of the water wave system.

\begin{proposition}
\label{prop::paralinearization-WW}
Let $ (\eta,\psi) \in \HC{\mu+1/2,\delta}_{k} \times \HC{\mu,\delta}_{k} $, with $ \mu-1/2 \in \N $, $ \mu - \delta k > 3 +  d/2 $, then $ (\eta,\psi) $ solves the water wave equation if and only if
\begin{equation*}
(\pt + \P{V}\cdot\nabla + \L)\binom{\eta}{\psi} = f(\eta,\psi)
\end{equation*}
where
\begin{equation*}
\L = Q^{-1} \begin{pmatrix} 0 & - \P{\lambda} \\ \P{\ell} & 0 \end{pmatrix} Q, 
\quad \mathrm{with} \quad
Q = \begin{pmatrix} \Id & 0 \\ -\P{B} & \P{\lambda} \end{pmatrix},
\end{equation*}
and $ \displaystyle f(\eta,\psi) = Q^{-1} \binom{f_1}{f_2} \in \HC{\mu+1/2}_{k} \times \HC{\mu}_{k} $ is defined by
\begin{align*}
f_1 & = G(\eta) \psi - \{ \P{\lambda} (\psi - \P{B}\eta) - \P{V} \cdot \nabla \eta \},\\
f_2 & = -\frac{1}{2} |\nabla\psi|^2 + \frac{1}{2} \frac{(\nabla\eta \cdot \nabla\psi + G(\eta)\psi)^2}{1+|\nabla\eta|^2} + H(\eta) \\
& \quad \quad + \P{V} \cdot \nabla\psi - \P{B} \P{V} \cdot \nabla\eta - \P{B} G(\eta)\psi + \P{\ell} \eta - g\eta.
\end{align*}
\end{proposition}

\subsection{Symmetrization}
\label{sec::symmetrization}

\begin{definition}
For $ T > 0 $, $ \gamma \in \R $ and two operators~$ \A, \B \in \Linf([0,T],\OP{\gamma}{0}) $, we say that $ \A \sim_\gamma \B $, or simply $ \A \sim \B $ when there is no ambiguity of the choice of~$ \gamma $, if 
\begin{equation*}
\A - \B \in \Linf([0,T],\OP{\gamma-3/2}{0}).
\end{equation*}
\end{definition}

By \cite{ABZ11capillary}, there exists symbols which depend solely on $ \eta $, 
\begin{equation*}
\gamma  = \gamma^{(3/2)} + \gamma^{(1/2)}, \quad
p  = p^{(1/2)} + p^{(-1/2)}, \quad
q  = q^{(0)},
\end{equation*}
whose principal symbols being explicitly
\begin{equation*}
\gamma^{(3/2)} = \sqrt{\ell^{(2)} \lambda^{(1)}}, \quad
p^{(1/2)} = (1+|\nabla\eta|^2)^{-1/2} \sqrt{\lambda^{(1)}}, \quad
q^{(0)} = (1+|\nabla\eta|^2)^{1/4}, 
\end{equation*}
such that
\begin{equation}
\label{eq:symbol-symmetrization}
\P{p} \P{\lambda} \sim_{\frac{3}{2}} \P{\gamma} \P{q}, \quad
\P{q} \P{\ell} \sim_2 \P{\gamma} \P{p}, \quad
\P{\gamma} \sim_{\frac{3}{2}} (\P{\gamma})^*.
\end{equation} 
Define the symmetrizer
\begin{equation*}
S = \begin{pmatrix} \P{p} & 0 \\ 0 & \P{q} \end{pmatrix} Q,
\end{equation*}
then the first two relations in~\eqref{eq:symbol-symmetrization} can be rephrased as
\begin{equation}
\label{eq::symmetrization}
S \L \sim \begin{pmatrix} 0 & -\P{\gamma} \\ \P{\gamma} & 0 \end{pmatrix} S.
\end{equation}
where the equivalence relation $ \sim $ is applied separately to each component of the matrices.

\subsection{Approximate system}

Set the mollifier $ J_\varepsilon = \P{j_\varepsilon} $ where $ j_\varepsilon = j_\varepsilon^{(0)} + j_\varepsilon^{(-1)} $,
\begin{equation*}
j_\varepsilon^{(0)} = \exp(-\varepsilon \gamma^{(3/2)}), \quad
j_\varepsilon^{(-1)} = \frac{1}{2} \pxi \cdot D_x j_\varepsilon^{(0)}.
\end{equation*}
Then uniformly for $ \varepsilon > 0 $, we have
\begin{equation*}
J_\varepsilon \P{\gamma} \sim_{\frac{3}{2}} \P{\gamma} J_\varepsilon,\quad
J_\varepsilon^* \sim_0 J_\varepsilon. 
\end{equation*}
Let $ \tilde{p} = \tilde{p}^{(-1/2)} + \tilde{p}^{(-3/2)} $ with
\begin{align*}
\tilde{p}^{(-1/2)} &= 1/p^{(1/2)}, \\
\tilde{p}^{(-3/2)} &= -\big( \tilde{p}^{(-1/2)}p^{(-1/2)} + \frac{1}{i}\pxi\tilde{p}^{(-1/2)} \cdot \px p^{(1/2)} \big) / p^{(1/2)},
\end{align*}
then we have
\begin{equation*}
\P{p} \P{\tilde{p}} \sim_0 \Id, \quad
\P{q} \P{1/q} \sim_0 \Id.
\end{equation*}
Let
\begin{equation*}
\L_\varepsilon = \L Q^{-1} \begin{pmatrix} \P{\tilde{p}} J_\varepsilon \P{p} & 0 \\ 0 & \P{1/q} J_\varepsilon \P{q} \end{pmatrix} Q,
\end{equation*}
then as in~\eqref{eq::symmetrization} we have
\begin{equation}
\label{eq::symmetrization-L-epsilon}
S\L_\varepsilon \sim \begin{pmatrix}
0 & -\P{\gamma} \\ \P{\gamma} & 0
\end{pmatrix}
J_\varepsilon S.
\end{equation}
We define the approximate system
\begin{equation}
\label{eq:equation-ww-epsilon}
(\pt + \P{V}\cdot\nabla J_\varepsilon + \L_\varepsilon) \binom{\eta}{\psi} = f(J_\varepsilon\eta, J_\varepsilon\psi).
\end{equation}

\subsection{A priori estimate}

From now on we restrict ourselves to the case where $ \delta = 1/2 $.
The weighted Sobolev spaces $ \HC{\mu+1/2,1/2}_k \times \HC{\mu,1/2}_k $ are the spaces where we do the energy estimates.

\begin{proposition}
\label{prop::a-priori-estimate}
Let $ (\eta,\psi) \in C^1([0,T],\HC{\mu+1/2,1/2}_{k} \times \HC{\mu,1/2}_{k}) $ with $ \mu-1/2 \in \N $, $ \mu - k/2 > 3 + d/2 $ solve the approximate system~\eqref{eq:equation-ww-epsilon}.
Define
\begin{equation*}
M_T = \sup_{0\le t \le T} \|(\eta,\psi)(t)\|_{\HC{\mu+1/2}_{k} \times \HC{\mu}_{k}}, \quad
M_0 = \|(\eta,\psi)(0)\|_{\HC{\mu+1/2}_{k} \times \HC{\mu}_{k}}.
\end{equation*}
Then there exists some non-decaying function~$ C : \R_{\ge 0} \to \R_{\ge 0} $ such that
\begin{equation*}
M_T \le C(M_0) + TC(M_T).
\end{equation*}
\end{proposition}
\begin{proof}
For $ 0 \le j \le k $, set
\begin{align*}
M^j_T & = \sup_{0\le t \le T} \|(\eta,\psi)(t)\|_{\H{\mu+1/2-j/2}_j \times \H{\mu-j/2}_j}, \\
M^j_0 & = \|(\eta,\psi)(0)\|_{\H{\mu+1/2-j/2}_{j} \times \H{\mu-j/2}_{j}}.
\end{align*}
By~\cite{ABZ11capillary}, we know
\begin{equation*}
M^0_T \le C(M^0_0) + TC(M^0_T).
\end{equation*} 
It remains to prove that for $ 1 \le j \le k $, we have
\begin{equation*}
M^j_T \le C(M^j_0) + TC(M_T).
\end{equation*}
To do this, let $ \Lambda^\mu_j = \P{m^{\mu-j/2}_{j}} $, and set 
\begin{equation*}
\displaystyle \Phi = \Lambda^\mu_j S \binom{\eta}{\psi}.
\end{equation*}
Then
\begin{equation*}
(\pt + \P{V}\cdot\nabla J_\varepsilon) \Phi + \begin{pmatrix}
0 & - \P{\gamma} \\ \P{\gamma} & 0
\end{pmatrix} J_\varepsilon \Phi = F_\varepsilon
\end{equation*}
where $ F_\varepsilon = F_\varepsilon^1 + F_\varepsilon^2 + F_\varepsilon^3, $ with
\begin{align*}
F_\varepsilon^1 & = \Lambda^\mu_k S f(J_\varepsilon\eta, J_\varepsilon\psi), \\
F_\varepsilon^2 & = [\pt + \P{V}\cdot\nabla J_\varepsilon, \Lambda^\mu_j S] \binom{\eta}{\psi}, \\
F_\varepsilon^3 & = \begin{pmatrix}
0 & - \P{\gamma} \\ \P{\gamma} & 0
\end{pmatrix} J_\varepsilon \Lambda^\mu_j S \binom{\eta}{\psi} - \Lambda^\mu_j S \L_\varepsilon \binom{\eta}{\psi}.
\end{align*}
By Propositions~\ref{prop::paralinearization-WW}, \ref{prop::paralinearization-G(eta)psi-weighted}, \ref{prop::paralinearization-surface-tension-weighted} and~\ref{prop::paralinearization-other-terms-weighted},
\begin{equation*}
\|f(J_\varepsilon\eta, J_\varepsilon\psi)\|_{\HC{\mu+1/2,1/2}_{k} \times \HC{\mu,1/2}_{k}}
\le C\big(\|(J_\varepsilon\eta, J_\varepsilon\psi)\|_{\HC{\mu+1/2,1/2}_{k} \times \HC{\mu,1/2}_{k}}\big) 
\le C\big(\|(\eta, \psi)\|_{\HC{\mu+1/2,1/2}_{k} \times \HC{\mu,1/2}_{k}}\big).
\end{equation*}
Therefore, 
\begin{equation*}
\|F_\varepsilon^1\|_{\Linf([0,T],\Ltwo)} \le C(M_T).
\end{equation*}

As $ \P{V}\cdot\nabla J_\varepsilon $ is a scalar operator, Proposition~\ref{prop::composition-P} gives
\begin{align*}
\| [\pt + \P{V}\cdot\nabla J_\varepsilon, \Lambda^\mu_j S] \|_{\Linf([0,T],\H{\mu+1/2-j/2}_j \times \H{\mu-j/2}_j \to \Ltwo \times \Ltwo)} \le C(M_T),
\end{align*}
which implies 
\begin{equation*}
\|F_\varepsilon^2\|_{\Linf([0,T],\Ltwo)} \le C(M_T).
\end{equation*}

By~\eqref{eq::symmetrization-L-epsilon}, the operator
$ \begin{pmatrix}
0 & - \P{\gamma} \\ \P{\gamma} & 0
\end{pmatrix} J_\varepsilon  S -  S \L_\varepsilon
$ sends $ \H{\mu+1/2} \times \H{\mu} $ to $ \H{\mu}\times\H{\mu} $. 
Unfortunately, 
\begin{align*}
R :\!&= \begin{pmatrix}
0 & - \P{\gamma} \\ \P{\gamma} & 0
\end{pmatrix} J_\varepsilon \Lambda^\mu_k S - \Lambda^\mu_k S \L_\varepsilon  \\
& = \begin{pmatrix}
0 & - \P{\gamma} \\ \P{\gamma} & 0
\end{pmatrix} J_\varepsilon [\Lambda^\mu_k,S]
+ [S \L_\varepsilon,\Lambda^\mu_k] 
+ \bigg( \begin{pmatrix}
0 & - \P{\gamma} \\ \P{\gamma} & 0
\end{pmatrix} J_\varepsilon S - S \L_\varepsilon \bigg) \Lambda^\mu_k \\
& =: (\mathrm{I}) + (\mathrm{II}) + (\mathrm{III})
\end{align*} 
does not send $ \H{\mu+1/2-j/2}_{j} \times \H{\mu-j/2}_{j} $ to $ \Ltwo \times \Ltwo $ because  the sub-principal symbol can not be canceled out in the symbolic calculus, due to the existence of $ \Lambda^\mu_j $. Particularly, we need to use Proposition~\ref{prop::composition-P} to estimate the commutators $ [\Lambda^\mu_j,S] $ and $ [S \L_\varepsilon,\Lambda^\mu_j]  $, and obtain
\begin{equation*}
\Big\|R \binom{\eta}{\psi}\Big\|_{\Ltwo\times\Ltwo}
\lesssim \|(\eta,\psi)\|_{\H{\mu+1-j/2}_{j-1} \times \H{\mu+1/2-j/2}_{j-1}}
+ \|(\eta,\psi)\|_{\H{\mu+1/2-j/2}_{j} \times \H{\mu-j/2}_{j}}.
\end{equation*}
More precisely, the first term on the right hand side comes from (I) and (II) while the second term comes from (III). When $ j \ge 1 $, 
\begin{equation*}
\H{\mu+1-j/2}_{j-1} \times \H{\mu+1/2-j/2}_{j-1}
= \H{\mu+1/2-(j-1)/2}_{j-1} \times \H{\mu-(j-1)/2}_{j-1}
\supset \HC{\mu+1/2,1/2}_k \times \HC{\mu,1/2}_k,
\end{equation*}
and we deduce that 
\begin{equation*}
\|F_\varepsilon^3\|_{\Linf([0,T],\Ltwo)} \le C(M_T).
\end{equation*}

Finally by an exact same energy estimate as in \cite{ABZ11capillary}, we conclude that
\begin{equation*}
M^j_T \lesssim \|\Phi\|_{\Linf([0,T],\Ltwo)} \le C(M^j_0) + T C(M_T).\qedhere
\end{equation*}
\end{proof}

\subsection{Existence}

\begin{lemma}
\label{lem::approximate-Cauchy}
For all $ (\eta_0,\psi_0) \in \HC{\mu+1/2,1/2}_{k} \times \HC{\mu,1/2}_{k} $ where $ \mu - 1/2 \in \N $ and $ \mu - k/2 > 3 + d/2 $, and for all $ \varepsilon > 0 $, the Cauchy problem of the approximate system~\eqref{eq:equation-ww-epsilon} has a unique maximal solution 
\begin{equation*}
(\eta_\varepsilon,\psi_\varepsilon) \in C([0,T_\varepsilon[,\HC{\mu+1/2,1/2}_{k} \times \HC{\mu,1/2}_{k}).
\end{equation*}
Moreover, there exists $ T_0 > 0 $ such that 
\begin{equation*}
\inf_{\varepsilon \in ]0,1]} T_\varepsilon \ge T_0.
\end{equation*}
\end{lemma}
\begin{proof}
Following \cite{ABZ11capillary}, the existence follows from the existence theory of ODEs by writing~\eqref{eq:equation-ww-epsilon} in the compact form 
\begin{equation*}
\pt X = \mathcal{F}_\varepsilon(X),
\end{equation*}
where $ \mathcal{F}_\varepsilon $ is a Lipschitz map on $ \HC{\mu+1/2,1/2}_{k} \times \HC{\mu,1/2}_{k} $. Indeed, $ J_\varepsilon \in \OP{-\infty}{0} $ is a smoothing operator.\footnote{We do not need $ J_\varepsilon \in \OP{-\infty}{-\infty} $ because the operators such as $ \P{V} \cdot \nabla, \L $, etc., are all of non-positive orders with respect to the spatial decay.} The estimates to proving the Lipschitz regularity can be carried out similarly as in the proof of Proposition~\ref{prop::a-priori-estimate}. The only nontrivial term that remains is the Dirichlet--Neumann operator, whose regularity follows by combining Proposition~\ref{prop::D-N-higher-regularity} and the shape derivative formula (which goes back to Zakharov~\cite{Zakharov98shape:derivative},
\begin{equation*}
\langle \d G(\eta) \psi, \varphi \rangle
\bydef \lim_{h \to 0} \frac{1}{h} \big(G(\eta+h\varphi) - G(\eta)\big) \psi 
= -G(\eta)(B\varphi) - \nabla \cdot (V\varphi).
\end{equation*}
A standard abstract argument then shows that $ T_\varepsilon $ has a strictly positive lower bound, we refer to \cite{ABZ11capillary} for more details.
\end{proof}

\begin{proof}[Proof of Theorem~\ref{thm::ww-weighted-sobolev-existence}]

By Lemma~\ref{lem::approximate-Cauchy}, we obtain a sequence $ \{(\eta_\varepsilon,\psi_\varepsilon)\}_{0<\varepsilon\le 1} $ which satisfies the equation~\eqref{eq:equation-ww-epsilon} and is uniformly bounded in $ \Linf([0,T],\HC{\mu+1/2,1/2}_{k} \times \HC{\mu,1/2}_{k}) $ for some $ T > 0 $. By~\eqref{eq:equation-ww-epsilon}, the time derivatives $ \{(\pt\eta_\varepsilon,\pt\psi_\varepsilon)\}_{0<\varepsilon\le 1} $ are uniformly bounded in $ \Linf([0,T],\HC{\mu-1,1/2}_{k} \times \HC{\mu-3/2,1/2}_{k})) $. By~\cite{ABZ11capillary}, there exists 
\begin{equation}
\label{eq::continuity-eta-psi}
(\eta,\psi) \in C([0,T],\H{\mu+1/2}\times\H{\mu})
\end{equation}
which solves~\eqref{eq::equation-water-wave}, such that as $ \varepsilon \to 0 $, we have $ (\eta_\varepsilon,\psi_\varepsilon) \to (\eta,\psi) $ weakly in $ \Ltwo([0,T],\HC{\mu+1/2,1/2}_{k} \times \HC{\mu,1/2}_{k}) $, and strongly in $ C([0,T],\HC{\mu-1,1/2}_{k} \times \HC{\mu-3/2,1/2}_{k}) $. We then prove that for $ 1 \le j \le k $, 
\begin{equation*}
\Phi = \Phi(\eta,\psi) \bydef \Lambda^\mu_j S(\eta,\psi) \binom{\eta}{\psi}
\end{equation*}
lies in $ C([0,T],\Ltwo) $, where $ \Lambda^\mu_j $ is defined in Proposition~\ref{prop::a-priori-estimate}, and $ S = S(\eta,\psi) $ is the symmetrizer. Up to an extraction of a subsequence, we may assume by weak convergence that 
\begin{align*}
(\eta,\psi) & \in \Linf([0,T],\HC{\mu+1/2,1/2}_{k} \times \HC{\mu,1/2}_{k}), \\
(\pt\eta,\pt\psi) & \in \Linf([0,T],\HC{\mu-1,1/2}_k \times \HC{\mu-3/2,1/2}_k),
\end{align*}
with 
\begin{equation*}
\|(\eta,\psi)\|_{\Linf([0,T],\HC{\mu+1/2,1/2}_{k} \times \HC{\mu}_{k}) \cap \Holder{1}([0,T],\HC{\mu-1,1/2}_k \times \HC{\mu-3/2}_k)} \le C(\|(\eta_0,\psi_0)\|_{\HC{\mu+1/2,1/2}_{k} \times \HC{\mu,1/2}_{k}}).
\end{equation*}
This already implies that $ (\eta,\psi) $ is weakly continuous in $ \HC{\mu+1/2,1/2}_{k} \times \HC{\mu,1/2}_{k} $. By the analysis in the previous section,
\begin{equation*}
(\pt + \P{V}\cdot\nabla )  \Phi + \begin{pmatrix}
0 & - \P{\gamma} \\ \P{\gamma} & 0
\end{pmatrix}  \Phi = F,
\end{equation*}
with
\begin{equation*}
\|F\|_{\Linf([0,T],\Ltwo)} \le C(\|( \eta_0, \psi_0)\|_{\HC{\mu+1/2,1/2}_{k}\times\HC{\mu,1/2}_{k}}).
\end{equation*}
Let $ J_h = \op_h(e^{-|x|^2-|\xi|^2}) $. 
Now that $ e^{-h^2|x|^2-h^2|\xi|^2} \in S^{0}_0 $, we have the commutator estimate
\begin{equation*}
[J_h,\P{V}\cdot\nabla] = \O(1)_{\OP{0}{-1}}, \quad
[J_h,\P{\gamma}] = \O(1)_{\OP{1/2}{-1}}.
\end{equation*}
Because $ k \ge 1 $, by the same spirit of estimating~$ R $ in Proposition~\ref{prop::a-priori-estimate}, we obtain the following energy estimate
\begin{equation*}
\frac{\d}{\dt} \|J_h\Phi(t)\|_{\Ltwo}^2 \le C(\|(\eta_0,\psi_0)\|_{\HC{\mu+1/2,1/2}_{k}\times\HC{\mu,1/2}_{k}}).
\end{equation*}
Therefore, $ t \mapsto \|J_h\Phi(t)\|_{\Ltwo}^2 $ are uniformly Lipschitizian. Consequently, by Arzel\`{a}--Ascoli theorem, $ t \mapsto \|\Phi(t)\|_{\Ltwo}^2 $ is continuous, because $ J_h\Phi \to \Phi $ as $ h \to 0 $. Combining the weak continuity, we deduce by functional analysis that $ \Phi \in C([0,T],\Ltwo) $. By~\eqref{eq::continuity-eta-psi}, the paradifferential calculus, and the definition of~$ \Phi $, we deduce that
\begin{equation*}
(\eta,\psi) \in C([0,T],\HC{\mu+1/2,1/2}_{k} \times \HC{\mu,1/2}_{k}).
\end{equation*}
Thus we finish the proof of Theorem~\ref{thm::ww-weighted-sobolev-existence}.
\end{proof}

\section{Propagation of singularities for water waves}

\label{sec::proof-main-thm}

\subsection{Finer paralinearization and symmetrization}

To study the propagation of singularities, we need much finer results of paralinearization and symmetrization than Proposition~\ref{prop::paralinearization-G(eta)psi-weighted} and Proposition~\ref{prop::paralinearization-WW} so as to gain regularities in the remainder terms.

\begin{proposition}
\label{prop::paralinearization-D-N-higher-order}
If $ (\eta,\psi) \in \H{\mu+1/2} \times \H{\mu} $ with $ \mu-1/2 \in \N $ and $ \mu > 3+ d/2 $, then there exists $ \lambda = \lambda^{(1)} + \lambda^{(0)} + \cdots \in \Sigma^{1,\mu-1/2-\tilde{d}} $ such that
\begin{equation*}
G(\eta)\psi = \P{\lambda}(\psi-\P{B}\eta) - \P{V} \cdot \nabla\eta + R(\eta,\psi),
\end{equation*}
where $ R(\eta,\psi) \in \H{2\mu-K-d/2} $ for some $ K > 0 $ independent of the dimension~$ d $.
Moreover $ \lambda^{(1-j)} $, when it is defined, is a function of derivatives $ \partial_x^\alpha \eta $ where $ |\alpha| \le 1+j $.
\end{proposition}
\begin{proof}
This theorem follows by replacing the usual paradifferential calculus with the dyadic paradifferential calculus in the analysis of \cite{AM09:paralinearization}. In \cite{AM09:paralinearization}, the explicit expression for~$ \lambda $ is given. We write it down for the sake of later applications.
\begin{equation*}
\lambda = (1+|\nabla\eta|^2) a_+ - i \nabla\eta\cdot\xi,
\end{equation*}
where $ a_\pm = \sum_{j \le 1} a_\pm^{(j)} \in \Sigma^{1,\mu-1-d/2} $ is defined as follows. Setting $ c = \frac{1}{1+|\nabla\eta|^2} $, then
\begin{align*}
a^{(1)}_- & = ic\nabla\eta \cdot \xi - \sqrt{c|\xi|^2-(c\nabla\eta\cdot\xi)^2}, &
a^{(1)}_+ & = ic\nabla\eta \cdot \xi + \sqrt{c|\xi|^2-(c\nabla\eta\cdot\xi)^2},\\
a^{(0)}_- & = \frac{i\pxi a^{(1)}_- \cdot \px a^{(1)}_+ - c \Delta \eta a^{(1)}_-}{a^{(1)}_+-a^{(1)}_-}, &
a^{(0)}_+ & = \frac{i\pxi a^{(1)}_- \cdot \px a^{(1)}_+ - c \Delta \eta a^{(1)}_+}{a^{(1)}_--a^{(1)}_+}.
\end{align*}
Suppose that $ a_\pm^{(j)} $ are defined for $ m \le j \le 1 $, then we define
\begin{align*}
a_-^{(m-1)} & = \frac{1}{a_-^{(1)} - a_+^{(1)}} \sum_{m \le k \le 1} \sum_{m \le \ell \le 1} \sum_{|\alpha| = k + \ell - m} \frac{1}{\alpha!} \pxi^\alpha a_-^{(k)} D_x^\alpha a_+^{(\ell)} \\
a_+^{(m-1)} & = - a_-^{(m-1)}.
\end{align*}
The principal and sub-principal symbols of~$ \lambda $ coincide with the ones given by Proposition~\ref{prop::paralinearization-G(eta)psi-weighted}.
\end{proof}

\begin{proposition}
\label{prop::paralinearization-ww-fine}
Let $ (\eta,\psi) \in \H{\mu+1/2} \times \H{\mu} $ with $ \mu-1/2 \in \N $ and $ \mu > 3+ d/2 $.
Let $ \Lambda^\mu = \P{(\gamma^{(3/2)})^{2\mu/3}} $,
and set
\begin{equation*}
w = \Lambda^\mu U S \binom{\eta}{\psi}, \quad
U = \begin{pmatrix}
-i & 1 \\ \phantom{-}i & 1
\end{pmatrix}.
\end{equation*}
Then there exists $ Q \in M_{2\times2}(\Sigma^{0,\mu-1/2-2-\tilde{d}}_{0,0}) $ and $ \zeta \in \Sigma^{-1/2,\mu-1/2-2-\tilde{d}}_{0,0} $ such that for some $ K > 0 $ which is independent of the dimension~$ d $, we have
\begin{equation}
\label{eq::equation-w-sys}
(\pt +  \P{V} \cdot \nabla + \P{Q}) w
+ i \P{\gamma} \begin{pmatrix}
1 & 0 \\ 0 & -1
\end{pmatrix} w
+ \frac{ig}{2} \P{\zeta} \begin{pmatrix*}
1 & -1 \\ 1 & -1
\end{pmatrix*}
\in \H{\mu-K-d/2}.
\end{equation}
\end{proposition}

\begin{remark}
\label{rmk::w-u-u-bar}
Because $ \chi $ in the definition of paradifferential operators is an even function, we verify that $ \Lambda^\mu $, $ \P{p} $, $ \P{q} $, $ \P{B} $ all map real-valued functions to real-valued functions. Therefore, $ \displaystyle w = \binom{u}{\bar{u}} $ with 
\begin{equation}
\label{eq::def-u}
u = \Lambda^\mu (-i, 1) S \binom{\eta}{\psi} 
=  \Lambda^\mu \P{q}\omega - i\Lambda^\mu\P{p} \eta.
\end{equation}
recalling that $ \omega = \psi - \P{B} \eta $ is the good unknown of Alinhac.
\end{remark}

\begin{proof}
Combining Proposition~\ref{prop::paralinearization-D-N-higher-order} and Proposition~\ref{prop::paralinearization-WW}, moving the term $ g\eta $ to the left hand side,
\begin{equation*}
(\pt + \P{V}\cdot\nabla + \L)\binom{\eta}{\psi} + g \binom{0}{\eta} = f(\eta,\psi),
\end{equation*}
where $ \displaystyle f(\eta,\psi) = Q^{-1} \binom{f_1}{f_2} \in \H{2\mu+1/2-K-d/2} \times \H{2\mu-K-d/2} $ for some $ K > 0 $ and is defined by
\begin{align*}
f_1 & = G(\eta) \psi - \{ \P{\lambda} (\psi - \P{B}\eta) - \P{V} \cdot \nabla \eta \},\\
f_2 & = -\frac{1}{2} |\nabla\psi|^2 + \frac{1}{2} \frac{(\nabla\eta \cdot \nabla\psi + G(\eta)\psi)^2}{1+|\nabla\eta|^2} + H(\eta) \\
& \quad \quad + \P{V} \cdot \nabla\psi - \P{B} \P{V} \cdot \nabla\eta - \P{B} G(\eta)\psi + \P{\ell} \eta.
\end{align*}
Given two time-dependent operators $ \A, \B : \swtz \to \swtz' $, we say that $ \A \sim \B $ if 
\begin{equation*}
\A - \B \in \Linf([0,T],\OP{-\mu+d/2+K}{0}).
\end{equation*}
By the ellipticity of~$ \gamma^{(3/2)} $, $ p^{(1/2)} $ and $ q^{(0)} $, we can find paradifferential operators~$ \tilde{\Lambda}^{\mu} $ and~$ \tilde{S} $ by a routine construction of parametrix such that 
$ \tilde{\Lambda}^{\mu} \Lambda^\mu \sim \Id, $  $ \tilde{S} S \sim \Id. $
We can find $ \zeta \in \Sigma^{-1/2,\mu-1/2-2-\tilde{d}} $ with principal symbol $ \zeta^{(-1/2)} = q^{(0)} / p^{(1/2)} $ which implies (note that the only nonzero entries in the following matrices are in the lower left corners)
\begin{equation*}
\begin{pmatrix}
0 & 0 \\ \P{\zeta} & 0
\end{pmatrix} \Lambda^\mu S
- \Lambda^\mu S \begin{pmatrix}
0 & 0 \\ 1 & 0
\end{pmatrix} 
\sim 0.
\end{equation*}

Then by~\eqref{eq::symmetrization} and the fact that the Poisson bracket between the symbol of $ \Lambda^\mu $ and $ \gamma $ vanishes, we find by the symbolic calculus two symbols $ A, B \in M_{2\times 2}(\Sigma^{0,\mu-1/2-2-\tilde{d}}) $ such that
\begin{align*}
\A \bydef [\pt + \P{V}\cdot\nabla, \Lambda^\mu S] 
\sim  [\pt + \P{V}\cdot\nabla, \Lambda^\mu S] \tilde{S} \tilde{\Lambda}^{\mu} \Lambda^\mu S 
& \sim \P{A} \Lambda^\mu S, \\
\B \bydef \begin{pmatrix}
0 & - \P{\gamma} \\ \P{\gamma} & 0
\end{pmatrix} \Lambda^\mu S  - \Lambda^\mu S \L 
\sim \bigg( \begin{pmatrix}
0 & - \P{\gamma} \\ \P{\gamma} & 0
\end{pmatrix} - \Lambda^\mu S \L \tilde{S} \tilde{\Lambda}^{\mu} \bigg)  \Lambda^\mu S
& \sim \P{B} \Lambda^\mu S.
\end{align*}
In fact, by Proposition~\ref{prop::composition-P}, the symbol $A$ is a finite sum of symbols which is given by the symbolic calculus of the operator $[\pt + \P{V}\cdot\nabla, \Lambda^\mu S] \tilde{S} \tilde{\Lambda}^{\mu} $, whereas the symbol $B$ is given by the symbolic calculus of the operator $\begin{pmatrix}
0 & - \P{\gamma} \\ \P{\gamma} & 0
\end{pmatrix} - \Lambda^\mu S \L \tilde{S} \tilde{\Lambda}^{\mu} $.
Clearly $A$ is or zeroth order. The reason why $B$ is of zeroth order is the condition~\eqref{eq::symmetrization} according to which we constructed the symbols $\gamma,p,q$.

Let $ \displaystyle \Phi = \Lambda^\mu S\binom{\eta}{\psi} $, and write
$ g\begin{pmatrix}
0 \\ \eta
\end{pmatrix} = \begin{pmatrix}
0 & 0 \\ g & 0
\end{pmatrix} \begin{pmatrix}
\eta \\ \psi
\end{pmatrix}, $
we obtain by the analysis above that
\begin{equation*}
(\pt + \P{V}\cdot\nabla) \Phi + \begin{pmatrix}
0 & - \P{\gamma} \\ \P{\gamma} & 0
\end{pmatrix} \Phi
+ \begin{pmatrix}
0 & 0 \\ g \P{\zeta} & 0
\end{pmatrix} \Phi
= \P{A} \Phi + \P{B} \Phi + F,
\end{equation*}
where 
\begin{align*}
F = (\A + \B)\binom{\eta}{\psi} - \P{A+B} \Phi + \begin{pmatrix}
0 & 0 \\ g \P{\zeta} & 0
\end{pmatrix} \Phi  - g \Lambda^\mu S \binom{0}{\eta} + \Lambda^\mu S f(\eta, \psi)
\in \H{\mu-K-d/2}.
\end{align*}
Finally, observe that
\begin{align*}
U \begin{pmatrix}
0 & - \P{\gamma} \\ \P{\gamma} & 0
\end{pmatrix} U^{-1} 
& = i \begin{pmatrix}
\P{\gamma} & 0 \\ 0 & -\P{\gamma}
\end{pmatrix}, \\
U \begin{pmatrix}
0 & 0 \\ \P{\zeta} & 0
\end{pmatrix} U^{-1} 
& = \frac{i}{2} \begin{pmatrix}
\P{\zeta} & -\P{\zeta} \\ \P{\zeta} & -\P{\zeta}
\end{pmatrix},
\end{align*}
We conclude by setting
\begin{equation*}
Q = -\frac{1}{2} U(A+B)U^{-1}.\qedhere
\end{equation*}
\end{proof}

\begin{remark}
\label{remark::symbol-polynomial-decay}
By Proposition~\ref{prop::paralinearization-D-N-higher-order} and the symbolic calculus, the symbols that we have encountered, such as $ \lambda $, $ \zeta $ and $ Q $ etc., are of the form $ a = a^{(m)} + a^{(m-1)} + \cdots $ such that $ a^{(m-j)} $, whenever it is defined, is a function of $ (\nabla\eta,\ldots,\nabla^{j+1}\eta) $.
To be precise $ a^{(m-j)} = f_j(\nabla\eta,\ldots,\nabla^{j+1}\eta,\xi) $ where $ f_j $ is homogeneous of degree $ m-j $ in~$ \xi $ and $ f_0(0,\ldots,0,\xi) = |\xi|^{m} $, $ f_j(0,\ldots,0,\xi) = 0 $ for $ j \ge 1 $.
Note that if $ \eta \in \HC{\mu+1/2,1/2}_k $, then for all $j \le \mu+1/2-\tilde{d}$, we have $\nabla^j \eta \in \HC{\mu+1/2-j,1/2}_k$.
Therefore, by Lemma~\ref{lem:weighted-symbol-injection},
\begin{equation}
\label{eq:decay-eta}
\nabla^j \eta \in
\Holder{\min\{[2(\mu+1/2-j-\tilde{d})/3],k\}}_{0,1}
\cap
\jp{x}^{-\min\{2(\mu+1/2-j-\tilde{d}),k\}} L^\infty,
\end{equation}
and consequently
\begin{equation}
\label{eq:symbol-optimal-decay}
a^{(m)}-|\xi|^m 
\in \Gamma^{m,0}_{-\min\{2\mu-1-2\tilde{d},k\},0},\quad
a^{(m-j)}
\in \Gamma^{m-j,0}_{-\min\{2\mu-1-2j-2\tilde{d},k\},0}.
\end{equation}
As another consequence of~\eqref{eq:decay-eta}, we also have
\begin{equation}
\label{eq::symbol-decay-complete}
\begin{split}
a^{(m)} - |\xi|^m \in  & \Gamma^{m,\min\{[2(\mu-1/2-\tilde{d})/3],k\}}_{0,1},\\
a^{(m-j)} \in \Gamma^{m-j,\min\{[2(\mu-1/2-j-\tilde{d})/3],k\}}_{-j,1}
\subset & \Gamma^{m-j,\min\{[2(\mu-1/2-\tilde{d})/3],k\}-j}_{-j,1}.
\end{split}
\end{equation}
\end{remark}

\begin{lemma}
\label{lem::WF-eta-psi==u}
Let~$ u $ be defined as~\eqref{eq::def-u}. If $ (\eta,\psi) \in \H{\mu+1/2} \times \H{\mu} $ with $ \mu-1/2 \in \N $, then for $ 0 \le \sigma \le r \in \N $ with $ r < \mu-1/2-1-\tilde{d} $,
\begin{equation*}
\WF_{0,1}^\sigma(u)^\circ = \WF_{0,1}^{\mu+1/2+\sigma}(\eta)^\circ \cup \WF_{0,1}^{\mu+\sigma}(\psi)^\circ.
\end{equation*}
If $ (\eta,\psi) \in \HC{\mu+1/2}_{k} \times \HC{\mu}_{k} $, with $ k \le \frac{2}{3}(\mu-1-\tilde{d}) $, then for $ 0 \le \sigma \le \frac{3}{2} k $, 
\begin{equation*}
\WF_{1/2,1}^\sigma(u)^\circ = \WF_{1/2,1}^{\mu+1/2+\sigma}(\eta)^\circ \cup \WF_{1/2,1}^{\mu+\sigma}(\psi)^\circ.
\end{equation*}
\end{lemma}
\begin{proof}
Clearly if $ \eta \in \H{\mu+1/2} $, then $ (\gamma^{(3/2)})^{2\mu/3} \in \Gamma^{\mu,r} $, $ p^{(1/2)} \in \Gamma^{1/2,r} $, $ q^{(0)} \in \Gamma^{0,r} $, $ B \in \Gamma^{0,r} $.
By~\eqref{eq::symbol-decay-complete}, if $ \eta \in \HC{\mu+1/2}_k $, then $ (\gamma^{(3/2)})^{2\mu/3} \in \Gamma^{\mu,k}_{0,1} $, $ p^{(1/2)} \in \Gamma^{1/2,k}_{0,1} $, $ q^{(0)} \in \Gamma^{0,k}_{0,1} $, $ B \in \Gamma^{0,k}_{0,1} $.
By Lemma~\ref{lem::WF-paradiff-elliptic} and~\eqref{eq::def-u}, for either $ \epsilon = 0 $ or $ \epsilon = 1/2 $,
\begin{align*}
\WF_{\epsilon,1}^{\sigma}(u)^\circ
& = \WF_{\epsilon,1}^{\sigma}(\Lambda^\mu \P{p}\eta)^\circ \cup \WF_{\epsilon,1}^{\sigma}(\Lambda^\mu \P{q}(\psi-\P{B}\eta))^\circ \\
& = \WF_{\epsilon,1}^{\mu+1/2+\sigma}(\eta)^\circ \cup \WF_{\epsilon,1}^{\mu+\sigma}(\psi-\P{B}\eta)^\circ \\
& \subset \WF_{\epsilon,1}^{\mu+1/2+\sigma}(\eta)^\circ \cup \big( \WF_{\epsilon,1}^{\mu+\sigma}(\psi)^\circ \cup \WF_{\epsilon,1}^{\mu+\sigma}(\P{B}\eta)^\circ \big) \\
& \subset \WF_{\epsilon,1}^{\mu+1/2+\sigma}(\eta)^\circ \cup \big( \WF_{\epsilon,1}^{\mu+\sigma}(\psi)^\circ \cup \WF_{\epsilon,1}^{\mu+\sigma}(\eta)^\circ \big) \\
& = \WF_{\epsilon,1}^{\mu+1/2+\sigma}(\eta)^\circ \cup \WF_{\epsilon,1}^{\mu+\sigma}(\psi)^\circ.
\end{align*}
Conversely, as $ \WF_{\epsilon,1}^{\mu+\sigma}(\P{B}\eta)^\circ
\subset \WF_{\epsilon,1}^{\mu+1/2+\sigma}(\eta)^\circ $, we have
\begin{align*}
\WF_{\epsilon,1}^{\mu+1/2+\sigma}&(\eta)^\circ \cup \WF_{\epsilon,1}^{\mu+\sigma}(\psi)^\circ \\ 
& = \WF_{\epsilon,1}^{\mu+1/2+\sigma}(\eta)^\circ \cup \big( \WF_{\epsilon,1}^{\mu+\sigma}(\psi)^\circ \backslash \WF_{\epsilon,1}^{\mu+1/2+\sigma}(\eta)^\circ \big) \\
& =\WF_{\epsilon,1}^{\mu+1/2+\sigma}(\eta)^\circ \cup \big( \WF_{\epsilon,1}^{\mu+\sigma}(\psi-\P{B}\eta)^\circ \backslash \WF_{\epsilon,1}^{\mu+1/2+\sigma}(\eta)^\circ \big) \\
& = \WF_{\epsilon,1}^{\mu+1/2+\sigma}(\eta)^\circ \cup \WF_{\epsilon,1}^{\mu+\sigma}(\psi-\P{B}\eta)^\circ\\
& = \WF_{\epsilon,1}^{\sigma}(u)^\circ.
\end{align*}
The lemma follows.
\end{proof}

\subsection{Proof of Theorem~\ref{thm::main-infinite}}

By Lemma~\ref{lem::WF-eta-psi==u}, it is equivalent to prove the following theorem.

\begin{theorem}
\label{thm::main-infinite-u}
Under the hypothesis of Theorem~\ref{thm::main-infinite}, let~$ u $ be defined by~\eqref{eq::def-u}, and let
\begin{equation*}
(x_0,\xi_0) \in \WF_{1/2,1}^\sigma(u_0)^\circ
\end{equation*}
with $ 0 \le \sigma < k/2 - 3/2 $. 
Let $ t_0 \in [0,T] $, and suppose that
\begin{equation*}
x_0+\frac{3}{2}t|\xi_0|^{-1/2}\xi_0 \ne 0, 
\end{equation*}
for all $ t \in [0,t_0] $, then
\begin{equation*}
\Big(x_0+\frac{3}{2}t_0|\xi_0|^{-1/2}\xi_0,\xi_0\Big) \in \WF_{1/2,1}^\sigma(u(t_0))^\circ.
\end{equation*}
\end{theorem}

\begin{proof}
For $ \nu \in \R $, denote 
\begin{equation*}
X^\nu = \sum_{k \in \Z} H^{\nu-k/2}_k.
\end{equation*}
By Lemma~\ref{lem::basic-properties-WF}, if $ f \in X^\nu $, then $ \WF_{1/2,1}^{\nu}(f)^\circ = \emptyset $. 
Also note that if $ f \in X^\nu $ and $ a \in \Sigma^{m,r}_{0,1} $, then $ \P{a} f \in X^{\nu-m} $.
As $ k < 2\mu-d $, we have $ V \in \HC{\mu}_{k} \subset \jp{x}^k H^{\mu-k/2} \subset \jp{x}^{-k} L^\infty $ which implies
\begin{equation*}
\P{V} \cdot \nabla w \subset \P{V} H^{-1} \subset H^{-1}_k \subset X^{k/2-1}.
\end{equation*}
By Remark~\ref{remark::symbol-polynomial-decay}, particularly~\eqref{eq:symbol-optimal-decay},
\begin{equation*}
\P{Q} w
\in \sum_{j<\mu-\tilde{d}} H^{j}_{\min\{2\mu-1-2j-2\tilde{d},k\}}
\subset 
\sum_{j<\mu-\tilde{d}} X^{\min\{\mu-1-\tilde{d},j+k/2\}}
\subset X^{k/2}.
\end{equation*}
Similarly
\begin{align*}
\P{\gamma} w - \P{|\xi|^{3/2}} w 
& 
\in \sum_{j<\mu-\tilde{d}} H^{j-3/2}_{\min\{2\mu-1-2j-2\tilde{d},k\}}
\subset X^{k/2-3/2}, \\
\P{\zeta} w - \P{|\xi|^{-1/2}} w 
& \in \sum_{j<\mu-\tilde{d}} H^{j+1/2}_{\min\{2\mu-1-2j-2\tilde{d},k\}}
\subset X^{k/2+1/2}.
\end{align*}
By the hypothesis on $ m $, we thus obtain
\begin{equation}
\pt w'
+ i |D_x|^{3/2} \begin{pmatrix}
1 & 0 \\ 0 & -1
\end{pmatrix} w'
+ \frac{ig}{2} |D_x|^{-1/2} \begin{pmatrix}
1 & -1 \\ 1 & -1
\end{pmatrix} w'
\in X^{k/2-3/2},
\end{equation}
where $ w' = \pi(D_x) w $, and $ \pi \in \Cinf(\Rd) $ which vanishes near the origin, and equals to~$ 1 $ out side a neighborhood of the origin. Moreover, we require that $ \supp \pi \subset \{\tilde{\pi}=1\} $ such that $ 1 - \tilde{\pi} \in \Ccinf(\Rd) $ and $ \tilde{\pi}(\xi) = 0 $ if $ |\xi|^2 \le |g| $. Observe that the matrix
\begin{equation*}
M = |\xi|^{3/2} \begin{pmatrix}
1 & 0 \\ 0 & -1
\end{pmatrix}
+ \frac{g}{2} |\xi|^{-1/2} \begin{pmatrix}
1 & -1 \\ 1 & -1
\end{pmatrix}
\end{equation*}
is symmetrizable when restricted to $ \supp \pi $. Indeed, let 
\begin{equation*}
P = \frac{1}{2} \begin{pmatrix}
\phantom{-} 1+\theta & \phantom{-}1-\theta \\ -(1-\theta) & -(1+\theta)
\end{pmatrix},
\end{equation*}
where $ \theta = \sqrt{\tilde{\pi}(\xi) \cdot (g|\xi|^{-2}+1)} $, then $ P \in \OP{0}{0} $.
For $ \xi \in \supp \pi $, we have
\begin{equation*}
PMP^{-1} = |\xi|^{3/2}\theta(\xi) \begin{pmatrix}
1 & 0 \\ 0 & -1
\end{pmatrix}.
\end{equation*}
Set 
\begin{equation*}
\tilde{w} = P(D_x) w'
= P(D_x)\binom{u'}{\overline{u'}} =
\binom{\phantom{-}\Re\, u' + i \theta(D_x) \Im\, u'}{-\Re\, u' + i \theta(D_x) \Im\, u'},
\end{equation*}
where $ u' = \pi(D_x) u $, then
\begin{equation*}
\pt \tilde{w} + |D_x|^{3/2}\theta(D_x) \begin{pmatrix}
1 & 0 \\ 0 & -1
\end{pmatrix} \tilde{w} \in X^{k/2-3/2}.
\end{equation*}
Finally, let $ v = \Re\, u' + i \theta(D_x) \Im\, u' $, then $ \WF^\sigma_{1/2,1}(u)^\circ = \WF^\sigma_{1/2,1}(v)^\circ $, and
\begin{equation*}
\pt v + |D_x|^{3/2}\theta(D_x) v \in X^{k/2-3/2}.
\end{equation*}
We are left to prove that if $ (x_0,\xi_0) \in \WF_{1/2,1}^\sigma(v(0))^\circ $, then
\begin{equation*}
\Big(x_0+\frac{3}{2}t_0|\xi_0|^{-1/2}\xi_0,\xi_0\Big) \in \WF_{1/2,1}^\sigma(v(t_0)).
\end{equation*}
Because $ \theta(\xi) \sim 1 $ in the high frequency regime, similar proof as~\ref{thm::model-eq-infinity} of Theorem~\ref{thm::model-eq} yields the conclusion.
\end{proof}

\subsection{Proof of Theorem~\ref{thm::main-finite}}

\subsubsection{Hamiltonian flow}

Let $ \Phi = \Phi_s : \Rd \times (\Rd \backslash 0) \to \Rd \times (\Rd \backslash 0) $ be the Hamiltonian flow of 
\begin{equation*}
H(x,\xi) = \gamma^{(3/2)}(0,x,\xi) 
= \Big(|\xi|^2 - \frac{(\nabla\eta_0 \cdot \xi)^2}{1+|\nabla\eta_0|^2} \Big)^{3/4}.
\end{equation*}
That is 
\begin{equation*}
\ps \Phi_s(x,\xi) = X_{H}(\Phi_s(x,\xi)), \quad
\Phi|_{s=0} = \Id_{\Rd \times (\Rd \backslash 0)},
\end{equation*}
where $ X_{H} = (\pxi H, -\px H) $. We use~$ s $ to denote the time variable in accordance to the semiclassical time variable in the following section. Observe that 

\begin{lemma}
\label{lem::Geo-Phi-relation}
For $ (x,\xi) \in \Rd \times (\Rd \backslash 0) $, we have 
\begin{equation*}
\Phi_s(x,\xi) = \Geo_{\varphi_s(x,\xi)}(x,\xi),
\end{equation*}
where $ \Geo $ is the geodesic flow defined in \S\ref{sec::intro-Microlocal-Smoothing-Effect}, and
\begin{equation*}
\varphi_s(x,\xi) = \frac{3}{4} \int_0^s G(\Phi_\sigma(x,\xi))^{-1/4}  \d\sigma.
\end{equation*}
\end{lemma}
\begin{proof}
We have $ \Geo_{\varphi_0(x,\xi)}(x,\xi) = \Geo_{0}(x,\xi) = (x,\xi) = \Phi_0(x,\xi). $ Then observe that
\begin{equation*}
H(x,\xi) = G(x,\xi)^{3/4} = \varrho_x^{-1}(\xi,\xi)^{3/4}.
\end{equation*}
Therefore,
\begin{align*}
\tfrac{\d}{\ds} \Geo_{\varphi_s(x,\xi)}(x,\xi)
& = \tfrac{\d}{\ds}\varphi_s(x,\xi) (\tfrac{\d}{\ds} \Geo)_{\varphi_s(x,\xi)}(x,\xi)  \\
& = \tfrac{3}{4}G(\Geo_{\varphi_s(x,\xi)}(x,\xi))^{-1/4} X_G(\Geo_{\varphi_s(x,\xi)}(x,\xi)) \\
& = X_{H}(\Geo_{\varphi_s(x,\xi)}(x,\xi)).
\end{align*}
We conclude by the uniqueness of solutions to Hamiltonian ODEs.
\end{proof}

\begin{lemma}
\label{lem::hamiltonian-flow-to-infinity}
Suppose that for some $ \epsilon > 0 $, $ \nabla\eta_0 \in \Holder{0}_{1/2+\epsilon} $, $ \nabla^2 \eta_0 \in \Holder{0}_{1+\epsilon}. $ Let $ (x_0,\xi_0) \in \Rd \times (\Rd \backslash 0) $ such that the co-geodesic $ \{(x_s,\xi_s) = \Phi_s(x_0,\xi_0)\}_{s\in\R} $ is forwardly non-trapping. Set
\begin{equation*}
z_s = x_s - x_0 - \frac{3}{2} \int_0^s |\xi_\sigma|^{-1/2} \xi_\sigma \d\sigma,
\end{equation*}
then there exists $ (z_{+\infty},\xi_{+\infty}) \in \Rd \times (\Rd \backslash 0) $ such that \begin{equation*}
\lim_{s\to+\infty} (z_s,\xi_s) = (z_{+\infty},\xi_{+\infty}).
\end{equation*}
Consequently, by Lemma~\ref{lem::Geo-Phi-relation}, let $ (x'_s,\xi'_s) = \Geo_s(x_0,\xi_0) $, then 
\begin{equation*}
\lim_{s \to + \infty} \xi'_s = \xi_{+\infty}.
\end{equation*}
\end{lemma}

\begin{proof}
Because $ \{(x_s,\xi_s)\}_{s\in\R} $ is forwardly non-trapping, and we only consider the limiting behavior when~$ s \to +\infty $, we may assume that $ \varepsilon_0 \bydef \| \jp{x} \nabla^2 \eta_0\|_{\Linf} $ is sufficiently small. As $ \nabla\eta_0 \in \Linf $, we have $ H(\cdot,\xi) \simeq |\xi|^{3/2} $. Then 
\begin{equation*}
\frac{\d}{\ds} (x_s \cdot \xi_s) = \pxi H(x_s,\xi_s) \cdot \xi_s - x_s \cdot \px H(x_s,\xi_s), 
\end{equation*}
where 
\begin{equation*}
\pxi H(x_s,\xi_s) \cdot \xi_s = \frac{3}{2} H(x_s,\xi_s) = \frac{3}{2} H(x_0,\xi_0) \simeq |\xi_0|^{3/2},
\end{equation*}
and
\begin{align*}
\px H(x_s,\xi_s) & = \frac{3}{4} H(x_s,\xi_s)^{-1/3} \px G(x_s,\xi_s) \\
& = \frac{3}{4} H(x_s,\xi_s)^{-1/3} \Big( \frac{2 \nabla\eta_0 \cdot \xi_s}{1+|\nabla\eta_0|^2} \nabla^2\eta_0 \xi_s - \frac{2(\nabla\eta_0\cdot\xi_s)^2}{(1+|\nabla\eta_0|^2)^2} \nabla^2\eta_0 \nabla\eta_0 \Big)\Big|_{x=x_s}.
\end{align*}
Therefore 
\begin{equation*}
x_s \cdot \px H(x_s,\xi_s) = \O(\varepsilon|\xi_s|^{3/2}) = \O(\varepsilon |\xi_0|^2),
\end{equation*}
and consequently, 
\begin{equation}
\label{eq::growth-rate-x-cdot-xi}
\frac{\d}{\ds} (x_s \cdot \xi_s) \gtrsim |\xi_0|^{3/2}.
\end{equation}
So for any bounded set $ B \subset \Rd $, 
\begin{align}
\label{eq::escaping-time-estimate}
\lambda (s \ge 0 : x_s \in B ) 
& \lesssim \frac{\sup\{ |x \cdot \xi| : (x,\xi) \in B \times \Rd, H(x,\xi) = H(x_0,\xi_0) \}}{|\xi_0|^{3/2}} \\
& \lesssim \sup_{x\in B} |x| \jp{\xi_0}^{-1/2}, \nonumber
\end{align}
where~$ \lambda $ is the Lebesgue measure on~$ \R $. Let 
\begin{equation*}
E(x,\xi) = H(x,\xi) - |\xi|^{3/2},
\end{equation*}
then by the hypothesis of the decay of~$ \eta_0 $, $ E \in \Gamma^{3/2,1}_{-1-\epsilon,0} $. By the definition of~$ z_s $, we have
\begin{equation*}
\frac{\d}{\ds}(z_s,\xi_s) = (\pxi E,-\px E)(x_s,\xi_s)
= \O(\jp{x_s}^{-1-\epsilon}),
\end{equation*}
where we used the conversation of $ H(x_s,\xi_s) $ to deduce the boundedness of $ \xi_s $.
By~\eqref{eq::escaping-time-estimate},
\begin{equation*}
\int_{0}^\infty \jp{x_s}^{-1-\epsilon} \ds
= (1+\epsilon) \int_0^{\infty} t^\epsilon \lambda(s \ge 0 : \jp{x_s}^{-1} > t ) \d t 
\lesssim \int_0^1 t^\epsilon \sqrt{t^{-2}-1} \d t
< \infty.
\end{equation*}
Therefore, for any $ 0 < s^- < s^+ $ with $ s^- \to \infty $,
\begin{align*}
|(z_{s^+},\xi_{s^+})-(z_{s^-},\xi_{s^-})| 
& \lesssim \int_{s^-}^{s^+} \jp{x_\sigma}^{-1-\epsilon} \d\sigma \to 0,
\end{align*}
implying that $ (x_s,\xi_s) $ is a Cauchy sequence as $ s \to \infty $.
\end{proof}

\subsubsection{Construction of symbol}

For $ h \ge 0 $, and $ h^{1/2} s \le T $, set
\begin{equation*}
H_h(s,x,\xi) = \gamma^{(3/2)}(h^{1/2}s, x,\xi),
\end{equation*}
so in particular $ H(x,\xi) \equiv H_0(s,x,\xi) $. For $ h > 0 $, the semiclassical time variable $ s = h^{-1/2} t $ was inspired by Lebeau \cite{Lebeau1992schrodinger}, see also \cite{Zhu20control} for an application in theory of control for water waves.

For $ a \in \Cinf([0,h^{-1/2}T[ \times \R^{2d}) $, set
\begin{equation*}
\Lag^\pm_{h,s} a = \ps a \pm \{H_h,a\}.
\end{equation*}

\begin{lemma}
\label{lem::chi-symbol-construction-main}
Suppose that for some $ \epsilon > 0 $, $ \nabla \eta_0 \in \Holder{0}_{1/2+\epsilon} $, $ \nabla^2 \eta_0 \in \Holder{0}_{1+\epsilon} $, $ \nabla^3 \eta_0 \in \Holder{0}_{3/2+\epsilon} $. Let $ (x_0,\xi_0) \in \Rd \times (\Rd \backslash 0) $ such that the co-geodesic $ \{(x_s,\xi_s) = \Phi_s(x_0,\xi_0)\}_{s\in\R} $ is forwardly non-trapping, then there exists $ s_0 > 0 $, $ K > 0 $ and
\begin{equation}
\label{eq::symbol-chi-space}
\chi^\pm \in \Holder{1}(\R_{\ge 0},\Upsilon^{\mu-K-\tilde{d}}) \cap \Holder{1}(\R_{\ge s_0},S^{-\infty}_0)
\end{equation}
in the sense that
\begin{equation*}
\|N^{\mu-K-\tilde{d}}(\chi^\pm)\|_{\Linf(\R_{\ge 0})} + \|N^{\mu-K-\tilde{d}}(\ps \chi^\pm)\|_{\Linf(\R_{\ge 0})} < +\infty,
\end{equation*}
and satisfies the following conditions:
\begin{enumerate}
\item $ \chi^\pm(0,x,\xi) \in S^{-\infty}_{-\infty} $ is elliptic at $ (x_0,\pm \xi_0) $;
\item for all $ t_0 > 0 $, $ \chi^\pm(s,\frac{s}{t_0}x,\xi) \in S^{-\infty}_{-\infty} $ is elliptic at $ (\frac{3}{2}t_0|\xi_\infty|^{-1/2}\xi_\infty,\pm\xi_{\infty}) $ for sufficiently large~$ s $;~and
\item if $ \Omega $ is a neighborhood of $ (\frac{3}{2} t_0 |\xi_\infty|^{-1/2} \xi_\infty, \pm\xi_{\infty}) $, then $ \chi^\pm $ can be chosen such that 
\begin{equation*}
\supp \chi^\pm\big(s,\frac{s}{t_0}x,\xi\big) \subset \Omega
\end{equation*}
for sufficiently large $ s $.
\end{enumerate}
Moreover, if $ (\eta,\psi) \in \HC{\mu+1/2}_k \times \HC{\mu}_k $ with $ \mu > 3+d/2 $ and $ m \ge 2 $, then
\begin{equation*}
\Lag^\pm_{h,s} \chi^\pm  \in\Linf([0,h^{-1/2}T],\jp{x}^{-1}\Upsilon^{\mu-K-\tilde{d}-1}),
\end{equation*}
and
\begin{equation}
\label{eq::chi-pm-quasi-positivity}
\Lag^\pm_{h,s} \chi^\pm 
\ge \O(h^{1/2})_{\Linf([0,h^{-1/2}T],\jp{x}^{-1}\Upsilon^{\mu-K-\tilde{d}-1})}.
\end{equation}
\end{lemma}
\begin{proof}
Let $ \phi \in \Ccinf(\Rd) $ such that
\begin{enumerate}[label=(\roman*)]
\item \label{phi::support} $ \phi \ge 0 $, $ \phi(x) = 1 $ for $ |x| \le 1/2 $, $ \phi(x) = 0 $ for $ |x| \ge 1 $, $ \supp \phi = \{|x| \le 1 \} $;
\item \label{phi::decay} $ x \cdot \nabla\phi(x) \le 0 $ for all $ x \in \Rd $;
\item \label{phi::radial} $ y \cdot \nabla\phi(x) = 0 $ for all $ x,y \in \Rd $ with $ x \cdot y = 0 $.
\end{enumerate}
Such $ \phi $ can be constructed by setting $ \phi(x) = \varphi(|x|) $ where $ \varphi : \R \to \R $ satisfies $ 0 \le \varphi \le 1 $, $ \varphi(z) = 1 $ if $ z \le 1/2 $, $ \varphi(z) = 0 $ if $ z \ge 1 $. 
For $ \rho > 0 $, $ \delta > 0 $, $ \lambda > 0 $, $ \nu > 0 $ and sufficiently large $ s > 0 $, set
\begin{equation*}
\tilde{\chi}^\pm(s,x,\xi) = \phi\Big(  \frac{x - x_s}{\rho \lambda \delta s} \Big) \phi\Big( \frac{\xi \mp \xi_s}{\rho(\delta - s^{-\nu})}\Big).
\end{equation*}
We verify that $ \Lag_{0,s}^\pm \tilde{\chi}^\pm(s,\cdot) \ge 0 $ for $ s > 0 $ sufficient large. Indeed,
\begin{align*}
\Lag_{0,s}^\pm & \tilde{\chi}^\pm(s,x,\xi) \\
& = \Big( \pm\frac{\pxi H(x,\xi) - \pxi H(x_s, \mp \xi_s)}{\rho \lambda \delta s} - \frac{x-x_s}{\rho \lambda \delta s^2} \Big) \nabla\phi\Big( \frac{x-x_s}{ \rho \lambda \delta s} \Big) \phi\Big(\frac{\xi \mp \xi_s}{\rho(\delta - s^{-\nu})}\Big) \\
& \quad + \Big( \pm \frac{\px H(x_s, \mp \xi_s) - \px H(x,\xi)}{\rho (\delta - s^{-\nu})} - \nu\frac{\xi\mp\xi_s}{\rho(\delta - s^{-\nu})^2s^{\nu+1}} \Big) \phi\Big( \frac{x-x_s}{\rho \lambda \delta s} \Big) \nabla\phi\Big(\frac{\xi \mp \xi_s}{\rho(\delta - s^{-\nu})}\Big).
\end{align*}
By~\ref{phi::support},
\begin{align*}
\supp \phi\Big( \frac{\cdot - x_s}{ \rho \lambda \delta s} \Big)
& \subset \big\{x \in \Rd : |x-x_s| \le \rho \lambda \delta s \big\}, \\
\supp \phi\Big(\frac{\cdot \mp \xi_s}{\rho(\delta - s^{-\nu})}\Big)
& \subset \big\{\xi \in \Rd : |\xi \mp \xi_s| \le \rho(\delta - s^{-\nu}) \big\}, \\
\supp \nabla\phi\Big( \frac{\cdot - x_s}{ \rho\lambda \delta s} \Big)
& \subset \big\{x \in \Rd : \frac{1}{2} \rho \lambda \delta s \le |x-x_s| \le \rho \lambda \delta s \big\}, \\
\supp \nabla\phi\Big(\frac{\cdot \mp \xi_s}{\rho(\delta - s^{-\nu})}\Big)
& \subset \big\{\xi \in \Rd : \frac{1}{2} \rho (\delta - s^{-\nu})  \le |\xi \mp \xi_s| \le \rho (\delta - s^{-\nu}) \big\}.
\end{align*}
By Lemma~\ref{lem::hamiltonian-flow-to-infinity},
\begin{equation*}
x_s
= x_0 + \frac{3}{2} \int_0^s |\xi_\sigma|^{-1/2} \xi_\sigma \d\sigma + z_s
= \frac{3}{2} s |\xi_{\infty}|^{-1/2} \xi_{\infty} + o(s).
\end{equation*}
Therefore, by writing
\begin{equation*}
\tilde{\chi}^\pm \big(s,\frac{s}{t_0}x,\xi\big) 
= \phi\Big( \frac{x - \frac{3}{2}t_0|\xi_{\infty}|^{-1/2} \xi_{\infty} + o(1)}{\rho \lambda \delta t_0} \Big) \phi\Big(\frac{\xi \mp \xi_\infty + o(1)}{\rho(\delta - s^{-\nu})}\Big),
\end{equation*}
we see that $ \tilde{\chi}^\pm(s,\frac{s}{t_0}x,\xi) $ is elliptic at $ (\frac{3}{2}t_0|\xi_\infty|^{-1/2}\xi_\infty,\pm\xi_{\infty}) $ for sufficiently large~$ s $. Moreover, if $ \rho\lambda\delta $ is sufficiently small and~$ s $ is sufficiently large, then
\begin{equation*}
\supp \phi\Big( \frac{\cdot - x_s}{ \rho \lambda \delta s} \Big) \subset \{x \in \Rd : |x| \gtrsim s \}.
\end{equation*}
Therefore, by the hypothesis on $ \eta_0 $,  we have for $ (x,\xi) \in \supp \tilde{\chi}^\pm(s,\cdot) $,
\begin{equation*}
\nabla_{x\xi}^2 H(x,\xi) = \begin{pmatrix}
\nabla_x^2 H & \nabla_x\nabla_\xi H \\
\nabla_\xi \nabla_x H & \nabla_\xi^2 H
\end{pmatrix}(x,\xi)
= \begin{pmatrix}
\O(s^{-2-\epsilon}) & \O(s^{-3/2-\epsilon}) \\ \O(s^{-3/2-\epsilon}) & \O(1)
\end{pmatrix},
\end{equation*}
and consequently, by the finite increment formula,
\begin{align*}
|\pxi H(x_s, \mp \xi_s) - \pxi H(x,\xi)| 
& \lesssim s^{-3/2-\epsilon} |x-x_s| + |\xi\mp\xi_s| 
\lesssim s^{-1/2-\epsilon} \rho \lambda \delta + \rho \delta;  \\
|\px H(x_s, \mp \xi_s) - \px H(x,\xi)| 
& \lesssim s^{-2-\epsilon} |x-x_s| + s^{-3/2-\epsilon} |\xi\mp\xi_s|  
\lesssim \rho \lambda \delta s^{-1-\epsilon} + \rho \delta s^{-3/2-\epsilon}.
\end{align*}
By~\ref{phi::radial} and the estimates above,
\begin{align*}
\big( \pxi H(x,\xi) & - \pxi H(x_s, \mp \xi_s) \big) \cdot \nabla\phi\Big( \frac{x-x_s}{ \rho \lambda \delta s} \Big) \\
& = \big( \pxi H(x,\xi) - \pxi H(x_s, \mp \xi_s) \big) \cdot \frac{x-x_s}{|x-x_s|^2} (x-x_s)\cdot \nabla\phi\Big( \frac{x-x_s}{ \rho \lambda \delta s} \Big) \\
& = \O(s^{-3/2-\epsilon}+\lambda^{-1}s^{-1}) (x-x_s)\cdot \nabla\phi\Big( \frac{x-x_s}{ \rho \lambda\delta s} \Big); \\
\big( \px H(x_s, \mp \xi_s) & - \px H(x,\xi) \big) \cdot \nabla\phi\Big(\frac{\xi \mp \xi_s}{\rho (\delta - s^{-\nu})}\Big) \\
& = \big( \px H(x_s, \mp \xi_s) - \px H(x,\xi) \big) \cdot \frac{\xi \mp \xi_s}{|\xi \mp \xi_s|^2} (\xi \mp \xi_s)\cdot \nabla\phi\Big(\frac{\xi \mp \xi_s}{\rho(\delta - s^{-\nu})}\Big) \\
& = \O(\lambda s^{-1-\epsilon}+s^{-3/2-\epsilon}) (\xi \mp \xi_s)\cdot \nabla\phi\Big(\frac{\xi \mp \xi_s}{\rho(\delta - s^{-\nu})}\Big).
\end{align*}
Finally, we fix $ 0 < \nu < \epsilon $, $ \delta > 0 $. Then, when $ \lambda $ is sufficiently large, and $ s \ge s_0 - 1 > 0 $ with~$ s_0 $ being sufficiently large, by~\ref{phi::decay},
\begin{align}
\label{eq::Lag_0-chi}
\Lag_{0,s}^\pm \tilde{\chi}^\pm
& = -\frac{1+\O(s^{-1/2-\epsilon}+\lambda^{-1})}{\rho\lambda\delta s^{2}} (x-x_s) \cdot \nabla\phi\Big( \frac{x-x_s}{ \rho \lambda \delta s} \Big) \phi\Big(\frac{\xi \mp \xi_s}{\rho(\delta - s^{-\nu})}\Big) \\
& \qquad - \frac{\nu-\O(\lambda)s^{\nu-\epsilon}}{\rho(\delta-s^{-\nu})^{2}s^{\nu+1}}   (\xi\mp\xi_s) \cdot \phi\Big( \frac{x-x_s}{\rho \lambda \delta s} \Big) \nabla\phi\Big(\frac{\xi \mp \xi_s}{\rho(\delta - s^{-\nu})}\Big) 
\ge 0. \nonumber
\end{align}
We verify as in Lemma~\ref{lem::symbol-prop-to-infinifty-model-eq} that 
\begin{equation*}
\tilde{\chi}^\pm \in \Holder{\infty}(\R_{\ge s_0},S^{-\infty}_0), \quad
\Lag_{0,s}^\pm \tilde{\chi}^\pm \in \Holder{\infty}(\R_{\ge s_0},\Gamma^{-\infty,\mu-K-\tilde{d}}_{-1,0}). 
\end{equation*} 
We then choose $ \rho>0 $ sufficiently small such that $ \rho\lambda\delta $ is small and that $ \supp \tilde{\chi}^\pm(s,\frac{s}{t_0}x,\xi) \subset \Omega $ when~$ s $ is large.
Next, we set for $ s \ge s_0 $, 
\begin{equation*}
\chi^\pm(s,x,\xi) = \tilde{\chi}^\pm(s,x,\xi).
\end{equation*}
To define $ \chi^\pm $ for $ s \le s_0 $, we choose $ \rho \in \Cinf(\R) $ such that $ 0 \le \rho \le 1 $, $ \rho(s) = 1 $ for $ s \ge s_0 $, and $ \rho(s) = 0 $ for $ s \le s_0-\alpha $ for some small $ \alpha > 0 $ to be specified later, and solve the transport equation on $ [0,s_0] $,
\begin{equation*}
\Lag_{0,s}^\pm \chi^\pm(s,x,\xi) = \rho(s) \Lag_{0,s}^\pm \tilde{\chi}^\pm(s,x,\xi), \quad
\chi^\pm(s_0,x,\xi) = \tilde{\chi}^\pm(s_0,x,\xi).
\end{equation*}
Because the vector field involved in the definition of $ \Lag_{0,s}^\pm $ is in $ \Holder{\mu-K-\tilde{d}} $ with respect to the~$ x $ variable, we deduce that $ \chi^\pm \in \Holder{1}(\R_{\ge 0},\Upsilon^{\mu-K-\tilde{d}}) $ and thus $ \chi^\pm $ satisfies~\eqref{eq::symbol-chi-space}.
Clearly
\begin{equation}
\label{eq::positivity-Lag-chi-0}
\Lag_{0,s}^\pm \chi^\pm \ge 0.
\end{equation}
Moreover, because
\begin{equation*}
\chi^\pm(s,x,\xi) = \tilde{\chi}^\pm(s_0,\Phi_{\pm(s_0-s)}(x,\xi)) - \int_s^{s_0} \rho(\sigma) \Lag_{0,s}^\pm \tilde{\chi}^\pm(\sigma,\Phi_{\pm(\sigma-s)}(x,\xi)) \d\sigma,
\end{equation*}
if we choose $ \alpha > 0 $ sufficiently small, then
\begin{align*}
\chi^\pm(0,x_0,\pm\xi_0) 
& = \tilde{\chi}^\pm(s_0,x_{s_0},\pm\xi_{s_0}) - \int_{s_0-\alpha}^{s_0} \rho(\sigma) \Lag_{0,s}^\pm \tilde{\chi}^\pm(\sigma,x_{\sigma},\pm\xi_{\sigma}) \d\sigma \\
& \ge 1 - \|\Lag_{0,s}^\pm \tilde{\chi}^\pm(\sigma,x_{\sigma},\pm \xi_{\sigma})\|_{\Lone_\sigma([s_0-\alpha,s_0])} > 0.
\end{align*}
Therefore, $ \chi^\pm(0,\cdot) $ is elliptic at $ (x_0,\pm \xi_0) $. 

To estimate $ \Lag_{h,s}^\pm \chi^\pm $, we use
\begin{align*}
H_h(s,x,\xi) - H_0(s,x,\xi) 
& = H_h(s,x,\xi) - H_h(0,x,\xi)  \\
& = \int_0^s (\ps H_h)(\sigma,x,\xi) \d\sigma
= h^{1/2} \int_0^s (\pt \gamma^{(3/2)})(h^{1/2}\sigma,x,\xi) \d\sigma,
\end{align*}
and write
\begin{align*}
\Lag_{h,s}^\pm \chi^\pm(s,\cdot) - \Lag_{0,s}^\pm \chi^\pm(s,\cdot) 
& = \pm \{H_h-H_0,\chi^\pm\}(s,\cdot) \\
& =  \pm h^{1/2} \int_0^s \{ \pt \gamma^{(3/2)}(h^{1/2}\sigma,\cdot), \chi^\pm(s,\cdot) \} \d\sigma.
\end{align*}
Observe that
\begin{equation*}
\pt \gamma^{(3/2)}
= - \frac{3}{2} \Big(|\xi|^2 - \frac{(\nabla\eta \cdot \xi)^2}{1+|\nabla\eta|^2} \Big)^{-1/4} \Big( \frac{\nabla\eta \cdot \xi}{1+|\nabla\eta|^2} \nabla G(\eta)\psi \cdot \xi - \frac{(\nabla\eta \cdot \xi)^2}{(1+|\nabla \eta|^2)^2} \nabla G(\eta)\psi \cdot \nabla\eta\Big).
\end{equation*}
By hypothesis and Proposition~\ref{prop::D-N-higher-regularity}, $ \nabla G(\eta)\psi \in \HC{\mu-2,1/2}_{k} \subset \H{\mu-3}_2 $ as $ k \ge 2 $. Therefore,
\begin{equation*}
\pt \gamma^{(3/2)}(h^{1/2}\cdot,\cdot) \in \Linf([0,h^{-1/2}T],\Gamma^{3/2,\mu-K-\tilde{d}}_{-2,0}).
\end{equation*}
Using $ |x| \sim s $ on $ \supp \chi^\pm(s,\cdot) $, we have, uniformly for all $ s \in [0,h^{-1/2}T] $,
\begin{equation*}
\jp{s} \{ \pt \gamma^{(3/2)}(h^{1/2}\sigma,\cdot), \chi^\pm(s,\cdot) \}
\in \Linf_\sigma([0,h^{-1/2}T],\jp{x}^{-1}\Upsilon^{\mu-K-\tilde{d}-1}).
\end{equation*}
Therefore, 
\begin{align*}
\Lag_{h,s}^\pm \chi^\pm(s,\cdot) - \Lag_{0,s}^\pm \chi^\pm(s,\cdot) 
& = \pm h^{1/2} \jp{s}^{-1}  \int_0^s \O(1)_{\Linf([0,h^{-1/2}T],\jp{x}^{-1}\Upsilon^{\mu-K-\tilde{d}-1})} \d\sigma \\
& =  \pm h^{1/2} \jp{s}^{-1} \O(s)_{\jp{x}^{-1}\Upsilon^{\mu-K-\tilde{d}-1}} \\
& = \O(h^{1/2})_{\jp{x}^{-1}\Upsilon^{\mu-K-\tilde{d}-1}}
\end{align*}
which, together with~\eqref{eq::positivity-Lag-chi-0}, proves~\eqref{eq::chi-pm-quasi-positivity}.
\end{proof}

\subsubsection{Propagation}

Now we prove Theorem~\ref{thm::main-finite}. By Lemma~\ref{lem::WF-eta-psi==u} and Lemma~\ref{lem::Geo-Phi-relation}, it suffices to prove the following propagation theorem for $u$ defined as in~\eqref{eq::def-u}.

\begin{theorem}
\label{thm::main-finite-u}
Under the hypothesis of Theorem~\ref{thm::main-finite}, let~$ u $ be defined as~\eqref{eq::def-u}. Let 
\begin{equation*}
(x_0,\xi_0) \in \WF_{0,1}^\sigma(u_0)^\circ,
\end{equation*}
with $ 0 \le \sigma < \min\{(\mu-K-\tilde{d})/2,3k/2\} $ for some $ K > 0 $, such that the co-geodesic $ \{(x_s,\xi_s) = \Phi_s(x_0,\xi_0)\}_{s\in\R} $ is forwardly non-trapping. Set 
\begin{equation*}
\xi_\infty = \lim_{s\to+\infty} \xi_s,
\end{equation*}
then for all $ t_0 \in (0,T] $, we have
\begin{equation*}
\Big(\frac{3}{2}t_0|\xi_{\infty}|^{-1/2}\xi_{\infty},\xi_{\infty}\Big) \in \WF_{1/2,1}^\sigma(u(t_0)).
\end{equation*}
\end{theorem}

Under the semiclassical time variable $ s = h^{-1/2} t $, \eqref{eq::equation-w-sys} becomes
\begin{align*}
(\ps + h^{1/2}\P{V} \cdot \nabla + h^{1/2}\P{Q}) w 
+ i h^{1/2} \begin{pmatrix}
\P{\gamma} & 0 \\ 0 & -\P{\gamma}
\end{pmatrix} w +  \frac{ih^{1/2}g}{2} \P{\zeta} \begin{pmatrix*}
1 & -1 \\ 1 & -1
\end{pmatrix*} w \\
= F_h = \O(h^{1/2})_{\H{\mu-K-\tilde{d}}},
\end{align*}
for some $ K > 0 $.
We define $ \L_s^h $ which applies to time dependent operators $ \A : \swtz \to \swtz' $,
\begin{equation*}
\L^h_s \A = \ps \A + h^{1/2} \Big[ \P{V} \cdot \nabla + \P{Q} +  i \P{\gamma} \begin{pmatrix}
1 & 0 \\ 0 & -1
\end{pmatrix} + \frac{ig}{2} \P{\zeta} \begin{pmatrix*}
1 & -1 \\ 1 & -1
\end{pmatrix*} , \A \Big].
\end{equation*}
We also define $ \Lag_s^h $ which applies to symbols of the diagonal form $ A = \begin{pmatrix}
A^+ & 0 \\ 0 & A^-
\end{pmatrix} $,
\begin{equation*}
\Lag_s^h A = \begin{pmatrix}
\Lag^+_{h,s} A^+ & 0 \\ 0 & \Lag^-_{h,s} A^-
\end{pmatrix}.
\end{equation*}

\begin{proof}[Proof of Theorem~\ref{thm::main-finite-u}]
We shall from now on denote $ \rho = \mu - K - \tilde{d} $ for some sufficiently large $ K > 0 $, also denote $ I_h = [0,h^{-1/2}T] $ and
\begin{equation*}
Y_h^\rho = \Linf\Big( I_h,M_{2\times 2}\Big(\sum_{j=0}^\rho h^j \Upsilon^{\rho-j}\Big) \Big)
\end{equation*}
for simplicity. 
More precisely, a symbol $ A_h = \sum_{j=0}^\rho h^j A_h^j \in Y_h^\rho $ if
\begin{equation*}
\sup_{h\in(0,1]} \sup_{s\in[0,h^{-1/2}T]} N^{\rho-j}(A_h^j) <+\infty,
\end{equation*}
where the norm $ N^{\rho-j}(A_h^j) $ is applied to every component of $ A_h^j $.
Choose a strictly increasing sequence $ \{\lambda_j\}_{j \ge 0} \subset [1,1+\epsilon) $ with $ \epsilon > 0 $ being sufficiently small. Define $ \chi^{\pm}_j $ as in Lemma~\ref{lem::chi-symbol-construction-main} where we replace~$ \phi $ with $ \phi(\cdot/\lambda_j) $. Then
\begin{equation*}
\supp \chi^\pm_j \subset \{\chi^\pm_{j+1} > 0\}
\end{equation*}
for all $ j \in \N $.
Set
\begin{equation*}
\chi_j = \begin{pmatrix}
\chi^{+}_j & 0 \\ 0 & \chi^{-}_j
\end{pmatrix}.
\end{equation*}
We shall construct an operator $ \A_h \in \Linf(I_h,\Ltwo\to\Ltwo) $ such that:
\begin{enumerate}
\item \label{chi::support} $ \A_h $ is a paradifferential operator, more precisely, there exists 
\begin{equation*}
A_h^\pm \in \Holder{1}(\R_{\ge 0},\Upsilon^{\rho+1}) \cap \Holder{1}(\R_{\ge s_0},S^{-\infty}_0)
\end{equation*}
for some $ s_0 > 0 $, such that
\begin{equation*}
\A_h - \P{A_h}^h = \O(h^{\rho})_{\Linf(I_h,\Ltwo\to\Ltwo)}, \quad
A_h = \begin{pmatrix}
A_h^+ & 0 \\ 0 & A^-_h
\end{pmatrix}.
\end{equation*}
Moreover, we require that
\begin{equation*}
\supp A_h^\pm \subset \bigcup_{j\ge0} \supp \chi^\pm_j.
\end{equation*}

\item $ A_h^\pm(0,x,\xi) $ is elliptic at $ (x_0,\pm \xi_0) $;
\item \label{chi::elliptic-long-time} $ A_h^\pm\big(s,\frac{s}{t_0}x,\xi\big) \in S^{-\infty}_{-\infty} $ is elliptic at $ (\frac{3}{2}t_0|\xi_{\infty}|^{-1/2}\xi_{\infty},\xi_{\infty}) $ for $ s > 0 $ sufficiently large;
\item $ \L^h_s \A_h  \ge \O(h^{\rho})_{\Linf(I_h,\Ltwo\to\Ltwo)} $. \label{item:positivity-lagrange}
\end{enumerate}
We shall construct $ \A_h $ of the form 
\begin{equation*}
\A_h = \sum_{j\ge 0}^{2\rho} h^{j/2} \varphi^j \A_h^j,
\end{equation*}
where $ \varphi \in P_j $, recalling the definition~\eqref{eq::def-P_j}, and $ \A^j_h \in \Linf(I_h,\Ltwo\to\Ltwo) $. We begin by setting
\begin{equation*}
\A_h^0 = (\P{\chi_0}^h)^* \P{\chi_0}^h, \quad \varphi^0 \equiv 1.
\end{equation*}
Therefore, by the symbolic calculus, Lemma~\ref{lem::chi-symbol-construction-main} and Corollary~\ref{cor::homogenization-semiclassical} (observe that the symbol of $ \A^0_h $ belongs to $ \sigma_0 $, and that $ \gamma $ is a sum of homogeneous symbols),
\begin{equation*}
\ps \A_h^0 + h^{1/2} \Big[i \P{\gamma} \begin{pmatrix}
1 & 0 \\ 0 & -1
\end{pmatrix}, \A_h^0 \Big]
= 2 \P{\chi_0 \Lag^h_s \chi_0}^h + h \P{b_h^0}^h + \O(h^{\rho})_{\Linf(I_h,\Ltwo\to\Ltwo)},
\end{equation*}
for some symbol $ b_h^0 $ such that $ \jp{x} b_h^0 \in Y^\rho_h $.
This $ \jp{x} $ factor comes from the spatial decay of $ \partial_{x,\xi} \gamma $.
Moreover, we have $ \supp b_h^0 \subset \supp \chi_0 $, which implies $ \jp{s} b_h^0 \in Y^\rho_h $. Similarly,
\begin{equation*}
h^{1/2}[\P{V} \cdot \nabla,\A_h^0]
= h^{1/2} \P{b^1_h}^h + \O(h^{\rho})_{\Linf(I_h,\Ltwo\to\Ltwo)},
\end{equation*}
where $ \jp{s} b^1_h \in Y_h^\rho $, with $ \supp b^1_h \subset \supp \chi_0 $. Be careful that, because~$ Q $ and $ \P{\zeta} \begin{pmatrix*}
1 & -1 \\ 1 & -1
\end{pmatrix*} $ are not diagonal matrices, their commutators with~$ \A_h^0 $ do not gain an extra~$ h $, for the principal symbols do not cancel each other. So,
\begin{align*}
h^{1/2}[\P{Q},\A_h^0]
& = h^{1/2} \P{b^2_h}^h + \O(h^{\rho})_{\Linf(I_h,\Ltwo\to\Ltwo)}, \\\
h^{1/2}\Big[\P{\zeta} \begin{pmatrix*}
1 & -1 \\ 1 & -1
\end{pmatrix*},\A_h^0\Big]
& = h \P{b^3_h}^h + \O(h^{\rho})_{\Linf(I_h,\Ltwo\to\Ltwo)},
\end{align*}
where $ \jp{s} b^2_h, \jp{s} b^3_h \in Y^\rho_h $, with $ \supp b^2_h \cup \supp b^3_h \subset \supp \chi_0 $. By Lemma~\ref{lem::chi-symbol-construction-main},
\begin{equation*}
\chi_0 \Lag^h_s \chi_0 \ge h^{1/2} b_h^4,
\end{equation*}
where $ \jp{s} b_h^4 \in Y^\rho_h $, with $ \supp b_h^4 \subset \supp \chi_0 $. Therefore, combining the idea described above~\eqref{eq:from-garding-to-symbol-construction} and the paradifferential G{\aa}rding inequality (Lemma~\ref{lem::Garding-paradiff}, where we take $ \epsilon = 1/2 $),
\begin{equation*}
\P{\chi_0 \Lag^h_s \chi_0}^h - h^{1/2} \P{b_h^4}^4 \ge h^{1/2} \P{b_h^5}^h + \O(h^{\rho})_{\Ltwo\to\Ltwo},
\end{equation*}
for some $ b^5_h \in Y^\rho_h $ with $ \supp b^5_h \subset \{\chi_1 > 0\} $. 
In fact, choose $c_h \in L^\infty(\mathbb{R}_{\ge 0}, S^{-\infty}_0)$ such that \begin{equation*}
  \supp a_h \subset \{c_h = 1\} \subset \supp c_h \subset \supp \chi_1.
\end{equation*}
Then for all $v \in L^2$, we have
\begin{align*}
\bigl\langle v, \bigl(\P{\chi_0 \Lag^h_s \chi_0}^h - h^{1/2} \P{b_h^4}^4\bigr) v \bigr\rangle_{L^2}
& = \langle \P{c_h} v, \bigl(\P{\chi_0 \Lag^h_s \chi_0}^h - h^{1/2} \P{b_h^4}^4\bigr) \P{c_h} v \rangle_{L^2} + \O(h^\rho) \\
& \gtrsim -Ch^{1/2} \|\P{c_h} v\|_{L^2}^2 + \O(h^\rho).
\end{align*}
Therefore, it suffices to choose $b_h^5$ such that
\begin{equation*}
\P{b_h^5} - C \P{c_h}^* \P{c_h} = \O(h^\rho)_{L^\infty(I_h,L^2\to L^2)},
\end{equation*}
which can be achived by Proposition~\ref{prop::composition-P-h} and Proposition~\ref{prop::adjoint-P-h}.
Set
\begin{equation*}
\alpha_h^0 = \jp{s}(b_h^1 + b_h^2 + 2 b_h^4 + 2b_h^5) \in Y^\rho_h, \quad
\beta_h^0 = \jp{s}(b_h^0 + b_h^3) \in Y^\rho_h.
\end{equation*}
Then
\begin{equation*}
\L^h_s \A_h^0 \ge h^{1/2} \jp{s}^{-1} \P{\alpha_h^0 + h^{1/2}\beta_h^0}^h + \O(h^{\rho})_{\Linf(I_h,\Ltwo\to\Ltwo)}.
\end{equation*}
Suppose that we have found $ \A_h^j \in \Linf(I_h,\Ltwo\to\Ltwo) $, $ \varphi^j \in P_j $ for $ j = 0,\ldots, \ell-1 $, and $ \psi^{\ell-1} \in P_{\ell-1} $, $ \alpha_h^{\ell-1}, \beta_h^{\ell-1} \in Y^\rho_h $ with 
\begin{equation*}
\supp \alpha_h^{\ell-1} \cup \supp \beta_h^{\ell-1} \subset \{\chi_\ell > 0\},
\end{equation*}
such that
\begin{equation}
\label{eq::thm-main-finite-induction-hypothesis}
\L^h_s \Big( \sum_{j=0}^{\ell-1} h^{j/2} \varphi^j \A_h^j \Big)
\ge h^{\ell/2} \jp{s}^{-1} \psi^{\ell-1} \P{\alpha_h^{\ell-1} + h^{1/2} \beta_h^{\ell-1} }^h + \O(h^{\rho})_{\Linf(I_h,\Ltwo\to\Ltwo)}.
\end{equation}
Then as in the proof of~\eqref{thm::model-eq-finity} of Theorem~\ref{thm::model-eq}, we set 
\begin{equation*}
\varphi^\ell(s) = \int_0^s (1+\sigma)^{-1}\psi^{\ell-1}(\sigma) \d \sigma, \quad
\A_h^{\ell} = C_\ell \varphi^\ell \P{\chi_{\ell}}^{h},
\end{equation*}
where the constant~$ C_\ell $ is sufficiently large, such that by Lemma~\ref{lem::chi-symbol-construction-main}, in the sense of positivity of matrices,
\begin{align*}
C_\ell \Lag_h^s (\varphi^\ell \chi_\ell) 
& = C_\ell(1+s)^{-1}\psi^{\ell-1}\chi_\ell + C_\ell \varphi^\ell \Lag_h^s \chi_\ell \\
& \ge \jp{s}^{-1}\psi^{\ell-1} \alpha_h^{\ell-1} + \varphi^\ell h^{1/2} \jp{s}^{-1} \tilde{\beta}_h^\ell.
\end{align*}
for some $ \tilde{\beta}_h^\ell \in Y^\rho_h $. By the paradifferential G{\aa}rding inequality, and a routine construction of parametrix, we find $ \tilde{\alpha}_h^\ell \in Y^\rho_h $, with $ \supp \tilde{\alpha}_h^\ell \subset \{ \chi_{\ell+1} > 0 \} $, such that
\begin{equation*}
\P{C_\ell \Lag_h^s (\varphi^\ell \chi_\ell)}^h - \jp{s}^{-1} \P{ \psi^{\ell-1}  \alpha_h^{\ell-1} + h^{1/2} \varphi^\ell \tilde{\beta}_h^\ell }^h
\ge h \jp{s}^{-1} \P{(\psi^{\ell-1} + \varphi^\ell) \tilde{\alpha}_h^\ell}^h  + \O(h^{\rho})_{\Linf(I_h,\Ltwo\to\Ltwo)}.
\end{equation*}
Similarly as in the estimate of $ \A_h^0 $, by a symbolic calculus, we find $ \ul{\alpha}_h^\ell, \ul{\beta}_h^\ell \in Y^\rho_h $, with 
\begin{equation*}
\supp \ul{\alpha}_h^\ell \cup \supp \ul{\beta}_h^\ell \subset \supp \chi_{\ell},
\end{equation*}
such that
\begin{equation*}
\L^h_s \A_h^\ell = \P{C_\ell \Lag_h^s (\varphi^\ell \chi_\ell)}^h + h^{1/2} \jp{s}^{-1} \varphi^\ell \P{\ul{\alpha}_h^\ell + h^{1/2}\ul{\beta}_h^\ell} + \O(h^{\rho})_{\Linf(I_h,\Ltwo\to\Ltwo)}.
\end{equation*}
Summing up the two inequalities above,
\begin{align}
\label{eq::thm-main-finite-induction-step}
\L^h_s \A_h^\ell - \jp{s}^{-1} \psi^{\ell-1} \P{ \alpha_h^{\ell-1}}^h 
\ge h^{1/2} \jp{s}^{-1} & \P{\varphi^\ell(\ul{a}_h^\ell+\tilde{\beta}_h^\ell) + h^{1/2}(\psi^{\ell-1} + \varphi^\ell) \tilde{\alpha}_h^\ell + h^{1/2} \varphi^\ell \ul{\beta}_h^\ell }^h \\
& \quad + \O(h^{\rho})_{\Linf(I_h,\Ltwo\to\Ltwo)}. \nonumber
\end{align}
Therefore, combining~\eqref{eq::thm-main-finite-induction-hypothesis} and~\eqref{eq::thm-main-finite-induction-step},
\begin{equation*}
\L^h_s \Big( \sum_{j=0}^{\ell} h^{j/2} \varphi^j \A_h^j \Big)
\ge h^{(\ell+1)/2} \jp{s}^{-1} \psi^\ell \P{\alpha_h^{\ell} + h^{1/2} \beta_h^{\ell} }^h + \O(h^{\rho})_{\Linf(I_h,\Ltwo\to\Ltwo)},
\end{equation*}
where 
\begin{equation*}
\psi^\ell = 1 + \psi^{\ell-1} + \varphi^\ell, \quad
\alpha_h^\ell = \frac{\psi^{\ell-1}}{\psi^\ell} \beta_h^{\ell-1} + \frac{\phi^{\ell}}{\psi^\ell} (\ul{\alpha}_h^\ell + \tilde{\beta}_h^\ell), \quad
\beta_h^\ell = \frac{\psi^{\ell-1}+\varphi^\ell}{\psi^\ell} \tilde{\alpha}_h^\ell + \frac{\varphi^\ell}{\psi^\ell}\ul{\beta}_h^\ell.
\end{equation*}
Thus we close the induction procedure.

To finish the proof, suppose that 
\begin{align*}
\Big(\frac{3}{2}t_0|\xi_{\infty}|^{-1/2}\xi_{\infty},\xi_{\infty} \Big) & \not\in \WF_{1/2,1}^\sigma(u(t_0)), \\
\Big(\frac{3}{2}t_0|\xi_{\infty}|^{-1/2}\xi_{\infty},-\xi_{\infty}\Big) & \not\in \WF_{1/2,1}^\sigma(\overline{u(t_0)}).
\end{align*}
By Lemma~\ref{lem::chi-symbol-construction-main}, we can choose $ \phi $ such that for sufficiently small $ h>0 $,
\begin{align*}
\supp \theta_{1/h,*}^{1/2,0} \chi_j^+|_{s=h^{-1/2}t_0} \subset \R^{2d} \backslash \WF_{1/2,1}^\sigma(u(t_0)), \\
\supp \theta_{1/h,*}^{1/2,0} \chi_j^-|_{s=h^{-1/2}t_0} \subset \R^{2d} \backslash \WF_{1/2,1}^\sigma(\overline{u(t_0)}).
\end{align*}
So by Lemma~\ref{lem::un-paradiff-estimate-smooth-symbol} and Lemma~\ref{lem::support=decay}, 
\begin{equation*}
(\A_h w,w)_{\Ltwo}|_{s=h^{-1/2}t_0} = \O(h^{2\sigma}).
\end{equation*}
By our construction, $ \varphi^\ell(0) = 0 $ for all $ \ell \ge 1 $, so
\begin{equation*}
\A_h|_{s=0} = \A_h^0|_{s=0} = (\P{\chi_0}^h)^* \P{\chi_0}^h|_{s=0}.
\end{equation*}
Because $ F_h = \O(h^{1/2})_{\H{\rho}} $, we have, by Lemma~\ref{lem::basic-properties-WF}, that
$ \A_h F_h = \O(h^{\rho+1/2})_{\Ltwo}. $
Therefore, by~\eqref{item:positivity-lagrange},
\begin{align*}
\|\P{\chi_0}^h w|_{s=0}\|_{\Ltwo}^2 
= & \Re(\A_h w,w)_{\Ltwo}|_{s=h^{-1/2}t_0} 
- \int_0^{h^{-1/2}t_0} \Re(\L_s^h \A_h w,w)_{\Ltwo} \ds \\
& \qquad \qquad \qquad \qquad \qquad - \int_0^{h^{-1/2}t_0} \Re(\A_h F_h,w)_{\Ltwo} \ds \\
\le & \O(h^{2\sigma}) + \O(h^{\rho-1/2}) = \O(h^{2\sigma}).
\end{align*}
Observe that $ \chi_0|_{s=0} $ is of compact support with respect to~$ x $, we have
\begin{equation*}
\P{\chi_0|_{s=0}}^h = \T{\beta_h}^h + \O(h^{\rho})_{\Ltwo\to\Ltwo},
\end{equation*}
where 
\begin{equation*}
\beta_h = \sum_{j \ge 0} \psi_j \chi_0|_{s=0} \sharp_h \psi_j \in \sum_{j=0}^\rho h^j \Upsilon^{\rho-j}
\end{equation*}
is a finite summation.
By Lemma~\ref{lem::WF-characterization-paradiff} and~\eqref{eq::symbol-decay-complete}, we conclude that, if $ (x_0,\xi_0) \not \in \WF_{0,1}^\sigma(u_0) $ provided $ \sigma \le {3 \over 2} r $ where
\begin{equation*}
r = \min\{[2(\mu-1-\tilde{d})/3],k\}.
\end{equation*}
Under the hypothesis of theorem we have $ r = k $.
\end{proof}

\subsection{Proof of Corollary~\ref{cor::local-smoothing-effect}}

\label{sec::local-smoothing-effect}

The case when $ d = 1 $ is trivial. For the second case, we shall prove that on any co-geodesic $ \{(x_t,\xi_t)\}_{t\in\R} $, 
\begin{equation}
\label{eq::x-dot-xi-to-infinity}
\lim_{t\to+\infty} x_t \cdot \xi_t = \infty,
\end{equation}
so no geodesics can be trapped. The proof of~\eqref{eq::x-dot-xi-to-infinity} is almost finished by the proof of Lemma~\ref{lem::hamiltonian-flow-to-infinity}. Indeed, similar calculations imply that
\begin{equation*}
\frac{\d}{\dt} (x_t \cdot \xi_t) \gtrsim |\xi_0|^2.
\end{equation*}

\section*{List of notations}

\begin{itemize}[label={}]
  \item $\WF^\mu(u)$ wavefront set
  \item $\WF^\mu_{\delta,\rho}(u)$ quasi-homogeneous wavefront set
  \item $\op(a)$ pseudodifferential operator
  \item $\op_h(a)$ semiclassical pseudodifferential operator
  \item $\op_h^{\delta,\rho}(a)$ quasi-homogeneous semiclassical pseudodifferential operator
  \item $T_a$ paradifferential operator
  \item $\mathcal{P}_a$ dyadic paradifferential operator
  \item $\mathcal{P}_a^h$ semiclassical dyadic paradifferential operator
  \item $\mathcal{P}_a^{h,\epsilon}$ quasi-homogeneous semiclassical dyadic paradifferential operator
  \item $a \sharp_h^{\delta,\rho} b$ composition of symbols
  \item $\zeta_h^{\delta,\rho} a$ adjoint of symbols
  \item $\mathscr{S}$, $\mathscr{S}'$ Schwartz function space and tempered distribution space
  \item $\mathcal{H}^{\mu,\delta}_k$, $W^{r,\infty}_{k,\delta}$ weighted Sobolev spaces
  \item $\Gamma^{m,r}$ paradifferential symbol class
  \item $\Gamma^{m,r}_{k,\delta}$ weighted paradifferential symbol class
  \item $\Sigma^{m,r}_{k,\delta}$ weighted paradifferential poly-symbol class
  \item $M^{m,r}$, $M^{m,r}_{k,\delta}$ symbol norm and weighted symbol norm
  \item $\theta^{\delta,\rho}_h$ phase-space scaling operator
\end{itemize}

\bibliography{wws} 
\bibliographystyle{plain}

\end{document}